\documentclass[oneside,english]{amsart}
\usepackage[T1]{fontenc}
\usepackage[utf8]{inputenc}
\usepackage{mathrsfs}
\usepackage{amsbsy}
\usepackage{amstext}
\usepackage{amsthm}
\usepackage{amssymb}
\usepackage{wasysym}
\usepackage{esint}
\usepackage{calc}
\usepackage{hyperref}
\usepackage[colorinlistoftodos]{todonotes}
\usepackage{footmisc}
\usepackage{colonequals}
\makeatletter
\usepackage[top=1.3in,bottom=1.4in,left=1.0in,right=1.0in]{geometry}

\usepackage{tikz}
\def \no{\nonumber}

\newtheorem{remark}{Remark}[section]
\newtheorem{thm}{Theorem}[section]
\newtheorem{lem}[thm]{Lemma}
\newtheorem{proposition}[thm]{Proposition}

\newtheorem{cor}{Corollary}[section]
\newtheorem{defn}{Definition}[section]
\numberwithin{equation} {section}

\usepackage{babel}

\makeatother

\allowdisplaybreaks

\begin{document}
\title{A new boundary mass for asymptotically flat half-manifolds}

\author{Guofang Wang}
\address{Albert-Ludwigs-Universit\"at, Mathematisches Institut, Ernst-Zermelo-Str.~1, 79104 Freiburg, Germany}
\email{guofang.wang@math.uni-freiburg.de}

\author{Wei Wei}
\address{School of Mathematics, Nanjing University, Nanjing 210093, P.R. China}
\email{wei\_wei@nju.edu.cn}
%\thanks{Wei Wei is partially supported by NSFC No. 12201288, BK20220755
%and the Alexander von Humboldt Foundation. }

\begin{abstract}
We introduce a boundary analogue of the Gauss--Bonnet--Chern mass for asymptotically flat half-manifolds with non-compact boundary. We prove that this mass is well defined and establish the corresponding positive mass theorems for graphical and conformally flat graphs. Also we provide a Penrose-type inequality for the mass $\mathfrak{m}_{a,B}(g)$ and identify its Euler--Lagrange equations.
\end{abstract}

\maketitle

\section{Introduction}

A complete manifold $(M^{n},g)$ is called asymptotically flat (AF) of order $\tau$ (with one end) if there exist a compact set $K\subset M$ and $R>0$ such that $M\setminus K$ is diffeomorphic to $\mathbb{R}^{n}\setminus\overline{B_{R}(0)}$ and, in the induced Euclidean coordinates, the metric satisfies
\[
g_{ij}=\delta_{ij}+\sigma_{ij}
\]
with
\begin{equation}
\left|\sigma_{ij}\right|+r\left|\partial\sigma_{ij}\right|+r^{2}\left|\partial^{2}\sigma_{ij}\right|=O\left(r^{-\tau}\right),\label{eq:asymptotic behavior}
\end{equation}
where $r=|x|$ and $\partial$ denotes the standard derivative operator on $\mathbb{R}^{n}$.

The ADM mass is defined by
\[
\mathfrak{m}_{ADM}=c(n)\lim_{r\rightarrow+\infty} \int_{S_{r}^{n-1}}\left(g_{ij,j}-g_{jj,i}\right)\nu^{i}d\sigma,
\]
where $\nu$ is the outward unit normal to the coordinate sphere $S^{n-1}_r$. The integrand is tied to the scalar curvature; more precisely,
\[
R(x)=\partial_i\bigl(g_{ij,j}-g_{jj,i}\bigr)+O\bigl(r^{-2\tau-2}\bigr),\qquad r=|x|\ \text{large}.
\]
Introduced by Arnowitt--Deser--Misner \cite{ADM}, the ADM mass is a fundamental invariant in differential geometry and general relativity. Bartnik showed that $\mathfrak{m}_{ADM}$ is well defined and coordinate independent provided $R\in L^{1}(M)$ and the decay is sufficiently fast, namely $\tau>\frac{n-2}{2}$.

The positive mass theorem (PMT) asserts that an AF manifold with appropriate decay and nonnegative scalar curvature has nonnegative ADM mass. Schoen--Yau \cite{SY} proved the PMT for $n\le 7$ via minimal surface methods (the borderline case $n=8$ involves possible isolated singularities of minimizing hypersurfaces; see Li--Wang \cite{LiWang2022}), and Witten \cite{Witten} proved it for spin manifolds using a Bochner argument. Recently, in dimensions $9$ and $10$, the Schoen--Yau approach was extended using the generic regularity theory of Chodosh--Mantoulidis--Schulze \cite{CMS2023}; later, Chodosh--Mantoulidis--Schulze--Wang \cite{CMSW2025} pushed this further up to dimension $11$. Very recently, Bi--Hao--He--Shi--Zhu \cite{bi-etal2026} proved it up to dimension $19$.
Finally, Brendle--Wang \cite{brendleWang2026} established the PMT in all dimensions;

Beyond its origins in physics, the PMT has led to a rich array of techniques and applications (e.g. to the Yamabe problem). For three-dimensional manifolds, several approaches are now available; see \cite{AMO,BKKS,Miao,Ste,LiY}. Many related invariants have also been introduced to capture asymptotic geometry in other settings (for instance, asymptotically locally Euclidean geometry \cite{HL}) and to develop new applications; see \cite{chenshi1,chenshi2} and the survey of Herzlich \cite{Herz}. 

Most classical invariants are built from first-order expressions in the metric and are closely related to scalar curvature. A natural question is whether one can define meaningful analogues under weaker (``slow'') decay assumptions. For AF manifolds with slow decay, Ge--Wang--Wu \cite{GWW1} (and independently Li--Nguyen \cite{LN}) introduced the Gauss--Bonnet--Chern (GBC) mass
\[
\mathfrak{m}_{GBC}(g):=c(n,a)\lim_{R\rightarrow\infty}\int_{S_{R}}g_{jk,l}P_{a}^{ijkl}\nu_{i}d\sigma,
\]
where $P$ is a $4$-tensor defined in \eqref{eq:definition of P} and $c(n,a)>0$ depends only on $n$ and $a$. As in the ADM case, the corresponding flux is governed by a curvature quantity, namely the Gauss--Bonnet--Chern curvature $L_a$, which satisfies
\[
L_a=2\partial_i\bigl(g_{jk,l}P_{a}^{ijkl}\bigr)+O\bigl(r^{-(a+1)\tau-2a}\bigr).
\]
Ge--Wang--Wu \cite{GWW1} proved that $\mathfrak{m}_{GBC}$ is well defined and coordinate independent when $L_a\in L^{1}(M)$ and $\tau>\tau_a:=\frac{n-2a}{a+1}$ (note that $\tau_a$ decreases as $a$ increases). Alternative formulations were obtained in \cite{Wang-Wu}, and related slow-decay masses in the asymptotically hyperbolic setting were developed in \cite{GWW2}. Positive mass theorems for $\mathfrak{m}_{GBC}$ were established for AF graphs and conformally flat manifolds \cite{GWW1,GWW2}. Further developments include a GBC center of mass \cite{Herz1} and the constancy of the GBC mass along the Ricci flow \cite{Ho}. Other slow-decay masses have also appeared, for instance those related to the GJMS operator and $Q$-curvature \cite{Michel1,Michel2,ALL}.

A complete manifold $(M^{n},g)$ is called an asymptotically flat half-manifold of order $\tau$ (with one end) if there exist a compact set $K\subset M$ and $R>0$ such that $M\setminus K$ is diffeomorphic to $\overline{\mathbb{R}^{n}_{+}}\setminus\overline{B^{+}_{R}(0)}$ and, in the induced coordinates on $\overline{\mathbb{R}^{n}_{+}}$, the metric satisfies
\[
g_{ij}=\delta_{ij}+\sigma_{ij}
\]
with
\begin{equation}
\left|\sigma_{ij}\right|+r\left|\partial\sigma_{ij}\right|+r^{2}\left|\partial^{2}\sigma_{ij}\right|=O\left(r^{-\tau}\right),\label{eq:asymptotic behavior-half}
\end{equation}
where $\mathbb{R}_{+}^{n}=\{(x^{1},\cdots,x^{n})\mid x^{n}>0\}$, $B_{R}^{+}(0)=\{x\in\mathbb{R}^{n}\mid |x|<R,\ x^{n}>0\}$, and $r=|x|$.

%Ge, Wang, and Wu \cite{GWW1} introduced a mass for asymptotically
%flat manifolds by using the Gauss-Bonnet-Chern curvature, and showed
%a positive mass for asymptotically flat graphs and established some connections with various geometric integrals.

In this paper we introduce a boundary analogue of the Gauss--Bonnet--Chern mass for asymptotically flat half-manifolds with non-compact boundary. The construction is based on the Gauss--Bonnet--Chern (Lovelock) curvature in the interior together with a natural boundary curvature term, defined as follows:
\begin{equation}
L_{a}=\frac{1}{2^{a}}\sum_{i,j=1}^{n}\delta_{j_{1}\cdots j_{2a}}^{i_{1}\cdots i_{2a}}R_{i_{1}i_{2}}{}^{j_{1}j_{2}}\cdots R_{i_{2a-1}i_{2a}}{}^{j_{2a-1}j_{2a}},\label{eq:definition of La}
\end{equation}

\begin{equation}
\mathscr{B}^{a-1}=\frac{1}{2^{a-1}}\sum_{i,j=1}^{n-1}\delta_{j_{1}\cdots j_{2a-1}}^{i_{1}\cdots i_{2a-1}}R_{i_{1}i_{2}}^{\quad j_{1}j_{2}}\cdots R_{i_{2a-3}i_{2a-2}}{}^{j_{2a-3}j_{2a-2}}L_{i_{2a-1}}^{j_{2a-1}},\label{eq:definition of Ba-1}
\end{equation}
where $R_{ij}{}^{kl}$ is the Riemann curvature tensor of $g$ and $\nabla$ is the Levi--Civita connection. Along $\partial M$, we write
$$L_{\beta}^{\alpha}=g^{\alpha\gamma}\langle\nabla_{\partial_{\gamma}}\nu,\partial_{\beta}\rangle$$
for the second fundamental form of $\partial M\subset M$, where $\nu$ is the outward unit normal. Here $\sum_{i,j=1}^{n}$ (resp. $\sum_{i,j=1}^{n-1}$) denotes summation over the corresponding multi-indices $i_{1},\dots,i_{2a}$, $j_{1},\dots,j_{2a}$ (resp. $i_{1},\dots,i_{2a-1}$, $j_{1},\dots,j_{2a-1}$).

Throughout the paper we assume $a<\frac{n}{2}$ unless stated otherwise.

Before defining the mass, we introduce the following boundary notation:
\begin{equation}
P_{a}^{stlm}=\frac{1}{2^{a}}\sum_{i,j=1}^{n}\delta_{j_{1}\cdots j_{2a-2}j_{2a-1}j_{2a}}^{i_{1}\cdots i_{2a-2}st}R_{i_{1}i_{2}}{}^{j_{1}j_{2}}\cdots R_{i_{2a-3}i_{2a-2}}{}^{j_{2a-3}j_{2a-2}}g^{j_{2a-1}l}g^{j_{2a}m},\label{eq:definition of P}
\end{equation}
and for $a\ge 2$
\begin{align*}
  Q_{a,a-2}^{\eta\beta\alpha\gamma}
=  \frac{1}{2^{a-2}}\sum_{i,j=1}^{n-1}\delta_{jj_{1}\cdots j_{2a-4}j_{2a-3}j_{2a-2}}^{ii_{1}\cdots i_{2a-4}\eta\beta}L_{i}^{j}R_{i_{1}i_{2}}{}^{j_{1}j_{2}}\cdots R_{i_{2a-5}i_{2a-4}}{}^{j_{2a-5}j_{2a-4}}g^{j_{2a-3}\alpha}g^{j_{2a-2}\gamma}.
\end{align*}

\begin{defn}\label{mass defi}
Let $n\ge 2a+1$. Assume that $(M^{n},g)$ is an asymptotically flat half-manifold
with non-compact boundary $\partial M$ of decay order $\tau>\frac{n-2a}{a+1}$
and $\int_{M}L_{a}dv_{g}+2a\int_{\partial M}\mathscr{B}^{a-1}d\sigma_{g}$
is finite. We define the boundary Gauss-Bonnet-Chern mass by
\begin{align}
\mathfrak{m}_{a,B}(g) & :=\lim_{R\rightarrow\infty}c(n,a)\bigg(\int_{S_{R}^{+}}g_{jk,l}P_{a}^{ijkl}\nu_{i}d\sigma+\int_{S_{R}^{n-2}}g_{n\beta}P_{a}^{n\beta n\alpha}\mu_{\alpha}dl\label{eq:definition of boundary mass}\\
 & \quad\quad\quad\quad\quad\quad\quad\,+(a-1)\int_{S_{R}^{n-2}}Q_{a,a-2}^{\eta\beta\alpha\gamma}g_{\beta\alpha,\gamma}\mu_{\eta}dl\bigg),\nonumber 
\end{align}
where $S_{R}^{+}$ is a large coordinate hemisphere of radius $R$
with outward unit normal $\nu$ in $\mathbb{R}^{n}$ and $d\sigma$
is the area element of $S_{R}^{+}$ in $\mathbb{R}^{n}$; $\mu$ is
the outward pointing unit co-normal to $S_{R}^{n-2}=\partial S_{R}^{+}$,
oriented as the boundary of the bounded region $\Sigma_{R}\subset\partial M$,
and $dl$ is the area element of $S_{R}^{n-2}$ in $\mathbb{R}^{n-1}$;
along $\partial M$, $\{\partial_{\alpha}\}_{\alpha}$ span $T\partial M$
while $\partial_{n}$ points inward on $\Sigma$; $g_{ij,k}=\partial_{k}g_{ij}$
and $g_{\beta\alpha,\gamma}=\partial_{\gamma}g_{\beta\alpha}$ are
the ordinary partial derivatives; $c(n,a)=\frac{(n-2a)!}{2^{a}(n-1)!\omega_{n-1}}$
and $\omega_{n-1}$ is the area of unit sphere $\mathbb{S}^{n-1}$.
%The definition of $P_a^{ijkl}$ will be introduced in Section
%2, see (\ref{eq:definition of P}). 
\end{defn}

When $a=1$, we convent that $(a-1)Q_{a,a-2}^{\eta\beta\alpha\gamma}=0$. 
Up to a constant multiple, the definition coincides with the following mass  related
to scalar curvature in \cite{ABL}, 
\[
\mathfrak{m}_{ABL}=\lim_{r\rightarrow+\infty}\left\{ \int_{S_{r}^{+}}\left(g_{ij,j}-g_{jj,i}\right)\nu^{i}d\mathcal{S}_{r,+}^{n-1}+\int_{S_{r}^{n-2}}g_{\alpha n}\mu^{\alpha}d\mathcal{S}_{r}^{n-2}\right\},
\]
which was introduced by
Almarz, Barbosa, and De Lima \cite{ABL}. 

Now we state our main theorem as follows.
\begin{thm}
\label{thm:Main thm1-}Suppose that $(M^{n},g)$ is an asymptotically
flat half-manifold of decay order $\tau>\frac{n-2a}{a+1}$ and assume that $L_{a}\in L^1(M)$ and  $\mathscr{B}^{a-1}\in L^1(\partial M)$, then the mass $\mathfrak{m}_{a,B}(g)$  is well-defined
and depends only on $g.$
\end{thm}

%. The mass have defined a mass for asymptotically flat half-manifolds $M$ with non-compact boundary $\partial M$ with respect to scalar curvature as follows:

\iffalse
where $S_{r,+}^{n-1}$ is a large coordinate hemisphere of radius
$r$ with outward unit normal $\nu$, and $\eta$ is the outward pointing
unit co-normal to $S_{r}^{n-2}=\partial S_{r,+}^{n-1}$, oriented
as the boundary of the bounded region $\Sigma_{r}\subset\partial M$;
along $\partial M$, $\{\partial_{\alpha}\}_{1\le\alpha\le n-1}$
span $T\partial M$ while $\partial_{n}$ points inward in $M$; $g_{ij,k}=\partial_{k}g_{ij}$
is the ordinary partial derivatives. $\mathfrak{m}_{(M,g)}$ is an
invariant of the asymptotic geometry of $M.$ 
\fi

\iffalse
For scalar curvature case on manifolds without non-compact boundary, Bartnik \cite{Bartnik}
showed that the ADM mass is an invariant, independent of the choice
of asymptotically flat coordinates. Ge, Wang, and Wu \cite{GWW1} proved that on manifolds without non-compact boundary, the mass related to Lovelock curvature is an invariant of asymptotically flat coordinates. 
\fi

Boundary problems in geometric analysis have also received considerable attention in recent years, and boundary versions of mass have begun to play roles analogous to the classical ADM mass. Motivated by the Yamabe problem with prescribed boundary mean curvature, Almaraz--Barbosa--de Lima \cite{ABL} proved a positive mass theorem for the ABL mass when $3\le n\le 7$, and also for $n\ge 3$ under the spin assumption.

Almaraz--de Lima--Mari \cite{ALM} studied the ABL mass for initial data sets with non-compact boundary in a spacetime setting; see also the survey \cite{Lima} by de Lima. A major development is the Riemannian Penrose inequality for the ABL mass, proved by K\"orber \cite{Kor} and by Eichmair--K\"orber \cite{EK1}; the graphical case with free boundary was treated by Barbosa--Meira \cite{BM}. K\"orber's proof uses a weak free-boundary inverse mean curvature flow, building on earlier work of Marquardt \cite{Mar3}. This Penrose inequality is closely related to Huisken's conjecture for asymptotically flat support surfaces, studied by Volkmann \cite{Volk} and resolved by Eichmair--K\"orber \cite{EK}.

%the development of theory about the free boundary anisotropic minimal hypersurface or minimal hypersurface with constant contact angle is also closely connected to the ABL mass, which also links the exterior mass defined by \cite{Volkman} when the non-compact boundary of manifolds is an asymptotically flat support surface.  Naturally, the Penrose-type inequality related to the ABL mass and exterior mass was discussed by mathematicians.  With the help of inverse mean curvature flow with constant contact angle, Eichmair and K\"orber confirmed a Penrose-type conjecture of Huisken about the exterior mass, see \cite{EK}.

%The exterior mass Assuming that there exists a free boundary minimal surface horizon, K\"orber \cite{Kor} showed a Penrose inequality for 3-dimensional asymptotically flat spaces with non-compact boundary, which states that $m_{ABL}$ is bounded below in terms of the area of the free boundary minimal surface.  Lima, Marquedut paper...

There is also a substantial literature on special classes of AF manifolds (with compact boundary, or with empty boundary), notably AF graphs and conformally flat manifolds. Lam \cite{Lam} proved the positive mass theorem and the Penrose inequality for AF graphs; see \cite{HW1,HW2,LG1,LG2} for further generalizations. In the conformally flat setting, Penrose-type inequalities were obtained by Freire--Schwartz \cite{FS} and by Jauregui \cite{Jeff}. Schwartz \cite{Schwartz} related the mass to capacity and established a volumetric Penrose inequality for conformally flat manifolds.

In this work we investigate the boundary GBC mass $\mathfrak{m}_{a,B}$ in two settings where the relevant curvature quantities admit particularly transparent formulas: graphs of asymptotically flat functions and conformally flat manifolds. In both cases, the Gauss--Bonnet--Chern curvature can be expressed in terms of $\sigma_k$-type operators that are central in fully nonlinear PDE.

We begin with the graphical setting.
\begin{defn}
\label{def: assumption of f-1}Let $f:\mathbb{R}_{+}^{n}\rightarrow\mathbb{R}$
be a smooth function and let $f_{i},f_{ij}$ and $f_{ijk}$ denote
the first, the second and the third derivatives of $f$ respectively.
$f$ is called an asymptotically flat function of order $\tau$ if
\[
f_{i}(x)=O\left(|x|^{-\tau/2}\right),\quad|x|\left|f_{ij}(x)\right|+|x|^{2}\left|f_{ijk}(x)\right|=O\left(|x|^{-\tau/2}\right)
\]
 at infinity for some $\tau>(n-2a)/(a+1)$.
\end{defn}
Let $\left(\mathbb{R}_{+}^{n}\backslash\Omega,g\right)=\left(\mathbb{R}_{+}^{n}\backslash\Omega,g_{\mathbb{E}}+df\otimes df\right)$ ($\Omega$ possibly empty)
be the graph of a smooth asymptotically flat function $f:\mathbb{R}_{+}^{n}\rightarrow\mathbb{R}$
in Definition \ref{def: assumption of f-1} and $Df=(f_{1},\cdots,f_{n}).$

%We denote $\kappa_f=(\kappa_1,\cdots, \kappa_{n})$ as the principal curvatures of the graph $f$ and 

Define 
$$\sigma_k(\lambda)=\sum_{1\le i_1<\cdots<i_k\le n}\lambda_{i_1}\cdots \lambda_{i_k},$$
where $\lambda=(\lambda_1,\cdots, \lambda_n)\in \mathbb{R}^n.$ 
 
Let $\sigma_k(B):=\sigma_k(\lambda(B))$, where $\lambda(B):=(\lambda_1,\cdots, \lambda_n)$ are the eigenvalues of  matrix $B_{n\times n}$.
The $\sigma_k$ operator always connects with 
 the  positive-cone $\Gamma_a^+$ as follows:
 $$\Gamma_a^+=\{\lambda=(\lambda_1,\cdots, \lambda_n)\in \mathbb{R}^n|\sigma_1(\lambda)>0,\cdots, \sigma_a(\lambda)>0\}.$$

From Reilly's classical work \cite{Reilly}, we know if $k$ is even, then \[
\sigma_{k}(S)=\frac{1}{k!}L_{k/2},
\]
where $S_{ij}$ is the second fundamental form of hypersurface $M=(x,f(x))\subset\mathbb{R}^{n+1}$, 
and  now we state the formula of the boundary mass under this graphical setting.

For the reader's convenience, we  denote the second fundamental form of $\partial \Omega$ as follows:
 Let $\nu$ be the unit outer normal vector of $\partial \Omega$ pointing inward $\mathbb{R}^n\backslash\Omega$.
And  $h_{\alpha \beta}$  is the second fundamental form of $\partial \Omega$ in $\mathbb{R}^n$, that is $h_{\alpha \beta}=\langle D_{e_\alpha}e_\beta, -\nu\rangle$ and $e_\alpha,e_\beta$ are unit tangential vectors on $\partial \Omega$. The principal curvatures are denoted by $\kappa_{\partial \Omega}$ and $\sigma_1({\kappa_{\partial \Omega}})=\sigma_1(h_{\alpha \beta})$. We use similar rules of notations for the curvature of  other domains. 

\begin{thm}
\label{thm:Main thm 2}Let $f:\mathbb{R}_{+}^{n}\backslash\Omega\rightarrow\mathbb{R}$
be an asymptotically flat function of order $\tau$ on $\mathbb{R}_{+}^{n}\backslash\Omega$
of class $C^{2}$ up to boundary, with the assumption in Definition
\ref{def: assumption of f-1}. Let $\left(\mathbb{R}_{+}^{n}\backslash\Omega,g\right)$
be the graph of $f$. We assume that $f$ is constant on $\partial\Omega\backslash\partial\mathbb{R}_{+}^{n}$
and $|Df|\rightarrow\infty$ on $\partial\Omega\backslash\partial\mathbb{R}_{+}^{n}$ with the unit outer normal vector $\nu=\frac{Df}{|Df|}$  pointing outward to $\Omega$.
Take $\Omega'=\overline{\Omega}\cap\partial\overline{\mathbb{R}_{+}^{n}}.$
We have 
\begin{align*}
\frac{2\mathfrak{m}_{a,B}(g)}{(2a)!c(n,a)}=  & \int_{\mathbb{R}_{+}^{n}\backslash\Omega}\frac{1}{(2a)!}L_{a}dx+\frac{1}{2a}\int_{\partial\Omega\cap\mathbb{R}_{+}^{n}}\sigma_{2a-1}(\kappa_{\partial\Omega})d\sigma\\
&+\int_{\partial\mathbb{R}_{+}^{n}\backslash\Omega'}\sigma_{2a-1}({S}^{{T}})\frac{f_{n}}{\sqrt{1+|Df|^2}}d\sigma +\frac{1}{2a}\int_{\partial\Omega'}\sigma_{2a-2}(\kappa_{\partial\Omega'})(\sin\theta)^{2a-1}\cos\theta dl,
\end{align*}
where $\theta$ is the angle formed by the tangent plane of $\partial\Omega$
at $\partial\Omega'$ and $x_{n}=0$, $\kappa_{\partial\Omega}$ and
$\kappa_{\partial\Omega'}$ are the principal curvatures of $\partial\Omega$
and $\partial\Omega'$ in $\mathbb{R}_{+}^{n}$ and $\partial\mathbb{R}_{+}^{n}$
respectively, and ${S}^{{T}}$ is the tangential part of the
second fundamental form $S$ of hypersurface in $\mathbb{R}^{n+1}$ on $\partial\mathbb{R}_{+}^{n}.$
\end{thm}

\iffalse
\textcolor{red}{
If we assume addition that the eigenvalues of the shaper operator of graphs $A$ satisfy $\lambda(A)\in \bar \Gamma_{2a}^+$ in $\mathbb{R}^n_+$, then we have $\sigma_{2a-1}(A^{\mathbb{T}})\ge 0$ and $L_a=(2a)!\sigma_{2a}(A)\ge 0.$ Now, if $\mathfrak{m}_{a,B}(g)=0$ and assume also that $f_n\ge 0$ on $\partial \mathbb{R}^n_+$, then $\sigma_{2a}=0$ in $\mathbb{R}^n_+$ and $\sigma_{2a-1}(A^{\mathbb{T}}){f_{n}}=0$ on $\partial \mathbb{R}^n_+$. It seems that the following corollary is not obvious?
\begin{cor}
Assume as Theorem \ref{thm:Main thm 2}.
and the eigenvalues of the shaper operator of graphs $A$ satisfy $\lambda(A)\in \bar\Gamma_{2a}^+$ in $\mathbb{R}^n_+$ and $f_n\ge 0$ on $\partial \mathbb{R}^n_+$. Then $\mathfrak{m}_{a,B}(g)\ge 0$ and equality holds iff the graph is flat.???? 
\end{cor}
}

\fi
When $S\in \Gamma_{2a}^+$,  it follows that $\kappa_{\partial \Omega}\in \Gamma_{2a-1}^+$ and $\kappa_{\partial \Omega'}\in \Gamma_{2a-2}^+$ if $\Omega$ is not empty.
\iffalse
\textcolor{red}{There exist some problem here. In the proof, I choose the outer normal vector on level set $\partial $ as $\nu=\frac{Df}{|Df|}$. But it is possible to be $\nu=-\frac{Df}{|Df|}$. If no $\sigma_{2a-2}(\kappa_{\partial \Omega'})$, then of course, both direction yields the same identity. But with $\sigma_{2a-2}(\kappa_{\partial \Omega'})$, this terms will differ a symbol $-$. Maybe the appearance of this phenomenon always tells us that we should pose $f_n=0$, or add other terms to cancel it, or add some condition such that $\nu=\frac{Df}{|Df|}$ or $\nu=\frac{Df}{|Df|}$. }
\fi
When $f$ is constant on $\partial\Omega$, we have $f_{n}=|Df|\cos\theta$ on $\partial\Omega'$, and if $f_n=0$ in Theorem \ref{thm:Main thm 2}, $\theta=\frac\pi 2$. 

With some additional assumptions, the following corollary holds. 
\begin{cor}\label{cor:sharp lower bound for graph}
Assume as Theorem \ref{thm:Main thm 2}. Suppose $\lambda(S)\in \bar \Gamma_{2a}^+$
and $f_{n}\ge 0$ on $\partial\mathbb{R}_{+}^{n}\backslash\Omega.$ We
have 

\begin{align*}
\frac{2\mathfrak{m}_{a,B}(g)}{(2a)!c(n,a)}\ge & \frac{1}{2a}\int_{\partial\Omega\cap\mathbb{R}_{+}^{n}}\sigma_{2a-1}(\kappa_{\partial\Omega})d\sigma+\frac{1}{2a}\int_{\partial\Omega'}\sigma_{2a-2}(\kappa_{\partial\Omega'})(\sin\theta)^{2a-1}\cos\theta dl,
\end{align*}
where $\kappa_{\partial\Omega}$ denotes the principal curvature of $\partial\Omega$.
%$\kappa_{\partial\Omega'}$ are the principle curvature of $\partial\Omega$
%where $\theta$ is the angle formed by the tangent plane of $\partial\Omega$
%at $\partial\Omega'$ and $x_{n}=0$, $\kappa_{\partial\Omega}$ and
%$\kappa_{\partial\Omega'}$ are the principle curvature of $\partial\Omega$
%and $\partial\Omega',$ $A^{\mathbb{T}}$ is the tangential part of
%$A$ on $\partial\mathbb{R}_{+}^{n}.$ 
The equality is achieved by half-Schwarz manifolds.
%Moreover, if we assume that $\lambda(S)\in \bar \Gamma_{2a}^+$ and $\Omega$ is empty, then only the flat graph can achieve the equality, that is $\mathfrak{m}_{a,B}(g_{\mathbb{E}})=0$.
\end{cor}

This corollary is due to the well-known fact that if $\lambda(S)\in \bar \Gamma_{2a}^+$, then $\sigma_{2a-1}(S^T)\ge0$.
One can also replace the assumptions in Corollary \ref{cor:sharp lower bound for graph} by $L_a\ge 0$ and $f_n=0$.

\iffalse
We remark here that if we assume additionally $f_n=0$ on $\partial \mathbb R^n_+$ in Corollary \ref{cor:sharp lower bound for graph}, then the equality holds if and only if $f$ is constant. Equality yields that the graph $(x, f(x))$ satisfies $k$-th mean curvature $\sigma_k(A)=0$ in $ \mathbb R^n_+$ with free boundary condition $f_n=0$ on $\partial \mathbb R^n_+$. 
When $k=1$, the equation is the minimal surface equation, and the flatness can be just obtained by reflecting the graph and the regularity in whole space can be naturally improved. For $k\ge 2$, the $k$-th mean curvature equation is fully nonlinear and it seems unknown whether reflecting the graph can reduce the half-Liouville problem to Liouville theorem in entire space due to the possible regularity problem from fully nonlinear equations.   To show that $f$ is constant, one may carry out a similar process in \cite{WangX} with minor modification of the auxiliary functions.  We left it to interested readers.
\fi

We now describe the analogous phenomenon for conformally flat manifolds. 
\begin{defn}
A conformally flat (CF) manifold $(M, g)=(\mathbb{R}^n_+\backslash \Omega, e^{-2u}g_{\mathbb{E}})$ ($\Omega$ possibly empty) is called an
asymptotically flat  conformally flat (AF CF) manifold of decay order $\tau$ if  $$|u|+|x||D u|+|x|^2\left|D^2 u\right|=O\left(|x|^{-\tau}\right)$$
at infinity for some $\tau>(n-2a)/(a+1)$.
\end{defn}

The Schouten tensor $$A_g:=\frac{1}{n-2}\left(\operatorname{Ric}-\frac{R_g}{2(n-1)} g\right)$$
and we say $g\in \Gamma_a^+$ if  $\lambda(A_g)\in \Gamma_a^+$.

Particularly, when $g=e^{-2u}g_{\mathbb{E}}$, $$A_g=D^2 u+D u \otimes D u- \frac{|D u|^2}{2}g_{\mathbb{E}}.$$
It is well known that on locally conformally flat manifolds, \[
L_{a}=2^{a}a!\frac{(n-a)!}{(n-2a)!}\sigma_{a}(A_{g}).
\]
This inspires us to prove the following theorem:
\begin{thm}\label{thm1 on CF}
Let $(M, g)=(\mathbb{R}^n_+, e^{-2u}g_{\mathbb{E}})$ be AF CF manifolds of order $\tau$.  Assume that $g\in \overline\Gamma_a^+$ and $L_{a}\in L^1(M)$ and  $\mathscr{B}^{a-1}\in L^1(\partial M)$. Suppose the mean curvature $h_g$  of $\partial M$ is non-negative on $\partial \mathbb{R}^n_+$. 
Then $\mathfrak{m}_{a,B}(g)\ge 0$. Equality holds if and only if
$u=0$, that is, $M$ is Euclidean half-space.
\end{thm}
Also we provide a Penrose-type inequality for the mass $\mathfrak{m}_{a,B}(g)$ in conformally flat manifolds.
\begin{thm}\label{thm on CF}
On AF CF manifolds $(M, g)$ of order $\tau$, assume that $g\in \overline\Gamma_a^+$ and $\int_{M}L_{a}dv_{g}+2a\int_{\partial M}\mathscr{B}^{a-1}d\sigma_{g}$ is finite. Suppose $\partial \Omega$ is a horizon on $(M, g)$ (i.e., $\partial \Omega=\partial M\backslash\partial \mathbb{R}^n_+  \subset M$ is minimal) and $u$ is constant on $\partial \Omega$.  Suppose the mean curvature $h_g$ of $\partial M$ is non-negative on $\partial \mathbb{R}^n_+\backslash {\Omega}$ and $\partial \Omega$ is strictly mean-convex in $\mathbb{R}^n$. Assume that the angle $\theta$ formed by the tangent plane of $\partial\Omega$ at $\partial\Omega'$ and $x_n=0$ satisfies $0< \theta< \pi$. 
Then 
\begin{align}
\mathfrak{m}_{a,B}(g)\ge & \frac{(a-1)!(n-a)!}{2(n-1)!\omega_{n-1}}\bigg(\int_{\partial\Omega} \sigma_{a-1}(\kappa_{\partial \Omega})\left(\frac{\sigma_1(\kappa_{\partial \Omega})}{n-1}\right)^a d\sigma\no\\
&+\int_{\partial\Omega'} \sigma_{a-2}(\kappa_{\partial \Omega'})\left(\frac{\sigma_1(\kappa_{\partial \Omega})}{n-1}\right)^a(\sin \theta)^{a-1}\cos \theta dl\bigg),
\end{align}
where $\kappa_{\partial\Omega}$ and
$\kappa_{\partial\Omega'}$ are the principal curvatures of $\partial\Omega$
and $\partial\Omega'.$  Equality holds if and only if $g$ is flat.
\end{thm}

The mean curvature of $\partial M$  and the convexity of $\partial \Omega$ make the angle $0< \theta\le \frac \pi 2$. When $\theta=\frac{\pi}{2}$, $u_n=0$ on $\partial \Omega'$ and $\mu=\nu$, where $\nu$ and $\mu$ are unit outer normal vectors of $\partial \Omega$ and $\partial \Omega'$ respectively.

% In \cite{GWW2} Ge, Wang, and Wu explored $m_{GWW}$ in conformally flat manifolds and built a positive mass theorem.
%See also Li-Nguyen \cite{LN}. Various developments have been established in \cite{GWW3,GWWX,Wang-Wu} and others. One of our motivations to study the mass in conformally flat manifolds comes from the Yamabe problem, see \cite{Schoen} and \cite{ABL}.  

Before we end the introduction, 
we mention that when $a=\frac{n}{2}$, Gauss-Bonnet-Chern formula is 
\[
(4\pi)^{\frac{n}{2}}\chi(M,\partial M)=\int_{M}E_{n}dv_{g}+\int_{\partial M}\sum_{i=0}^{\frac{n-2}{2}}Q_{n/2}^i d\sigma_{g},
\]
where 
\[
E_{n}=\left(2^{\frac{n}{2}}\left(\frac{n}{2}\right)!\right)^{-1}\sum_{i,j=1}^{n}\delta^{i_{1}\cdots i_{n}}_{j_{1}\cdots j_{n}}R_{i_{1}i_{2}}^{\quad j_{1}j_{2}}\cdots R_{i_{n-1}i_{n}}^{\quad\,\,j_{n-1}j_{n}},
\]
and 
\[
Q_{n/2}^i=\frac{2^{\frac{n}{2}-2i}}{i!(n-1-2i)!!}\sum_{\alpha,\beta=1}^{n-1}\delta^{\alpha_{1}\cdots\alpha_{n-1}}_{\beta_{1}\cdots\beta_{n-1}}R_{\alpha_{1}\alpha_{2}}^{\quad\beta_{1}\beta_{2}}\cdots R_{\alpha_{2i-1}\alpha_{2i}}^{\quad\beta_{2i-1}\beta_{2i}}L_{\alpha_{2i+1}}^{\beta_{2i+1}}\cdots L_{\alpha_{n-1}}^{\beta_{n-1}}.
\]
If $a=\frac{n}{2}$, then $\mathscr{B}^{a-1}$ is $Q_{\frac{n}{2}}^{\frac{n}{2}-1}$, differing by a constant multiple.

When \( a \neq \frac{n}{2} \), the reason for assuming integrability of only \( Q_{\frac{n}{2}}^{\frac{n}{2}-1} \) is that, under the decay assumption, other terms such as \( \int_{\partial M} Q_{n/2}^i \) for \( i \le \frac{n}{2} - 2 \) are finite and thus do not contribute to the corresponding boundary quantities on large spheres. More precisely, if we assume the decay rate \( \tau > \frac{n - 2a}{a + 1} \), then the curvature tensor and second fundamental form satisfy
\[
|R_{ij}^{\quad kl}| \le Cr^{-\tau - 2}, \quad |L_{\alpha}^{\beta}| \le Cr^{-\tau - 1}
\]
for some constant \( C \). Moreover, for \( i \le \frac{n}{2} - 2 \), we estimate
\[
Q_{n/2}^i \le Cr^{-(\tau + 2)i} \cdot r^{-(\tau + 1)(n - 1 - 2i)} \le Cr^{-\tau(\frac{n}{2} + 1) - n + 1}.
\]
Hence, when \( \tau > 0 \), the integral \( \int_{\partial M} Q_{n/2}^i \) is finite for all \( i \le \frac{n}{2} - 2 \), and the corresponding contributions on \( S_R^{n-2} \) vanish as \( R \to \infty \). Therefore, under this decay condition, the terms \( Q_{n/2}^i \) for \( i \le \frac{n}{2} - 2 \) cannot be detected by the finiteness assumption in Definition~\ref{mass defi}.

We remark that the critical points of the Gibbons--Hawking--York action under conformal variations are metrics with constant scalar curvature and umbilic boundary; see \cite{Lott} for the general variation formula. %Naturally, a positive mass for manifolds with non-compact boundary is closely connected to the boundary Yamabe problem; see \cite{Escobar,ABL}.
The mass introduced here is also related to an Einstein-type functional with boundary terms discussed in the Appendix; see there for a comparison with $\mathfrak{m}_{a,B}(g)$.

\

\noindent{\it The rest of the paper is organized as follows.} In Section~2 we introduce notation and prove a differential identity on asymptotically flat half-manifolds near infinity. We then show that the mass is well defined and coordinate independent (Theorem~\ref{thm:Main thm1-}).
In Section~3 we study asymptotically flat graphs and prove Theorem~\ref{thm:Main thm 2}. In Section~4 we treat locally conformally flat manifolds and relate the new mass to geometric integrals, proving Theorems~\ref{thm1 on CF} and \ref{thm on CF}. Finally, in the Appendix we introduce a Lovelock-type variational functional on manifolds with boundary and identify its Euler--Lagrange equations.

\

\noindent\textbf{Acknowledgments.} This work was carried out while W.~Wei was visiting the University of Freiburg supported by an Alexander von Humboldt Research Fellowship. She thanks the Institute of Mathematics at the University of Freiburg for its hospitality. She was also partially supported by NSFC (No.~12571218, No.~12271244). The authors also thank Mingwei Zhang for helpful conversations.

\section{Boundary GBC mass}

In this section we collect the analytic and algebraic identities needed to define the boundary Gauss--Bonnet--Chern (GBC) mass and to prove its invariance.
We start with the Gauss--Bonnet--Chern (Lovelock) curvature
\[
L_{a}=\frac{1}{2^{a}}\delta_{j_{1}\cdots j_{2a}}^{i_{1}\cdots i_{2a}}R_{i_{1}i_{2}}{}^{j_{1}j_{2}}\cdots R_{i_{2a-1}i_{2a}}{}^{j_{2a-1}j_{2a}}.
\]
Throughout we use the Einstein summation convention. We recall the following basic identities for the tensor $P_a^{ijkl}$ (see Ge--Wang--Wu \cite{GWW1}):
\begin{equation}\label{basic}
P_a^{stlm}R_{stlm}=L_a,
\end{equation}
\begin{equation}\label{eq:antisymm}
P_a^{ijkl}=-P_a^{jikl}=-P_a^{ijlk}=P_a^{klij},
\end{equation}
\begin{equation}\label{eq:div-free}
\nabla_{i}P_a^{ijkl}=\nabla_{j}P_a^{ijkl}=\nabla_{k}P_a^{ijkl}=\nabla_{l}P_a^{ijkl}=0,
\end{equation}
where $\nabla$ denotes the Levi--Civita connection of $g$.

Let $(M^{n},g)$ be an asymptotically flat half-manifold of order $\tau$. For large $r=|x|$, the curvature tensor admits the asymptotic expansion
\begin{equation}\label{behavior of Rm}
R_{kl}{}^{ij}=\frac{1}{2}\left(\frac{\partial^{2}g_{li}}{\partial x^{k}\partial x^{j}}+\frac{\partial^{2}g_{kj}}{\partial x^{l}\partial x^{i}}-\frac{\partial^{2}g_{ki}}{\partial x^{l}\partial x^{j}}-\frac{\partial^{2}g_{lj}}{\partial x^{k}\partial x^{i}}\right)+O\bigl(r^{-2\tau-2}\bigr).
\end{equation}
In particular,
\begin{equation}\label{asy of Rm}
|R_{kl}{}^{ij}|=O\bigl(r^{-\tau-2}\bigr),\qquad |\Gamma_{ij}^k|=O\bigl(r^{-\tau-1}\bigr).
\end{equation}
Moreover,
\begin{equation}\label{asy of Pa}
P_a^{ijkl}=O\bigl(r^{-(\tau+2)(a-1)}\bigr),
\end{equation}
and hence, using \eqref{eq:div-free},
\begin{equation}\label{eq:firstderivative0}
\partial_{i}P_a^{ijkl}=O\bigl(r^{-(\tau+2)(a-1)-(\tau+1)}\bigr).
\end{equation}
In particular,
\begin{equation}\label{eq:first derivative}
\partial_{\alpha}P_a^{n\beta n\alpha}=O\bigl(r^{-(\tau+2)(a-1)-(\tau+1)}\bigr).
\end{equation}

Finally, by \cite[(3.3)]{GWW1}\footnote{In \cite{GWW1} the authors write the formula for $L_{2}$; the general case follows by the same argument.}, we have the almost-divergence structure
\begin{equation}\label{eq:div structure}
L_{a}=2\partial_{i}\bigl(g_{jk,l}P_a^{ijkl}\bigr)+O\bigl(r^{-\tau(a+1)-2a}\bigr)\qquad\text{as }|x|\to\infty.
\end{equation}
This also follows directly from \eqref{basic}, \eqref{behavior of Rm}, and \eqref{eq:firstderivative0}.

To avoid ambiguity, indices $\alpha,\beta,\gamma,\eta$ range over $\{1,\dots,n-1\}$.

We begin with a technical lemma controlling the boundary curvature terms that appear in the mass formula.

\begin{lem}
Suppose that $(M^{n},g)$ is an asymptotically flat half-manifold of order $\tau$. On $\partial M$, for large $r=|x|$, we have
\begin{equation}\label{eq:i1 part}
\mathscr{B}^{a-1}=Q_{a,a-2}^{\eta\beta\alpha\gamma}g_{\beta\alpha,\gamma\eta}+O\bigl(r^{-(a+1)\tau-2a+1}\bigr),
\end{equation}
and
\begin{align}\label{eq:boundary term 2}
 g_{\beta\alpha,\gamma}P_a^{n\beta\alpha\gamma}
=(a-1)\mathscr{B}^{a-1}-(a-1)\partial_{\eta}\bigl(Q_{a,a-2}^{\eta\beta\alpha\gamma}g_{\beta\alpha,\gamma}\bigr)+O\bigl(r^{-(a+1)\tau-2a+1}\bigr).
\end{align}
Here $g_{\beta\alpha,\gamma\eta}=\partial_{\gamma}\partial_{\eta}g_{\beta\alpha}$ denotes the ordinary second partial derivative.
\end{lem}

\begin{proof}
We first prove \eqref{eq:i1 part}. Substituting \eqref{behavior of Rm} into the definition of $\mathscr{B}^{a-1}$ yields
\begin{align}
&2^{a-1}\mathscr{B}^{a-1}\no\\
 =&\sum_{i,j=1}^{n-1}\delta_{ j_{2}\cdots j_{2a-3}j_{2a-2}\alpha\gamma}^{ i_{2}\cdots i_{2a-3}i_{2a-2}\eta\beta}L_{i_{2}}^{j_{2}}R_{\eta\beta}^{\quad\,\,\alpha\gamma}R_{i_{3}i_{4}}{}^{j_{3}j_{4}}\cdots R_{i_{2a-3}i_{2a-2}}{}^{j_{2a-3}j_{2a-2}}\no\\
= & \sum_{i,j=1}^{n-1}\delta_{ j_{2}\cdots j_{2a-2}\alpha\gamma}^{ i_{2}\cdots i_{2a-2}\eta\beta}L_{i_{2}}^{j_{2}}\frac{1}{2}\bigg(\frac{\partial^{2}g_{\gamma \eta}}{\partial x^{\alpha}\partial x^{\beta}}+\frac{\partial^{2}g_{\alpha\beta}}{\partial x^{\gamma}\partial x^{\eta}}-\frac{\partial^{2}g_{\alpha \eta}}{\partial x^{\gamma}\partial x^{\beta}}-\frac{\partial^{2}g_{\gamma\beta}}{\partial x^{\alpha}\partial x^{\eta}}\bigg)\cdots R_{i_{2a-3}i_{2a-2}}{}^{j_{2a-3}j_{2a-2}}\nonumber \\
 & +O\bigl(r^{-(a+1)\tau-2a+1}\bigr)\nonumber \\
= & 2\sum_{i,j=1}^{n-1}\delta_{ j_{2}\cdots j_{2a-3}j_{2a-2}\alpha\gamma}^{ i_{2}\cdots i_{2a-3}i_{2a-2}\eta\beta}L_{i_{2}}^{j_{2}}g_{\beta\alpha,\gamma \eta}R_{i_{3}i_{4}}{}^{j_{3}j_{4}}\cdots R_{i_{2a-3}i_{2a-2}}{}^{j_{2a-3}j_{2a-2}}+O\bigl(r^{-(a+1)\tau-2a+1}\bigr).\nonumber \\
=& 2^{a-1}Q_{a,a-2}^{\eta\beta\alpha\gamma}g_{\beta\alpha,\gamma\eta}+O\bigl(r^{-(a+1)\tau-2a+1}\bigr).\no
\end{align}
This proves \eqref{eq:i1 part}. We next analyze the term $g_{\beta\alpha,\gamma}P_a^{n\beta\alpha\gamma}$.

For $1\le p\le 2a-2$, set
\[
\mathrm{I}_{p}:=\frac{1}{2^{a}}\sum_{i,j=1}^{n-1}\delta_{j_{1}\cdots j_{p-1}n\cdots j_{2a-2}\alpha\gamma}^{i_{1}\cdots i_{p-1}i_{p}\cdots i_{2a-2}n\beta}R_{i_{1}i_{2}}{}^{j_{1}j_{2}}\cdots R_{i_{p}i_{p+1}}^{\quad\quad nj_{p+1}}R_{i_{2a-3}i_{2a-2}}{}^{j_{2a-3}j_{2a-2}}g_{\beta\alpha,\gamma}.
\]
By symmetry, $\mathrm{I}_{p}=\mathrm{I}_{1}$ for all $p$.

Consequently,
\begin{align}\label{boundary term1}
 & g_{\beta\alpha,\gamma}P_a^{n\beta\alpha\gamma}\no\\
%= & \frac{1}{2^{a}}\delta_{j_{1}\cdots j_{2a-2}j_{2a-1}j_{2a}}^{i_{1}\cdots i_{2a-2}n\beta}R_{i_{1}i_{2}}{}^{j_{1}j_{2}}\cdots R_{i_{2a-3}i_{2a-2}}{}^{j_{2a-3}j_{2a-2}}g^{j_{2a-1}\alpha}g^{j_{2a}\gamma}g_{\beta\alpha,\gamma}\no\\
= & \frac{1}{2^{a}}\delta_{j_{1}\cdots j_{2a-2}\alpha\gamma}^{i_{1}\cdots i_{2a-2}n\beta}R_{i_{1}i_{2}}{}^{j_{1}j_{2}}\cdots R_{i_{2a-3}i_{2a-2}}{}^{j_{2a-3}j_{2a-2}}g_{\beta\alpha,\gamma}+O(r^{-(\tau+2)(a-1)-(2\tau+1)})\no\\
= & \frac{1}{2^{a}}\sum_{i,j=1}^{n-1}\delta_{nj_{2}\cdots j_{2a-2}\alpha\gamma}^{i_{1}i_{2}\cdots i_{2a-2}n\beta}R_{i_{1}i_{2}}{}^{nj_{2}}\cdots R_{i_{2a-3}i_{2a-2}}{}^{j_{2a-3}j_{2a-2}}g_{\beta\alpha,\gamma}\no\\
 & +\frac{1}{2^{a}}\sum_{i,j=1}^{n-1}\delta_{j_{1}n\cdots j_{2a-3}j_{2a-2}\alpha\gamma}^{i_{1}i_{2}\cdots i_{2a-3}i_{2a-2}n\beta}R_{i_{1}i_{2}}{}^{j_{1}n}\cdots R_{i_{2a-3}i_{2a-2}}{}^{j_{2a-3}j_{2a-2}}g_{\beta\alpha,\gamma}\no\\
 & +\cdots\no\\
 & +O(r^{-(\tau+2)(a-1)-(2\tau+1)})\no\\
= & \mathrm{I_{1}+\cdots+I_{2a-2}}+O(r^{-(\tau+2)(a-1)-(2\tau+1)})\no\\
= & 2(a-1)\mathrm{I_{1}}+O(r^{-(\tau+2)(a-1)-(2\tau+1)}).
\end{align}

By \eqref{behavior of Rm} together with symmetry and antisymmetry of indices, we have
\[
\sum_{i,j=1}^{n-1}\delta_{nj_{2}\cdots j_{2a-2}\alpha\gamma}^{i_{1}i_{2}\cdots i_{2a-2}n\beta}L_{i_{2}}^{j_{2}}\bigg(R_{i_{3}i_{4}}{}^{j_{3}j_{4}}\cdots R_{i_{2a-3}i_{2a-2}}{}^{j_{2a-3}j_{2a-2}}\bigg)_{i_{1}}g_{\beta\alpha,\gamma}=O\bigl(r^{-(a+1)\tau-2a+1}\bigr).
\]
Combining this with the definition of $Q_{a,a-2}^{\eta\beta\alpha\gamma}$ yields
\begin{align}\label{deriv of Q}
&\sum_{\eta=1}^{n-1}\partial_{\eta }(Q_{a,a-2}^{\eta\beta \alpha\gamma})g_{\beta \alpha,\gamma}\no\\
&=\frac{1}{2^{a-2}}\sum_{i,j=1}^{n-1}\delta_{jj_{1}\cdots j_{2a-4}\alpha\gamma}^{ii_{1}\cdots i_{2a-4}\eta\beta}\partial_{\eta}(L_{i}^{j})R_{i_{1}i_{2}}{}^{j_{1}j_{2}}\cdots R_{i_{2a-5}i_{2a-4}}{}^{j_{2a-5}j_{2a-4}}g_{\beta \alpha,\gamma}+O(r^{-(a+1)\tau-2a+1}).
\end{align}

Combining with $R_{i_{1}i_{2}}{}^{nj_{2}}=L_{i_{2},i_{1}}^{j_{2}}-L_{i_{1},i_{2}}^{j_{2}}+O(r^{-2\tau-2})$
for $1\le i_{1},j_{2},i_{2}\le n-1,$ it follows that 
\begin{align}
&2^{a}\mathrm{I_{1}}\no\\
= & \sum_{i,j=1}^{n-1}\delta_{nj_{2}\cdots j_{2a-3}j_{2a-2}\alpha\gamma}^{i_{1}i_{2}\cdots i_{2a-3}i_{2a-2}n\beta}R_{i_{1}i_{2}}{}^{nj_{2}}\cdots R_{i_{2a-3}i_{2a-2}}{}^{j_{2a-3}j_{2a-2}}g_{\beta\alpha,\gamma}\no\\
= & \sum_{i,j=1}^{n-1}\delta_{nj_{2}\cdots j_{2a-3}j_{2a-2}\alpha\gamma}^{i_{1}i_{2}\cdots i_{2a-3}i_{2a-2}n\beta}(L_{i_{2},i_{1}}^{j_{2}}-L_{i_{1},i_{2}}^{j_{2}})R_{i_{3}i_{4}}{}^{j_{3}j_{4}}\cdots R_{i_{2a-3}i_{2a-2}}{}^{j_{2a-3}j_{2a-2}}g_{\beta\alpha,\gamma}+O(r^{-(a+1)\tau-2a+1})\nonumber \\
= & 2\sum_{i,j=1}^{n-1}\delta_{nj_{2}\cdots j_{2a-3}j_{2a-2}\alpha\gamma}^{i_{1}i_{2}\cdots i_{2a-3}i_{2a-2}n\beta}L_{i_{2},i_{1}}^{j_{2}}R_{i_{3}i_{4}}{}^{j_{3}j_{4}}\cdots R_{i_{2a-3}i_{2a-2}}{}^{j_{2a-3}j_{2a-2}}g_{\beta\alpha,\gamma}+O(r^{-(a+1)\tau-2a+1})\nonumber \\
%= & 2\sum_{i,j=1}^{n-1}\delta_{nj_{2}\cdots j_{2a-3}j_{2a-2}\alpha\gamma}^{i_{1}i_{2}\cdots i_{2a-3}i_{2a-2}n\beta}\bigg(L_{i_{2}}^{j_{2}}R_{i_{3}i_{4}}{}^{j_{3}j_{4}}\cdots R_{i_{2a-3}i_{2a-2}}{}^{j_{2a-3}j_{2a-2}}g_{\beta\alpha,\gamma}\bigg)_{i_{1}}\nonumber \\
% & +2\sum_{i,j=1}^{n-1}\delta_{\alpha j_{2}\cdots j_{2a-3}j_{2a-2}\gamma}^{i_{1}i_{2}\cdots i_{2a-3}i_{2a-2}\beta}L_{i_{2}}^{j_{2}}R_{i_{3}i_{4}}{}^{j_{3}j_{4}}\cdots R_{i_{2a-3}i_{2a-2}}{}^{j_{2a-3}j_{2a-2}}g_{\beta\alpha,\gamma i_{1}}\nonumber \\
 %& +O(r^{-(a+1)\tau-2a+1}),\nonumber 
 =&-2^{a-1}\sum_{i_1=1}^{n-1}\partial_{i_1}(Q_{a,a-2}^{i_1\beta \alpha\gamma}g_{\beta \alpha,\gamma})+2^{a-1}Q_{a,a-2}^{i_1\beta \alpha\gamma}g_{\beta \alpha,\gamma i_1}+O(r^{-(a+1)\tau-2a+1}),\no 
\end{align}
where in the last equality is due to \eqref{deriv of Q} and $\delta_{nj_{2}\cdots j_{2a-3}j_{2a-2}\alpha\gamma}^{i_{1}i_{2}\cdots i_{2a-3}i_{2a-2}n\beta}=-\delta_{\alpha j_{2}\cdots j_{2a-3}j_{2a-2}\gamma}^{i_{1}i_{2}\cdots i_{2a-3}i_{2a-2}\beta}$.

Combining with \eqref{eq:i1 part}, we obtain
\begin{align*}
\mathrm{I_{1}}= & -\frac 1 2 \sum_{i_1,\beta, \alpha,\gamma =1}^{n-1}\partial_{i_1}\bigl(Q_{a,a-2}^{i_1\beta \alpha\gamma}g_{\beta \alpha,\gamma}\bigr)+\frac 1 2  \mathscr{B}^{a-1}+O\bigl(r^{-(a+1)\tau-2a+1}\bigr).
\end{align*}
By \eqref{boundary term1}, we derive
\begin{align}
  g_{\beta\alpha,\gamma}P_a^{n\beta\alpha\gamma}
=-(a-1) \sum_{i_1,\beta, \alpha,\gamma =1}^{n-1}\partial_{i_1}\bigl(Q_{a,a-2}^{i_1\beta \alpha\gamma}g_{\beta \alpha,\gamma}\bigr)+(a-1)\mathscr{B}^{a-1}+O\bigl(r^{-(a+1)\tau-2a+1}\bigr).\no
\end{align}
This proves \eqref{eq:boundary term 2} and completes the proof.
\end{proof}

We next record the following key identity for the boundary curvature.
\begin{thm}\label{thm:key boundary equality}
Suppose that $(M^{n},g)$ is an asymptotically flat half-manifold of order $\tau$.
On $\partial M$, for $r=|x|$ sufficiently large, we have
\begin{equation}\label{eq:equality near infinity}
g_{jk,l}P_{a}^{njkl}=(g_{n\beta}P_a^{n\beta\alpha n})_{,\alpha}-(a-1)\partial_{\eta}\bigl(Q_{a,a-2}^{\eta\beta\alpha\gamma}g_{\beta\alpha,\gamma}\bigr)+a\mathscr{B}^{a-1}+O\bigl(r^{-(a+1)\tau-2a+1}\bigr),
\end{equation}
where $(g_{n\beta}P_a^{n\beta\alpha n})_{,\alpha}:=\partial_{\alpha}(g_{n\beta}P_a^{n\beta\alpha n})$ denotes an ordinary partial derivative.
\end{thm}

\begin{proof}
First, we decompose
\begin{align}
g_{jk,l}P_a^{njkl} & =g_{jn,l}P_a^{njnl}+g_{j\alpha,l}P_a^{nj\alpha l}\no\\
 & =g_{\alpha n,\beta}P_a^{n\alpha n\beta}+g_{\beta\alpha,l}P_a^{n\beta\alpha l}\nonumber \\
 & =g_{\alpha n,\beta}P_a^{n\alpha n\beta}+g_{\beta\alpha,n}P_a^{n\beta\alpha n}+g_{\beta\alpha,\gamma}P_a^{n\beta\alpha\gamma}.\label{eq:equality 1}
\end{align}
From \eqref{eq:asymptotic behavior} and the definition of $L_{\alpha\beta}$, we have
\begin{equation}\label{eq:second fundamental behavior}
\frac{\partial g_{\alpha\beta}}{\partial x^{n}}=\frac{\partial g_{n\alpha}}{\partial x^{\beta}}+\frac{\partial g_{n\beta}}{\partial x^{\alpha}}-2L_{\alpha\beta}+O\bigl(r^{-2\tau-1}\bigr).
\end{equation}
Meanwhile, $|L_{\alpha\beta}|=O\bigl(r^{-\tau-1}\bigr)$ by \eqref{eq:asymptotic behavior}.

Combining \eqref{eq:equality 1}, \eqref{eq:second fundamental behavior}, and \eqref{asy of Pa}, we obtain
\begin{align}\label{eq:boundary expansion 1}
 & g_{jk,l}P_a^{njkl}\no\\
= & \left(\frac{\partial g_{n\alpha}}{\partial x^{\beta}}+\frac{\partial g_{n\beta}}{\partial x^{\alpha}}-2L_{\alpha\beta}\right)P_a^{n\beta\alpha n}+g_{\alpha n,\beta}P_a^{n\alpha n\beta}+g_{\beta\alpha,\gamma}P_a^{n\beta\alpha\gamma}\no\\
&+O\bigl(r^{-2\tau-1-(2+\tau)(a-1)}\bigr)\nonumber \\
= & g_{n\beta,\alpha}P_a^{n\beta\alpha n}+g_{\beta\alpha,\gamma}P_a^{n\beta\alpha\gamma}+2L_{\alpha\beta}P_a^{n\beta n\alpha}+O\bigl(r^{-\tau(a+1)-2a+1}\bigr),
\end{align}
where we used \eqref{eq:antisymm}.

Recall
\[
P_a^{n\beta n\alpha}=\frac{1}{2^{a}}\sum_{i,j=1}^{n}\delta_{j_{1}\cdots j_{2a-2}j_{2a-1}j_{2a}}^{i_{1}\cdots i_{2a-2}n\beta}R_{i_{1}i_{2}}{}^{j_{1}j_{2}}\cdots R_{i_{2a-3}i_{2a-2}}{}^{j_{2a-3}j_{2a-2}}g^{j_{2a-1}n}g^{j_{2a}\alpha},
\]
where, by a slight abuse of notation, the symbols $i,j$ are used to indicate indices ranging over $\{1,\dots,n\}$.

It follows from \eqref{asy of Rm} that
\begin{align}\label{eq:LR}
P_a^{n\beta n\alpha}L_{\alpha\beta}
=\frac{1}{2}\mathscr{B}^{a-1}+O\bigl(r^{-(\tau+2)(a-1)-\tau-(\tau+1)}\bigr).
\end{align}
Moreover,
\begin{equation}\label{eq:derivative by parts}
g_{n\beta,\alpha}P_a^{n\beta\alpha n}=(g_{n\beta}P_a^{n\beta\alpha n})_{,\alpha}-g_{n\beta}(P_a^{n\beta\alpha n})_{,\alpha},
\end{equation}
where the derivatives are ordinary partial derivatives. Thus, using \eqref{eq:derivative by parts}, \eqref{eq:first derivative}, and \eqref{eq:asymptotic behavior}, we have
\begin{equation}\label{eq:boundary term 1}
g_{n\beta,\alpha}P_a^{n\beta\alpha n}=(g_{n\beta}P_a^{n\beta\alpha n})_{,\alpha}+O\bigl(r^{-(a+1)\tau-2a+1}\bigr).
\end{equation}

Finally, combining \eqref{eq:boundary expansion 1}, \eqref{eq:boundary term 1}, \eqref{eq:boundary term 2}, and \eqref{eq:LR} yields \eqref{eq:equality near infinity}.

\iffalse
\begin{align*}
 & g_{jk,l}P_a^{njkl}\\
= & (g_{n\beta}P_a^{n\beta\alpha n})_{,\alpha}+\frac{a-1}{2^{a-2}}\delta_{j_{3}\cdots j_{2a-2}nj_{2}\alpha\gamma}^{i_{3}\cdots i_{2a-2}i_{1}i_{2}n\beta}\bigg(L_{i_{2}}^{j_{2}}R_{i_{3}i_{4}}{}^{j_{3}j_{4}}\cdots R_{i_{2a-3}i_{2a-2}}{}^{j_{2a-3}j_{2a-2}}g_{\beta\alpha,\gamma}\bigg)_{i_{1}}\\
 & +\frac{a}{2^{a-1}}\delta_{\alpha j_{2}\cdots j_{2a-3}j_{2a-2}\gamma}^{i_{1}i_{2}\cdots i_{2a-3}i_{2a-2}\beta}L_{i_{2}}^{j_{2}}R_{i_{1}\beta}^{\quad\,\,\alpha\gamma}R_{i_{3}i_{4}}{}^{j_{3}j_{4}}\cdots R_{i_{2a-3}i_{2a-2}}{}^{j_{2a-3}j_{2a-2}}\\
 & +O(r^{-(a+2)\tau-2a+1}).
\end{align*}
\fi
\end{proof}
Recall $S_{R}^{+}=\{x=(x^{1},\cdots,x^{n})||x|=R,\,x^{n}>0\}$, $S_{R}^{n-2}=\{x=(x^{1},\cdots,x^{n})||x|=R,\,x^{n}=0\}$
and $B_{R}^{n-1}=\{x=(x^{1},\cdots,x^{n})||x|\le R,\,x^{n}=0\}.$

From (\ref{eq:div structure}), let $\nu$ be the outer normal vector
on $\partial B_{R}^{+}$, it holds that

\begin{align}
\int_{B_{R}^{+}}L_{a}dv_{g}  %=2\int_{S_{R}^{+}}g_{jk,l}P_a^{ijkl}\nu_{i}d\sigma+2\int_{B_{R}^{n-1}}g_{jk,l}P_a^{ijkl}\nu_{i}d\sigma+O(1)\label{eq:integration}\\
 =2\int_{S_{R}^{+}}g_{jk,l}P_a^{ijkl}\nu_{i}d\sigma-2\int_{B_{R}^{n-1}}g_{jk,l}P_a^{njkl}d\sigma+O(1).\nonumber 
\end{align}
Combining this with Theorem~\ref{thm:key boundary equality}, we obtain
\begin{align*}
\int_{B_{R}^{+}}L_{a}dv_{g}
= & \,2\int_{S_{R}^{+}}g_{jk,l}P_a^{ijkl}\nu_{i}\,d\sigma+O(1)\\
&-2\int_{B_{R}^{n-1}}\left((g_{n\beta}P_a^{n\beta\alpha n})_{,\alpha}-(a-1)\partial_{\eta}\bigl(Q_{a,a-2}^{\eta\beta\alpha\gamma}g_{\beta\alpha,\gamma}\bigr)+a\mathscr{B}^{a-1}\right)d\sigma.
\end{align*}
Therefore,
\begin{align}
 & \int_{B_{R}^{+}}L_{a}dv_{g}+2a\int_{B_{R}^{n-1}}\mathscr{B}^{a-1}d\sigma_{g}\label{eq:volume and surface interg}\\
= & \,2\int_{S_{R}^{+}}g_{jk,l}P_a^{ijkl}\nu_{i}\,d\sigma+2\int_{S_{R}^{n-2}}g_{n\beta}P_a^{n\beta n\alpha}\mu_{\alpha}\,dl\nonumber \\
 & +2(a-1)\int_{S_{R}^{n-2}}Q_{a,a-2}^{\eta\beta\alpha\gamma}g_{\beta\alpha,\gamma}\mu_{\eta}\,dl+O(1).\nonumber
\end{align}
Here $\mu$ denotes the outward unit normal vector on $S_{R}^{n-2}\subset B_{R}^{n-1}.$ 

As a consequence of \eqref{eq:volume and surface interg}, we obtain the following.
\begin{thm}
Suppose that $(M^{n},g)$ is an asymptotically flat half-manifold of decay order $\tau>\frac{n-2a}{a+1}$ and that
$\int_{M}L_{a}dv_{g}+2a\int_{\partial M}\mathscr{B}^{a-1}d\sigma_{g}$ is finite. Then the mass $\mathfrak{m}_{a,B}(g)$ is well defined.
\end{thm}

\begin{proof}
Denote
\[
A_{R^{*},R}^{n-1}=\{(x^1,\cdots,x^n)\mid R^{*}\le |x|\le R,\ x^{n}=0\},
\qquad
A_{R^{*},R}^{+}=\{(x^1,\cdots,x^n)\mid R^{*}\le |x|\le R,\ x^{n}>0\}.
\]

Since $\int_{M}L_{a}dv_{g}+2a\int_{\partial M}\mathscr{B}^{a-1}d\sigma_{g}$ is finite, by \eqref{eq:div structure} and Theorem~\ref{thm:key boundary equality} we have
\begin{align*}
0 =&\lim_{R\to\infty,\, R^{*}\to\infty}\left(\frac{1}{2}\int_{A_{R^{*},R}^{_{+}}}L_{a}dv_{g}+a\int_{A_{R^{*},R}^{n-1}}\mathscr{B}^{a-1}d\sigma_{g}\right)\\
= & \,\int_{S_{R}^{+}}g_{jk,l}P_a^{ijkl}\nu_{i}\,d\sigma+\int_{S_{R}^{n-2}}g_{n\beta}P_a^{n\beta n\alpha}\mu_{\alpha}\,dl+(a-1)\int_{S_{R}^{n-2}}Q_{a,a-2}^{\eta\beta\alpha\gamma}g_{\beta\alpha,\gamma}\mu^{\eta}\,dl\\
 & -\bigg(\int_{S_{R^{*}}^{+}}g_{jk,l}P_a^{ijkl}\nu_{i}\,d\sigma+\int_{S_{R^{*}}^{n-2}}g_{n\beta}P_a^{n\beta n\alpha}\mu_{\alpha}\,dl+(a-1)\int_{S_{R^{*}}^{n-2}}Q_{a,a-2}^{\eta\beta\alpha\gamma}g_{\beta\alpha,\gamma}\mu^{\eta}\,dl\bigg)\\
 & +\int_{R^{*}}^{R}O\bigl(r^{-(a+1)\tau-2a+1}\bigr)r^{n-2}dr+\int_{R^{*}}^{R}O\bigl(r^{-\tau(a+1)-2a}\bigr)r^{n-1}dr.
\end{align*}
This shows that the limit defining $\mathfrak{m}_{a,B}(g)$ exists.
\end{proof}
Note that the Gauss--Codazzi equation implies
\begin{equation}\label{eq:Guass-Codazzi}
\overline{R}_{\alpha\beta\gamma}{}^\delta=R_{\alpha\beta\gamma}{}^\delta-(L_{\alpha}^\delta L_{\beta\gamma}-L_{\alpha\gamma}L_{\beta}^\delta),
\end{equation}
where $\overline{R}_{\alpha\beta\gamma\delta}$ is the curvature tensor of the induced metric $\bar{g}$ on $\partial M$.
For asymptotically flat half-manifolds with non-compact boundary, this leads to an alternative boundary mass definition:
\begin{align*}
\overline{\mathfrak{m}}_{a,B}(g):= & \,\lim_{R\to\infty}c(n,a)\bigg(\int_{S_{R}^{+}}g_{jk,l}P_a^{ijkl}\nu_{i}\,d\sigma+\int_{S_{R}^{n-2}}g_{n\beta}\overline{P}_{a}^{n\beta n\alpha}\mu_{\alpha}\,dl\\
 & \quad\quad\quad\quad+(a-1)\int_{S_{R}^{n-2}}\overline{Q}_{a,a-2}^{\eta\beta\alpha\gamma}g_{\beta\alpha,\gamma}\mu_{\eta}\,dl\bigg),
\end{align*}
where
\begin{align*}
\overline{Q}_{a,a-2}^{\eta\beta\alpha\gamma} :=\frac{1}{2^{a-2}}\sum_{i,j=1}^{n-1}\delta_{jj_{1}\cdots j_{2a-5}j_{2a-4}j_{2a-3}j_{2a-2}}^{ii_{1}\cdots i_{2a-5}i_{2a-4}\eta\beta}L_{i}^{j}\overline{R}_{i_{1}i_{2}}{}^{j_{1}j_{2}}\cdots\overline{R}_{i_{2a-5}i_{2a-4}}{}^{j_{2a-5}j_{2a-4}}g^{j_{2a-3}\alpha}g^{j_{2a-2}\gamma},
\end{align*}
and
\[
\overline{P}_{a}^{n\beta\alpha n}:=\frac{1}{2^{a}}\sum_{\substack{i_1,\cdots,i_{2a-2}\\j_1,\cdots, j_{2a-2}=1}}^{n-1}\sum_{j_{2a-1}, j_{2a}=1}^n\delta_{j_{1}j_{2}\cdots j_{2a-3}j_{2a-2}j_{2a-1}j_{2a}}^{i_{1}i_{2}\cdots i_{2a-3}i_{2a-2}n\beta}\overline{R}_{i_{1}i_{2}}{}^{j_{1}j_{2}}\cdots\overline{R}_{i_{2a-3}i_{2a-2}}{}^{j_{2a-3}j_{2a-2}}g^{j_{2a-1}\alpha}g^{j_{2a}n}.
\]
By \eqref{eq:Guass-Codazzi} and \eqref{eq:asymptotic behavior}, we obtain
\[
\int_{S_{R}^{n-2}}g_{n\beta}P_a^{n\beta n\alpha}\mu_{\alpha}\,dl=\int_{S_{R}^{n-2}}g_{n\beta}\overline{P}_{a}^{n\beta n\alpha}\mu_{\alpha}\,dl+O\bigl(R^{-(a+1)\tau-2a+n}\bigr),
\]
\[
\int_{S_{R}^{n-2}}Q_{a,a-2}^{\eta\beta\alpha\gamma}g_{\beta\alpha,\gamma}\mu_{\eta}\,dl=\int_{S_{R}^{n-2}}\overline{Q}_{a,a-2}^{\eta\beta\alpha\gamma}\overline{g}_{\beta\alpha,\gamma}\mu_{\eta}\,dl+O\bigl(R^{-(a+1)\tau-2a+n}\bigr).
\]
It follows that $\overline{\mathfrak{m}}_{a,B}(g)=\mathfrak{m}_{a,B}(g)$.

We now prove that the mass $\mathfrak{m}_{a,B}(g)$ is independent of the choice of asymptotically flat coordinates. The strategy follows \cite{Bartnik, ABL} and \cite{GWW1}. 
\begin{thm}
Suppose that $(M^{n},g)$ is an asymptotically flat half-manifold of decay order $\tau>\frac{n-2a}{a+1}$ and that
$\int_{M}L_{a}dv_{g}+2a\int_{\partial M}\mathscr{B}^{a-1}d\sigma_{g}$ is finite. Then $\mathfrak{m}_{a,B}(g)$ depends only on $g$.
\end{thm}

\begin{proof}
Suppose $\{y^{i}\}$ and $\{x^{i}\}$ are two choices of asymptotic coordinates. As in \cite{ABL}, after possibly replacing $\tau$ by $\tau-\varepsilon$ for small $\varepsilon>0$, we may write
\[
y^{i}=x^{i}+\varphi^{i},\qquad \varphi^{i}\in C_{1-\tau}^{2,\alpha}.
\footnote{Proposition~3.8 in \cite{ABL} is stated for $\tau>(n-2)/2$, but the same proof works for $\tau>\frac{n-2a}{a+1}$; we omit the details.}
\]
Denote $\partial_{i}=\frac{\partial}{\partial x^{i}}$ and $\hat{\partial}_{i}=\frac{\partial}{\partial y^{i}}$, and set $g_{ij}=g(\partial_{i},\partial_{j})$ and $\hat{g}_{ij}=g(\hat{\partial}_{i},\hat{\partial}_{j})$.

Write $r=|x|$ and $s=|y|$; then $C^{-1}r\le s\le Cr$ for some constant $C>1$.

From \cite[(3.19)--(3.20)]{ABL}, for large $r=|x|$ we have
\begin{equation}\label{eq:asy 1}
\hat{\partial}_{i}=\partial_{i}-\partial_{i}\varphi^{s}\partial_{s}+O\bigl(r^{-\tau}\bigr),
\qquad
\partial_{\alpha}\varphi^{n}=0\quad\text{on }\partial M.
\end{equation}
Moreover,
\begin{equation}\label{eq:asy2}
\hat{g}_{ij}=g_{ij}-\partial_{i}\varphi^{j}-\partial_{j}\varphi^{i}+O\bigl(r^{-2\tau}\bigr),
\end{equation}
and
\begin{equation}\label{eq:asy3}
\hat{\partial}_{k}\hat{g}_{ij}=\partial_{k}g_{ij}-\partial_{k}\partial_{i}\varphi^{j}-\partial_{k}\partial_{j}\varphi^{i}+O\bigl(r^{-2\tau-1}\bigr).
\end{equation}
Note that
$\hat{\Gamma}_{\beta\gamma}^{n}-\Gamma_{\beta\gamma}^{n}=-\partial_{\gamma\beta}\varphi^{n}+O\bigl(r^{-2\tau-1}\bigr)=O\bigl(r^{-2\tau-1}\bigr)$,
and hence
\begin{equation}\label{eq:the second form}
\hat{L}_{\beta\gamma}-L_{\beta\gamma}=O\bigl(r^{-2\tau-1}\bigr).
\end{equation}
 We know that
\begin{equation}\label{eq:Rm difference}
\hat{R}_{ij}^{\,\,\,\,kl}=R_{ij}^{\,\,\,\,kl}+O\bigl(r^{-2-2\tau}\bigr).
\end{equation}
This can be seen as in \cite[p.~96]{GWW1}: indeed,
\begin{align*}
R_{i j \,l}^{\,\,\,\,k}-\hat R_{i j \,l}^{\,\,\,\,k}
= & -\frac{1}{2}\bigg(\partial_i \partial_k \partial_j \varphi^l+\partial_i \partial_k \partial_l \varphi^j-\partial_i \partial_l \partial_j \varphi^k-\partial_i \partial_l \partial_k \varphi^j-\partial_j \partial_k \partial_i \varphi^l \\
& \quad \quad -\partial_j \partial_k \partial_l \varphi^i+\partial_j \partial_l \partial_i \varphi^k+\partial_j \partial_l \partial_k \varphi^i\bigg)+O\bigl(r^{-2-2\tau}\bigr)
= O\bigl(r^{-2-2\tau}\bigr).
\end{align*}

Denote
\[
S_{R}^{+}=\{x=(x^{1},\cdots,x^{n})\mid r=R,\ x^{n}>0\},
\quad
\hat{S}_{R}^{+}=\{y=(y^{1},\cdots,y^{n})\mid s=R,\ y^{n}>0\},
\]
and
\[
S_{R}^{n-2}=\{x=(x^{1},\cdots,x^{n})\mid r=R,\ x^{n}=0\},
\quad
\hat{S}_{R}^{n-2}=\{y=(y^{1},\cdots,y^{n})\mid s=R,\ y^{n}=0\}.
\]
\medskip
Let
\begin{align*}
I_{R}= & \,\int_{S_{R}^{+}}g_{jk,l}P_a^{ijkl}\nu_{i}\,d\sigma+\int_{S_{R}^{n-2}}g_{n\beta}P_a^{n\beta n\alpha}\mu_{\alpha}\,dl+(a-1)\int_{S_{R}^{n-2}}Q_{a,a-2}^{\eta\beta\alpha\gamma}g_{\beta\alpha,\gamma}\mu_{\eta}\,dl\\
 & -\int_{\hat{S}_{R}^{+}}\hat{g}_{jk,l}\hat{P}_{a}^{ijkl}\hat{\nu}_{i}\,d\sigma-\int_{\hat{S}_{R}^{n-2}}\hat{g}_{n\beta}\hat{P}_{a}^{n\beta n\alpha}\hat{\mu}_{\alpha}\,dl-(a-1)\int_{\hat{S}_{R}^{n-2}}\hat{Q}_{a,a-2}^{\eta\beta\alpha\gamma}\hat{g}_{\beta\alpha,\gamma}\hat{\mu}_{\eta}\,dl\\
= & \,\mathrm{I}_{1}+\mathrm{I}_{2},
\end{align*}
where
\begin{align*}
\mathrm{I}_{1}= & \,\int_{S_{R}^{+}}g_{jk,l}P_a^{ijkl}\nu_{i}\,d\sigma+\int_{S_{R}^{n-2}}g_{n\beta}P_a^{n\beta n\alpha}\mu_{\alpha}\,dl+(a-1)\int_{S_{R}^{n-2}}Q_{a,a-2}^{\eta\beta\alpha\gamma}g_{\beta\alpha,\gamma}\mu_{\eta}\,dl\\
 & -\int_{S_{R}^{+}}\hat{g}_{jk,l}\hat{P}_{a}^{ijkl}\nu_{i}\,d\sigma-\int_{S_{R}^{n-2}}\hat{g}_{n\beta}\hat{P}_{a}^{n\beta n\alpha}\mu_{\alpha}\,dl-(a-1)\int_{S_{R}^{n-2}}\hat{Q}_{a,a-2}^{\eta\beta\alpha\gamma}\hat{g}_{\beta\alpha,\gamma}\mu_{\eta}\,dl,
\end{align*}
and
\begin{align*}
\mathrm{I}_{2}= & \,\int_{S_{R}^{+}}\hat{g}_{jk,l}\hat{P}_{a}^{ijkl}\nu_{i}\,d\sigma+\int_{S_{R}^{n-2}}\hat{g}_{n\beta}\hat{P}_{a}^{n\beta n\alpha}\mu_{\alpha}\,dl+(a-1)\int_{S_{R}^{n-2}}\hat{Q}_{a,a-2}^{\eta\beta\alpha\gamma}\hat{g}_{\beta\alpha,\gamma}\mu_{\eta}\,dl\\
 & -\int_{\hat{S}_{R}^{+}}\hat{g}_{jk,l}\hat{P}_{a}^{ijkl}\hat{\nu}_{i}\,d\sigma-\int_{\hat{S}_{R}^{n-2}}\hat{g}_{n\beta}\hat{P}_{a}^{n\beta n\alpha}\hat{\mu}_{\alpha}\,dl-(a-1)\int_{\hat{S}_{R}^{n-2}}\hat{Q}_{a,a-2}^{\eta\beta\alpha\gamma}\hat{g}_{\beta\alpha,\gamma}\hat{\mu}_{\eta}\,dl.
\end{align*}
Once we prove that $I_R\to 0$ as $R\to\infty$, the theorem follows.

Denote by $A_{R}$ the region bounded by $S_{R}^{+}$, $\hat{S}_{R}^{+}$ and $\{x^{n}=0\}$, and by $A'_{R}$ the region in $\{x^{n}=0\}$ bounded by $S_{R}^{n-2}$ and $\hat{S}_{R}^{n-2}$.
Then
\begin{align*}
\mathrm{I}_{2}
& =\int_{A_{R}}\hat{\partial}_{i}(\hat{g}_{jk,l}\hat{P}_{a}^{ijkl})\,dx+a\int_{A'_{R}}\mathscr{B}^{a-1}d\sigma+o(1)\\
& =\frac{1}{2}\int_{A_{R}}L_{a}dv_{g}+a\int_{A'_{R}}\mathscr{B}^{a-1}d\sigma_{g}+o(1)
\to 0\quad\text{as }R\to\infty.
\end{align*}
It remains to show that $\mathrm{I}_{1}\to 0$ as $R\to\infty$.  
Write
\[
\mathrm{I}_{1}=\mathrm{I}_{11}+\mathrm{I}_{12}+\mathrm{I}_{13},
\]
where
\[
\mathrm{I}_{11}=\int_{S_{R}^{+}}g_{jk,l}P_a^{ijkl}\nu_{i}\,d\sigma-\int_{S_{R}^{+}}\hat{g}_{jk,l}\hat{P}_{a}^{ijkl}\nu_{i}\,d\sigma,
\]
\[
\mathrm{I}_{12}=\int_{S_{R}^{n-2}}g_{n\beta}P_a^{n\beta n\alpha}\mu_{\alpha}\,dl-\int_{S_{R}^{n-2}}\hat{g}_{n\beta}\hat{P}_{a}^{n\beta n\alpha}\mu_{\alpha}\,dl,
\]
and
\[
\mathrm{I}_{13}=(a-1)\int_{S_{R}^{n-2}}Q_{a,a-2}^{\eta\beta\alpha\gamma}g_{\beta\alpha,\gamma}\mu_{\eta}\,dl-(a-1)\int_{S_{R}^{n-2}}\hat{Q}_{a,a-2}^{\eta\beta\alpha\gamma}\hat{g}_{\beta\alpha,\gamma}\mu_{\eta}\,dl.
\]

We rewrite
\[
\mathrm{I}_{11}=\int_{S_{R}^{+}}(g_{jk,l}-\hat{g}_{jk,l})P_a^{ijkl}\nu_{i}\,d\sigma+\int_{S_{R}^{+}}(P_a^{ijkl}-\hat{P}_{a}^{ijkl})\hat{g}_{jk,l}\nu_{i}\,d\sigma.
\]
By \eqref{eq:Rm difference}, \eqref{behavior of Rm}, and \eqref{eq:asymptotic behavior},
\[
\int_{S_{R}^{+}}(P_a^{ijkl}-\hat{P}_{a}^{ijkl})\hat{g}_{jk,l}\nu_{i}\,d\sigma\to 0
\quad\text{as }R\to\infty.
\]
When $\nu$ is on $S_{R}^{n-2}$, $\nu=\frac{x^{i}}{|x|}\frac{\partial}{\partial x^{i}}=\frac{x^{\alpha}}{|x'|}\frac{\partial}{\partial x^{\alpha}}=\mu$
and without loss of generality, we denote $\nu$ as $\frac{\partial}{\partial x_{n-1}}$
on $S_{R}^{n-2}$. 

By an argument similar to Ge--Wang--Wu \cite{GWW1} (see the identity for $I$ on p.~95 of \cite{GWW1}), using \eqref{eq:asy3}, \eqref{eq:antisymm}, and $|\partial\varphi|\le Cr^{-\tau}$, we obtain
\begin{align*}
\int_{S_{R}^{+}}(g_{jk,l}-\hat{g}_{jk,l})P_a^{ijkl}\nu_{i}\,d\sigma
= & \,\int_{S_{R}^{+}}\partial_{j}\bigl(P_a^{ijkl}\partial_{l}\varphi^{k}\nu_{i}\bigr)\,d\sigma+o(1)\\
= & -\int_{S_{R}^{n-2}}P_a^{n-1nkl}\partial_{l}\varphi^{k}\,dl+o(1),
\end{align*}
where we used $P_a^{ijkl}\partial_l\partial_k\varphi^j=0$ and
$P_a^{ijkl}\partial_j(\nu_i)=P_a^{ijkl}\bigl(\frac{\delta_{ij}}{|x|}-\frac{x_ix_j}{|x|^3}\bigr)=0$ in the first equality.

Next,
\begin{align*}
P_a^{n-1nkl}\partial_{l}\varphi^{k}
= & \,P_a^{n-1nkn}\partial_{n}\varphi^{k}+P_a^{n-1nk\alpha}\partial_{\alpha}\varphi^{k}\\
= & \,P_a^{n-1n\alpha n}\partial_{n}\varphi^{\alpha}+P_a^{n-1n\beta\alpha}\partial_{\alpha}\varphi^{\beta}+P_a^{n-1nn\alpha}\partial_{\alpha}\varphi^{n}.
\end{align*}
Therefore, since $\partial_\alpha \varphi^n=0$ on $\partial M$, we have
\begin{equation}\label{eq:I11}
\mathrm{I}_{11}=-\int_{S_{R}^{n-2}}P_a^{n-1n\alpha n}\partial_{n}\varphi^{\alpha}\,dl-\int_{S_{R}^{n-2}}P_a^{n-1n\beta\alpha}\partial_{\alpha}\varphi^{\beta}\,dl+o(1).
\end{equation}
Now from \eqref{eq:Rm difference} and the definition of $P_a$, together with $R_{ijkl}=O\bigl(R^{-\tau-2}\bigr)$, we obtain
\[
\hat{P}_{a}^{n\beta n\alpha}=P_a^{n\beta n\alpha}+O\bigl(R^{-(\tau+2)(a-2)-2-2\tau}\bigr).
\]
Together with \eqref{eq:asy2}, this yields
\begin{align*}
 & g_{n\beta}P_a^{n\beta n\alpha}\mu_{\alpha}-\hat{g}_{n\beta}\hat{P}_{a}^{n\beta n\alpha}\mu_{\alpha}\\
= & P_a^{n\beta n\alpha}\mu_{\alpha}(\partial_{n}\varphi^{\beta}+\partial_{\beta}\varphi^{n})+O\bigl(R^{-2\tau-(2+\tau)(a-1)}\bigr),
\end{align*}
and hence
\begin{align}\label{eq:I12}
\mathrm{I}_{12}
=\int_{S_{R}^{n-2}}P_a^{n\beta nn-1}\partial_{n}\varphi^{\beta}\,dl+O\bigl(R^{-2\tau-(2+\tau)(a-1)+n-2}\bigr),
\end{align}
where $O\bigl(R^{-2\tau-(2+\tau)(a-1)+n-2}\bigr)\to 0$ as $R\to\infty$.

We next estimate $\mathrm{I}_{13}$. From \eqref{eq:Rm difference} and \eqref{eq:the second form}, we have
\begin{align}\label{eq:I13}
 & Q_{a,a-2}^{\eta\beta\alpha\gamma}g_{\beta\alpha,\gamma}-\hat{Q}_{a,a-2}^{\eta\beta\alpha\gamma}\hat{g}_{\beta\alpha,\gamma}\\
= & \,\frac{1}{2^{a-2}}\sum_{i,j=1}^{n-1}\delta_{j_{2}j_{3}\cdots j_{2a-3}j_{2a-2}\alpha\gamma}^{i_{2}i_{3}\cdots i_{2a-3}i_{2a-2}\eta\beta}L_{i_{2}}^{j_{2}}R_{i_{3}i_{4}}{}^{j_{3}j_{4}}\cdots R_{i_{2a-3}i_{2a-2}}{}^{j_{2a-3}j_{2a-2}}g_{\beta\alpha,\gamma}\nonumber \\
 & -\frac{1}{2^{a-2}}\sum_{i,j=1}^{n-1}\delta_{j_{2}j_{3}\cdots j_{2a-3}j_{2a-2}\alpha\gamma}^{i_{2}i_{3}\cdots i_{2a-3}i_{2a-2}\eta\beta}\hat{L}_{i_{2}}^{j_{2}}R_{i_{3}i_{4}}{}^{j_{3}j_{4}}\cdots R_{i_{2a-3}i_{2a-2}}{}^{j_{2a-3}j_{2a-2}}\hat{g}_{\beta\alpha,\gamma}\nonumber \\
 & +O\bigl(R^{-(2+\tau)(a-3)-2(\tau+1)-2-2\tau}\bigr)\nonumber \\
= & \,\frac{1}{2^{a-2}}\sum_{i,j=1}^{n-1}\delta_{j_{2}j_{3}\cdots j_{2a-3}j_{2a-2}\alpha\gamma}^{i_{2}i_{3}\cdots i_{2a-3}i_{2a-2}\eta\beta}L_{i_{2}}^{j_{2}}R_{i_{3}i_{4}}{}^{j_{3}j_{4}}\cdots R_{i_{2a-3}i_{2a-2}}{}^{j_{2a-3}j_{2a-2}}g_{\beta\alpha,\gamma}\nonumber \\
 & -\frac{1}{2^{a-2}}\sum_{i,j=1}^{n-1}\delta_{j_{2}j_{3}\cdots j_{2a-3}j_{2a-2}\alpha\gamma}^{i_{2}i_{3}\cdots i_{2a-3}i_{2a-2}\eta\beta}(L_{i_{2}}^{j_{2}}-\partial_{j_{2}}\partial_{i_{2}}\varphi^{n})R_{i_{3}i_{4}}{}^{j_{3}j_{4}}\cdots R_{i_{2a-3}i_{2a-2}}{}^{j_{2a-3}j_{2a-2}}\hat{g}_{\beta\alpha,\gamma}\nonumber \\
 & +O\bigl(R^{-(2+\tau)(a-3)-2(\tau+1)-2-2\tau}\bigr)+O\bigl(R^{-(a+1)\tau-2a+2}\bigr)\nonumber \\
= & \,\frac{1}{2^{a-2}}\sum_{i,j=1}^{n-1}\delta_{j_{2}j_{3}\cdots j_{2a-3}j_{2a-2}\alpha\gamma}^{i_{2}i_{3}\cdots i_{2a-3}i_{2a-2}\eta\beta}L_{i_{2}}^{j_{2}}R_{i_{3}i_{4}}{}^{j_{3}j_{4}}\cdots R_{i_{2a-3}i_{2a-2}}{}^{j_{2a-3}j_{2a-2}}(\partial_{\gamma}\partial_{\beta}\varphi^{\alpha}+\partial_{\gamma}\partial_{\alpha}\varphi^{\beta})\nonumber \\
 & +O\bigl(R^{-(2+\tau)(a-3)-2(\tau+1)-2-2\tau}\bigr)+O\bigl(R^{-(a+1)\tau-2a+2}\bigr).\nonumber
\end{align}
Here the last equality follows from \eqref{eq:asy3}.

Collecting \eqref{eq:I11}, \eqref{eq:I12}, and \eqref{eq:I13}, we arrive at
\begin{align}
\mathrm{I}_{1}= & -\int_{S_{R}^{n-2}}P_a^{n-1n\beta\alpha}\partial_{\alpha}\varphi^{\beta}dl\label{eq:I1 final form}\\
 & +\frac{a-1}{2^{a-2}}\sum_{i,j=1}^{n-1}\int_{S_{R}^{n-2}}\delta_{j_{2}j_{3}\cdots j_{2a-3}j_{2a-2}\alpha\gamma}^{i_{2}i_{3}\cdots i_{2a-3}i_{2a-2}\eta\beta}L_{i_{2}}^{j_{2}}\cdots R_{i_{2a-3}i_{2a-2}}{}^{j_{2a-3}j_{2a-2}}(\partial_{\gamma}\partial_{\beta}\varphi^{\alpha}+\partial_{\gamma}\partial_{\alpha}\varphi^{\beta})\nonumber \\
= & -\int_{S_{R}^{n-2}}P_a^{\gamma n\beta\alpha}\partial_{\alpha}\varphi^{\beta}\mu_{\gamma}dl+(a-1)\int_{S_{R}^{n-2}} Q_{a,a-2}^{\eta \beta \alpha \gamma}\partial_{\beta }\partial_{\gamma}\varphi^{\alpha}\mu_\eta dl \nonumber 
 %& +\frac{a-1}{2^{a-2}}\sum_{i,j=1}^{n-1}\int_{S_{R}^{n-2}}\delta_{j_{2}j_{3}\cdots j_{2a-3}j_{2a-2}\alpha\gamma}^{i_{2}i_{3}\cdots i_{2a-3}i_{2a-2}\eta\beta}L_{i_{2}}^{j_{2}}\cdots R_{i_{2a-3}i_{2a-2}}{}^{j_{2a-3}j_{2a-2}}\partial_{\gamma}\partial_{\beta}\varphi^{\alpha}\mu_{\eta},\nonumber 
\end{align}
where the last equality holds due to the symmetry and anti-symmetry of $\alpha$ and $\gamma$ such that 
\[
\sum_{i,j=1}^{n-1}\delta_{j_{2}j_{3}\cdots j_{2a-3}j_{2a-2}\alpha\gamma}^{i_{2}i_{3}\cdots i_{2a-3}i_{2a-2}\eta\beta}L_{i_{2}}^{j_{2}}R_{i_{3}i_{4}}{}^{j_{3}j_{4}}\cdots R_{i_{2a-3}i_{2a-2}}{}^{j_{2a-3}j_{2a-2}}\partial_{\gamma}\partial_{\alpha}\varphi^{\beta}=0.
\]

\textbf{Claim.}
\begin{equation*}
-P_a^{\gamma n\beta\alpha}\partial_{\alpha}\varphi^{\beta}\mu_{\gamma}
=(a-1)\partial_{\gamma}\bigl(Q_{a,a-2}^{\eta\gamma\beta\alpha}\mu_{\eta}\partial_{\alpha}\varphi^{\beta}\bigr)
-(a-1)Q_{a,a-2}^{\eta\gamma\beta\alpha}\mu_{\eta}\partial_{\gamma}\partial_{\alpha}\varphi^{\beta}
+O\bigl(R^{-\tau(a+1)-2a+2}\bigr).
\end{equation*}
With this claim, the proof follows from \eqref{eq:I1 final form}.

We now prove the claim. Define, for $1\le p\le 2a-2$,
\[
\mathrm{J}_{p}:=\frac{1}{2^{a}}\sum_{j=1}^{n-1}\sum_{i=1}^{n-2}\delta_{j_{1}j_{2}\cdots n\cdots j_{2a-3}j_{2a-2}\beta\alpha}^{i_{1}i_{2}\cdots i_{p}\cdots i_{2a-3}i_{2a-2}n-1n}R_{i_{1}i_{2}}{}^{j_{1}j_{2}}\cdots R_{i_{p}i_{p+1}}^{\quad\quad nj_{p+1}}\cdots R_{i_{2a-3}i_{2a-2}}{}^{j_{2a-3}j_{2a-2}}.
\]
By symmetry, $\mathrm{J}_{p}=\mathrm{J}_{1}$ for all $p$.

We compute
\begin{align}\label{Claim boundary1}
P_a^{\gamma n\beta \alpha}\mu_{\gamma}
=&\,P_a^{n-1n\beta\alpha}\no\\
= & \,\mathrm{J}_{1}+\cdots+\mathrm{J}_{2a-2}+O\bigl(R^{-(\tau+2)(a-1)-\tau}\bigr)\no\\
=&\,(2a-2)\mathrm{J}_{1}+O\bigl(R^{-(\tau+2)(a-1)-\tau}\bigr),
\end{align}
where we recall
\[
P_a^{n-1n\beta\alpha}=\frac{1}{2^{a}}\sum_{i,j=1}^{n}\delta_{j_{1}j_{2}\cdots j_{2a-3}j_{2a-2}j_{2a-1}j_{2a}}^{i_{1}i_{2}\cdots i_{2a-3}i_{2a-2}n-1n}R_{i_{1}i_{2}}{}^{j_{1}j_{2}}\cdots R_{i_{2a-3}i_{2a-2}}{}^{j_{2a-3}j_{2a-2}}g^{j_{2a-1}\beta}g^{j_{2a}\alpha}.
\]
Since
\[
\partial_{i_{1}}R_{i_{3}i_{4}}{}^{j_{3}j_{4}}=\frac{1}{2}\frac{\partial}{\partial x^{i_{1}}}\bigg(\frac{\partial^{2}g_{i_{4}j_{3}}}{\partial x^{i_{3}}\partial x^{j_{4}}}+\frac{\partial^{2}g_{i_{3}j_{4}}}{\partial x^{i_{4}}\partial x^{j_{3}}}-\frac{\partial^{2}g_{i_{3}j_{3}}}{\partial x^{i_{4}}\partial x^{j_{4}}}-\frac{\partial^{2}g_{i_{4}j_{4}}}{\partial x^{i_{3}}\partial x^{j_{3}}}\bigg)+O\bigl(R^{-2\tau-3}\bigr),
\]
by symmetry and antisymmetry of indices (e.g., in $i_1,i_3$ and $i_1,i_4$) we obtain
\begin{equation}\label{vanishing term}
\sum_{i,j=1}^{n-1}\delta_{j_{2}\cdots j_{2a-2}\beta\alpha}^{i_{2}\cdots i_{2a-2}n-1i_{1}}L_{i_{2}}^{j_{2}}\partial_{i_{1}}\bigg(R_{i_{3}i_{4}}{}^{j_{3}j_{4}}\cdots R_{i_{2a-3}i_{2a-2}}{}^{j_{2a-3}j_{2a-2}}\bigg)\partial_{\alpha}\varphi^{\beta}=O\bigl(R^{-\tau(a+1)-2a+2}\bigr).
\end{equation}
Moreover,
\begin{align*}
-2^{a}\mathrm{J}_{1}\partial_{\alpha}\varphi^{\beta}
=&\sum_{i,j=1}^{n-1} -\delta_{nj_{2}\cdots j_{2a-3}j_{2a-2}\beta\alpha}^{i_{1}i_{2}\cdots i_{2a-3}i_{2a-2}n-1n}R_{i_{1}i_{2}}{}^{nj_{2}}R_{i_{3}i_{4}}{}^{j_{3}j_{4}}\cdots R_{i_{2a-3}i_{2a-2}}{}^{j_{2a-3}j_{2a-2}}\partial_{\alpha}\varphi^{\beta}\\
= &\sum_{i,j=1}^{n-1} -\delta_{nj_{2}\cdots j_{2a-3}j_{2a-2}\beta\alpha}^{i_{1}i_{2}\cdots i_{2a-3}i_{2a-2}n-1n}(L_{i_{2},i_{1}}^{j_{2}}-L_{i_{1},i_{2}}^{j_{2}})\cdots R_{i_{2a-3}i_{2a-2}}{}^{j_{2a-3}j_{2a-2}}\partial_{\alpha}\varphi^{\beta}\\
 & +O\bigl(R^{-(a+1)\tau-2a+2}\bigr)\\
= & \,\sum_{i,j=1}^{n-1}2\delta_{j_{2}\cdots j_{2a-3}j_{2a-2}\beta\alpha}^{i_{2}\cdots i_{2a-3}i_{2a-2}n-1i_{1}}L_{i_{2},i_{1}}^{j_{2}}\cdots R_{i_{2a-3}i_{2a-2}}{}^{j_{2a-3}j_{2a-2}}\partial_{\alpha}\varphi^{\beta}\\
 & +O\bigl(R^{-(a+1)\tau-2a+2}\bigr)\\
= & \,\sum_{i,j=1}^{n-1}2\delta_{j_{2}\cdots j_{2a-3}j_{2a-2}\beta\alpha}^{i_{2}\cdots i_{2a-3}i_{2a-2}n-1i_{1}}\partial_{i_{1}}\bigg(L_{i_{2}}^{j_{2}}R_{i_{3}i_{4}}{}^{j_{3}j_{4}}\cdots R_{i_{2a-3}i_{2a-2}}{}^{j_{2a-3}j_{2a-2}}\partial_{\alpha}\varphi^{\beta}\bigg)\\
 & -2\delta_{j_{2}\cdots j_{2a-3}j_{2a-2}\beta\alpha}^{i_{2}\cdots i_{2a-3}i_{2a-2}n-1i_{1}}L_{i_{2}}^{j_{2}}\partial_{i_{1}}\bigg(R_{i_{3}i_{4}}{}^{j_{3}j_{4}}\cdots R_{i_{2a-3}i_{2a-2}}{}^{j_{2a-3}j_{2a-2}}\bigg)\partial_{\alpha}\varphi^{\beta}\\
 & -2\delta_{j_{2}\cdots j_{2a-3}j_{2a-2}\beta\alpha}^{i_{2}\cdots i_{2a-3}i_{2a-2}n-1i_{1}}L_{i_{2}}^{j_{2}}R_{i_{3}i_{4}}{}^{j_{3}j_{4}}\cdots R_{i_{2a-3}i_{2a-2}}{}^{j_{2a-3}j_{2a-2}}\partial_{i_{1}}\partial_{\alpha}\varphi^{\beta}\\
 & +O\bigl(R^{-(a+1)\tau-2a+2}\bigr)\\
=&\,2^{a-1}\partial_{i_1}\bigl(Q_{a,a-2}^{n-1 i_1\beta \alpha}\partial_{\alpha} \varphi^{\beta}\bigr)-2^{a-1}Q_{a,a-2}^{n-1 \gamma \beta \alpha }\partial_{\gamma}\partial_{\alpha} \varphi^{\beta}+O\bigl(R^{-(a+1)\tau-2a+2}\bigr),
\end{align*}
where in the last equality we used \eqref{vanishing term}.
Together with \eqref{Claim boundary1}, this completes the proof of the claim and hence of the theorem.

\iffalse
By (\ref{eq:I1 final form}), we have 

\begin{align*}
\mathrm{I}_{1}= & -\frac{a-1}{2^{a-2}}\int_{S_{R}^{n-2}}\delta_{j_{2}\cdots j_{2a-3}j_{2a-2}\beta\alpha}^{i_{2}\cdots i_{2a-3}i_{2a-2}n-1\gamma}L_{i_{2},i_{1}}^{j_{2}}R_{i_{3}i_{4}}{}^{j_{3}j_{4}}\cdots R_{i_{2a-3}i_{2a-2}}{}^{j_{2a-3}j_{2a-2}}\partial_{\gamma}\partial_{\alpha}\varphi^{\beta}\\
 & +\frac{a-1}{2^{a-2}}\sum_{i,j=1}^{n-1}\int_{S_{R}^{n-2}}\delta_{j_{2}j_{3}\cdots j_{2a-3}j_{2a-2}\alpha\gamma}^{i_{2}i_{3}\cdots i_{2a-3}i_{2a-2}\eta\beta}L_{j_{2}i_{2}}R_{i_{3}i_{4}}{}^{j_{3}j_{4}}\cdots R_{i_{2a-3}i_{2a-2}}{}^{j_{2a-3}j_{2a-2}}\partial_{\gamma}\partial_{\beta}\varphi^{\alpha}\mu_{\eta}\\
 & +O(R^{-\tau-1-\tau}R^{-(\tau+2)(a-2)-\tau-1+n-2})\\
= & o(1),\quad\text{as }R\rightarrow\infty.
\end{align*}
Above all, we have $I_{R}\rightarrow0$ as $R\rightarrow\infty$.
The proof is completed.
\fi

\end{proof}
\begin{remark}
Inspired by Wang--Wu \cite{Wang-Wu}, we provide an alternative definition of the boundary mass.

Define, for $1\le\alpha,\beta\le n-1$,
\begin{align*}
m_{I,b}^{k}
= a(n,k)\lim_{R\to\infty}\bigg(-\int_{S_{R}^{+}}\mathcal{E}^{(k)}\Bigl(X,\frac{\partial}{\partial r}\Bigr)\,d\sigma+k\int_{S_{R}^{n-2}}\mathcal T_{k-1,\beta}^{\quad\,\,\,\alpha}x_{\alpha}\mu^{\beta}\,dl\bigg),
\end{align*}
where $X=r\frac{\partial}{\partial r}=x^{i}\frac{\partial}{\partial x^{i}}$, $a(n,k)=\frac{(n-2k-1)!}{2^{k}(n-1)!\omega_{n-1}}$, $d\sigma$ is the area element of $S_{R}^{+}$ in $\mathbb{R}^{n}$, and $\mu$ is the outward unit co-normal to $S_{R}^{n-2}=\partial S_{R,+}^{n-1}$ (oriented as the boundary of the bounded region $\Sigma_{r}\subset\partial M$). Here $dl$ denotes the area element of $S_{R}^{n-2}$ in $\mathbb{R}^{n-1}$. Moreover,
\[
\mathcal T_{k-1,\beta}^{\quad\,\,\,\alpha}=\frac{1}{2^{k-1}}\delta_{\beta j_{2}\cdots j_{2k-1}j_{2k}}^{\alpha i_{2}\cdots i_{2k-1}i_{2k}}L_{i_{2}}^{j_{2}}R_{i_{3}i_{4}}^{\quad\,j_{3}j_{4}}\cdots R_{i_{2k-1}i_{2k}}{}^{j_{2k-1}j_{2k}},
\]
\[
\mathcal{E}^{(k)i}{}_{j}=-\frac{1}{2^{k+1}}\delta_{jj_{1}j_{2}\cdots j_{2k-1}j_{2k}}^{ii_{1}i_{2}\cdots i_{2k-1}i_{2k}}R_{i_{1}i_{2}}{}^{j_{1}j_{2}}\cdots R_{i_{2k-1}i_{2k}}{}^{j_{2k-1}j_{2k}}.
\]
A similar definition of the ADM mass using conformal Killing fields on Euclidean space was proposed by Schoen (1988).
\end{remark}

\section{Graphical Case}

In this section we establish a positive mass theorem for graphs in $\mathbb{R}^{n+1}$ with non-compact boundary.

\medskip
\noindent\textbf{Algebraic notation.}
For later use (here and in Section~4), let $B$ and $C$ be $m\times m$ matrices. For integers $0\le r\le q\le m$, define the mixed elementary symmetric functions and the associated Newton transformations by
\begin{align*}
\sigma_{q, r}(C, B)
&=\frac{1}{q!}\,\delta^{i_1 \cdots i_q}_{j_1 \cdots j_q}
 C_{i_1}^{j_1} \cdots C_{i_r}^{j_r} B_{i_{r+1}}^{j_{r+1}} \cdots B_{i_q}^{j_q}, \\
T_{q, r}(C, B)_j^i
&=\frac{1}{q!}\,\delta^{i_1 \cdots i_q i}_{j_1 \cdots j_q j}
 C_{i_1}^{j_1} \cdots C_{i_r}^{j_r} B_{i_{r+1}}^{j_{r+1}} \cdots B_{i_q}^{j_q}.
\end{align*}
In particular,
\[
T_{k-1}(C)^i_j=\frac{1}{(k-1)!}\,\delta^{i_1\cdots i_{k-1}i}_{j_1\cdots j_{k-1}j}
C^{j_1}_{i_1}\cdots C^{j_{k-1}}_{i_{k-1}}.
\]

\medskip
\noindent\textbf{Geometric setting.}
Let $M=(x,f(x))\subset\mathbb{R}^{n+1}$ be the graph of a function $f$, equipped with the induced metric $g_{ij}=\delta_{ij}+f_i f_j$.
We write $Df=(\partial_{1}f,\dots,\partial_{n}f)$ and set $w=\sqrt{1+|Df|^{2}}$. Let $S$ denote the shape operator of $M$. Then
\begin{align}\label{expression of S in function}
S_{j}^{i}=g^{ik}S_{kj}=\Bigl(\frac{f_{i}}{w}\Bigr)_{j},
\end{align}
and the Gauss equation gives
\begin{equation}\label{RandL}
R_{ij}^{\quad kl}=S_{i}^{k}S_{j}^{l}-S_{i}^{l}S_{j}^{k},
\end{equation}
where $R_{ijkl}$ is the Riemannian curvature tensor of $(M,g)$.

By Reilly \cite{Reilly}, for even $k$ one has $\sigma_{k}(S)=\frac{1}{k!}L_{k/2}$. Moreover,
\begin{align}
\sigma_{k}(S)
& =\frac{1}{k!}\,\delta\!\left(\begin{array}{ccc}
i_{1} & \cdots & i_{k}\\
j_{1} & \cdots & j_{k}
\end{array}\right)S_{i_{1}}^{j_{1}}\cdots S_{i_{k}}^{j_{k}}\no\\
&=\frac{1}{k}(T_{k-1}(S))_{j}^{i}\Bigl(\frac{f_{j}}{w}\Bigr)_{i}\no\\
& =\frac{1}{k}\partial_{i}\Bigl(T_{k-1}(S){}_{j}^{i}\frac{f_{j}}{w}\Bigr).
\label{eq:div-structure of k-curvature}
\end{align}
In particular, $\sigma_{2}(S)=\frac{1}{2}R_g$, where $R_g$ is the scalar curvature of $(M,g)$.
Along $\partial\mathbb{R}^{n}_{+}$ we write $(S^T)_{\alpha}^{\beta}=S_{\alpha}^{\beta}$ for $\alpha,\beta\in\{1,\dots,n-1\}$, and we set $\Omega'=\overline{\Omega}\cap\partial\overline{\mathbb{R}_{+}^{n}}$.
Integrating \eqref{eq:div-structure of k-curvature} over $\mathbb{R}_{+}^{n}\setminus\Omega$ yields the following identity.
\begin{thm}\label{graph formula of mass}
Let $f:\mathbb{R}_{+}^{n}\setminus\Omega\to\mathbb{R}$ be an asymptotically flat function on $\mathbb{R}_{+}^{n}\setminus\Omega$ of class $C^{2}$ up to the boundary, satisfying Definition~\ref{def: assumption of f-1}.
Let $\left(\mathbb{R}_{+}^{n}\setminus\Omega,g\right)$ be the graph of $f$.
Assume that $f$ is constant on $\partial\Omega$, and let $\nu=\frac{Df}{|Df|}$ be the outward unit normal along $\partial\Omega$ (pointing out of $\Omega$).
Then
\begin{align}
 & \lim_{R\to\infty}\frac{1}{k}\int_{\partial B_{R}^{+}}(T_{k-1}(S))_{j}^{i}\frac{f_{j}}{w}\frac{x^{i}}{|x|}\,d\sigma+\frac{1}{k}\int_{\partial B_{R}^{n-1}}T_{k-2}(S^T)_{\alpha}^{i_{1}}\frac{f_{n}}{w}\frac{f_{\alpha}}{w}\frac{x^{i_{1}}}{|x|}\,dl\no\\
= & \,\int_{\mathbb{R}_{+}^{n}\setminus\Omega}\frac{1}{k!}L_{k/2}\,dx+\frac{1}{k}\int_{\partial\Omega\cap\mathbb{R}_{+}^{n}}\sigma_{k-1}(\kappa_{\partial\Omega})\left(\frac{|Df|}{w}\right)^{k}\,d\sigma+\int_{\partial\mathbb{R}_{+}^{n}\setminus\Omega}\sigma_{k-1}(S^T)\frac{f_{n}}{w}\,dx\no\\
 & +\frac{1}{k}\int_{\partial\Omega'}\sigma_{k-2}(\kappa_{\partial\Omega'})\left(\frac{|Df|}{w}\right)^{k}(\sin\theta)^{k-1}\cos\theta\,dl,\label{graph mass1}
\end{align}
where $\theta$ is the angle formed by the tangent plane of $\partial\Omega$ at $\partial\Omega'$ and $x^{n}=0$, and $\kappa_{\partial\Omega}$ and $\kappa_{\partial\Omega'}$ are the principal curvatures of $\partial\Omega\cap\mathbb{R}_{+}^{n}$ and $\partial\Omega'$, respectively.
\end{thm}

\begin{proof}
Now by \eqref{eq:div-structure of k-curvature},
\begin{align}
\int_{B_{R}^{+}\backslash\Omega}\frac{1}{k!}L_{k/2}dx\no %\int_{B_{R}^{+}\backslash\Omega}\sigma_{k}(S)dx\\
= & \frac{1}{k}\int_{S_{R}^{+}}(T_{k-1}(S))_{j}^{i}\frac{f_{j}}{w}\frac{x^{i}}{|x|}d\sigma-\frac{1}{k}\int_{B_{R}^{n-1}\backslash\Omega'}(T_{k-1}(S))_{j}^{n}\frac{f_{j}}{w}dx\no\\
 & +\frac{1}{k}\int_{\partial\Omega}(T_{k-1}(S))_{j}^{i}\frac{f_{j}}{w}(-\frac{f_{i}}{|Df|})d\sigma\no\\
= & \frac{1}{k}\int_{S_{R}^{+}}(T_{k-1}(S))_{j}^{i}\frac{f_{j}}{w}\frac{x^{i}}{|x|}d\sigma+\frac{1}{k}\int_{\partial\Omega}(T_{k-1}(S))_{j}^{i}\frac{f_{j}}{w}(-\frac{f_{i}}{|Df|})d\sigma\label{Lk integral}\\
 & -\frac{1}{k}\int_{B_{R}^{n-1}\backslash\Omega'}(T_{k-1}(S))_{n}^{n}\frac{f_{n}}{w}dx-\frac{1}{k}\int_{B_{R}^{n-1}\backslash\Omega'}(T_{k-1}(S))_{\alpha}^{n}\frac{f_{\alpha}}{w}dx,\no
\end{align}
where $\nu =\frac{Df}{|Df|}$ is the outward unit normal along $\partial\Omega$.

\textbf{Claim.}
\begin{equation}\label{tangential equality}
(T_{k-1}(S))_{\alpha}^{n}\frac{f_{\alpha}}{w}
=-\partial_{\beta }\left(T_{k-2}(S^T)_{\alpha}^{\beta}\frac{f_{n}f_{\alpha}}{1+|Df|^{2}}\right)+(k-1)\sigma_{k-1}(S^T)\frac{f_{n}}{w}.
\end{equation}
To show the claim, by the definition of $T_{k-1}$ and \eqref{expression of S in function},
\begin{align}
  T_{k-1}(S)_{\alpha}^{n}\frac{f_{\alpha}}{w}\no= & \frac{1}{(k-1)!}\delta^{
i_{1}  \cdots  i_{k-1}n}_
{j_{1}  \cdots  j_{k-1}\alpha}
S_{i_{1}}^{j_{1}}\cdots S_{i_{k-1}}^{j_{k-1}}\frac{f_{\alpha}}{w}\nonumber \\
= & \frac{1}{(k-2)!}\delta^{
i_{1}  \cdots  i_{k-1}n}_
{n  \cdots  j_{k-1}\alpha}S_{i_{1}}^{n}\cdots S_{i_{k-1}}^{j_{k-1}}\frac{f_{\alpha}}{w}\nonumber \\
= & \frac{1}{(k-2)!}\delta^{
i_{1}  \cdots  i_{k-1}n}_{
n  \cdots  j_{k-1}\alpha
}(\frac{f_{n}}{w})_{i_{1}}S_{i_{2}}^{j_{2}}\cdots S_{i_{k-1}}^{j_{k-1}}\frac{f_{\alpha}}{w}\nonumber \\
= & -\sum_{i_1=1}^{n-1}\partial_{i_1}\left(T_{k-2}(S^T)_{\alpha}^{i_1}\frac{f_{n}f_{\alpha}}{1+|Df|^{2}}\right)\label{eq:boundary tangential integral 1}\\
&-\frac{1}{(k-2)!}\delta^{
i_{1}  \cdots  i_{k-1}n}_
{n  \cdots  j_{k-1}\alpha}\frac{f_{n}}{w}\partial_{i_{1}}\big(S_{i_{2}}^{j_{2}}\cdots S_{i_{k-1}}^{j_{k-1}}\frac{f_{\alpha}}{w}\big).\nonumber 
\end{align}
By the Codazzi equation,
\begin{align*}
\partial_{i_{1}}\big(S_{i_{2}}^{j_{2}}\big) & =\nabla_{i_{1}}S_{i_{2}}^{j_{2}}+\Gamma_{i_{1}i_{2}}^{l}S_{l}^{j_{2}}-\Gamma_{i_{1}k}^{j_{2}}S_{i_{2}}^{k}\\
 & =\nabla_{i_{2}}S_{i_{1}}^{j_{2}}+\Gamma_{i_{1}i_{2}}^{l}S_{l}^{j_{2}}-\Gamma_{i_{1}k}^{j_{2}}S_{i_{2}}^{k},
\end{align*}
where $\Gamma_{i_{1}i_{2}}^{l}=\frac{f_{l}f_{i_{1}i_{2}}}{1+|Df|^{2}}$,
$\Gamma_{i_{1}k}^{j_{2}}=\frac{f_{j_{2}}f_{i_{1}k}}{1+|Df|^{2}}$
and $\nabla$ is the Levi--Civita connection of the induced metric $g$ on $M$.

By symmetry and antisymmetry in the indices $i_1,i_2$ and $j_2,\alpha$, we have
\[
\delta^{
i_{1}  \cdots  i_{k-1}n}_{
n  \cdots  j_{k-1}\alpha
}\frac{f_{n}}{w}\partial_{i_{1}}\big(S_{i_{2}}^{j_{2}}\big)\cdots S_{i_{k-1}}^{j_{k-1}}\frac{f_{\alpha}}{w}=0.
\]
Therefore, 
\begin{align}
  \frac{1}{(k-1)!}\delta^{
i_{1}  \cdots  i_{k-1}n}_{
n  \cdots  j_{k-1}\alpha
}\frac{f_{n}}{w}\partial_{i_{1}}\big(S_{i_{2}}^{j_{2}}\cdots S_{i_{k-1}}^{j_{k-1}}\frac{f_{\alpha}}{w}\big)\label{eq:third derivative vanishing}
= & \frac{1}{(k-1)!}\delta^{
i_{1}  \cdots  i_{k-1}n}_{
n  \cdots  j_{k-1}\alpha
}\frac{f_{n}}{w}S_{i_{2}}^{j_{2}}\cdots S_{i_{k-1}}^{j_{k-1}}\partial_{i_{1}}\big(\frac{f_{\alpha}}{w}\big) \\
= & -\frac{1}{(k-1)!}\sum_{i,j,\alpha=1}^{n-1}\delta^{
i_{1}i_{2}  \cdots i_{k-1}}_
{\alpha j_{2}  \cdots  j_{k-1}
}S_{i_{1}}^{\alpha}S_{i_{2}}^{j_{2}}\cdots S_{i_{k-1}}^{j_{k-1}}\frac{f_{n}}{w}\nonumber \\
= & -\sigma_{k-1}(S^T)\frac{f_{n}}{w}.\nonumber 
\end{align}
By \eqref{eq:boundary tangential integral 1} and \eqref{eq:third derivative vanishing}, we get the claim.
By \eqref{tangential equality},
\begin{align*}
  \int_{B_{R}^{n-1}\setminus\Omega'}(T_{k-1}(S))_{\alpha}^{n}\frac{f_{\alpha}}{w}\,dx
= & (k-1)\int_{B_{R}^{n-1}\setminus\Omega'}\sigma_{k-1}(S^T)\frac{f_{n}}{w}\,dx\\
 & -\int_{\partial B_{R}^{n-1}}T_{k-2}(S^T)_{\alpha}^{i_{1}}\frac{f_{n}}{w}\frac{f_{\alpha}}{w}\frac{x_{i_{1}}}{|x|}\,dl+\int_{\partial\Omega'}T_{k-2}(S^T)_{\alpha}^{i_{1}}\frac{f_{i_{1}}f_{\alpha}}{w|\overline{D}f|}\frac{f_{n}}{w}\,dl,
\end{align*}
where $\overline{D}f=(f_1,\cdots, f_{n-1})$ denotes the tangential gradient on $\partial\mathbb{R}_{+}^{n}$.

Together with \eqref{Lk integral}, it follows that
\begin{align*}
 & \int_{B_{R}^{+}\setminus\Omega}\frac{1}{k!}L_{k/2}\,dx\\
= & \,\frac{1}{k}\int_{S_{R}^{+}}(T_{k-1}(S))_{j}^{i}\frac{f_{j}}{w}\frac{x^{i}}{|x|}\,d\sigma-\frac{1}{k}\int_{\partial\Omega}(T_{k-1}(S))_{j}^{i}\frac{f_{j}}{w}\frac{f_{i}}{|Df|}\,d\sigma
 -\int_{B_{R}^{n-1}\setminus\Omega'}\sigma_{k-1}(S^T)\frac{f_{n}}{w}\,dx\\
 & +\frac{1}{k}\int_{\partial B_{R}^{n-1}}T_{k-2}(S^T)_{\alpha}^{i_{1}}\frac{f_{n}}{w}\frac{f_{\alpha}}{w}\frac{x^{i_{1}}}{|x|}\,dl-\frac{1}{k}\int_{\partial\Omega'}T_{k-2}(S^T)_{\alpha}^{i_{1}}\frac{f_{i_{1}}f_{\alpha}}{w|\overline{D}f|}\frac{f_{n}}{w}\,dl.
\end{align*}
Therefore,
\begin{align}
 & \lim_{R\to\infty}\frac{1}{k}\int_{\partial B_{R}^{+}}(T_{k-1}(S))_{j}^{i}\frac{f_{j}}{w}\frac{x^{i}}{|x|}\,d\sigma+\frac{1}{k}\int_{\partial B_{R}^{n-1}}T_{k-2}(S^T)_{\alpha}^{i_{1}}\frac{f_{n}}{w}\frac{f_{\alpha}}{w}\frac{x^{i_{1}}}{|x|}\,dl\nonumber \\
= & \,\int_{\mathbb{R}_{+}^{n}\setminus\Omega}\frac{1}{k!}L_{k/2}\,dx+\frac{1}{k}\int_{\partial\Omega}(T_{k-1}(S))_{j}^{i}\frac{f_{j}}{w}\Bigl(\frac{f_{i}}{|Df|}\Bigr)\,d\sigma+\int_{\partial\mathbb{R}_{+}^{n}\setminus\Omega'}\sigma_{k-1}(S^T)\frac{f_{n}}{w}\,dx\label{eq:formula 1}\\
 & +\frac{1}{k}\int_{\partial\Omega'}T_{k-2}(S^T)_{\alpha}^{i_{1}}\frac{f_{i_{1}}f_{\alpha}}{w|\overline{D}f|}\frac{f_{n}}{w}\,dl.\nonumber
\end{align}
Since $f$ is constant on $\partial\Omega\cap\mathbb{R}_{+}^{n},$
\begin{equation}
\int_{\partial\Omega}(T_{k-1}(S))_{j}^{i}\frac{f_{j}}{w}(\frac{f_{i}}{|Df|})d\sigma=\int_{\partial\Omega}\sigma_{k-1}(\kappa_{\partial\Omega})\bigg(\frac{|Df|}{w}\bigg)^{k}d\sigma.\label{eq:formula 1-1}
\end{equation}
To see \eqref{eq:formula 1-1}, let $\alpha_{1},\beta_{1}$ be tangential directions on the level set $\partial\Omega$, and let $n_{1}$ be the outward unit normal to $\partial\Omega$ pointing toward $x_{n}>0$. Rotating coordinates if necessary, we may assume that $|Df|=f_{n_{1}}$ and
\[
S_{\beta_{1}}^{\alpha_{1}}=\frac{f_{\alpha_{1}\beta_{1}}}{w}=h_{\alpha_{1}\beta_{1}}\frac{|Df|}{w}
\]
on $\partial\Omega$, where $h_{\alpha\beta}$ is the second fundamental form of $\partial\Omega$ in $\mathbb{R}_{+}^{n}$. Note that since $S\in \Gamma_k^+$, we have $\kappa_{\partial \Omega}\in \Gamma_{k-1}^+$.

A similar argument applies to \eqref{eq:formula 1-1}. Let $\alpha,\beta\in\{1,\dots,n-2\}$ be tangential directions along $\partial\Omega'$, and assume that $|\overline{D}f|=f_{n-1}$, where $\frac{\partial}{\partial x^{n-1}}$ is the outward unit normal to the level set $\partial\Omega'\subset\mathbb{R}^{n-1}$.
Then
$S_{\beta}^{\alpha}=h_{\alpha\beta}^{\partial\Omega'}\,\frac{|\overline{D}f|}{w}$ on $\partial\Omega'$, where $h_{\alpha\beta}^{\partial\Omega'}$ is the second fundamental form of $\partial\Omega'$ in $\partial\mathbb{R}_{+}^{n}$.
Consequently,
\begin{align*}
T_{k-2}(S^T)_{\alpha}^{i_{1}}\frac{f_{i_{1}}f_{\alpha}}{w|\overline{D}f|}\frac{f_{n}}{w}
& =T_{k-2}(S^T)_{n-1}^{n-1}\frac{|\overline{D}f|}{w}\frac{f_{n}}{w}\\
& =\sigma_{k-2}(\kappa_{\partial\Omega'})\left(\frac{|\overline{D}f|}{w}\right)^{k-1}\frac{f_{n}}{w}.
\end{align*}
Thus, since $f_{n}=|Df|\cos\theta$ and $|\overline{D}f|=|Df|\sin\theta$ on $\partial\Omega'$, we have
\begin{align}\label{eq:formula 1-1-1}
\int_{\partial\Omega'}T_{k-2}(S^T)_{\alpha}^{i_{1}}\frac{f_{i_{1}}f_{\alpha}}{w|\overline{D}f|}\frac{f_{n}}{w}\,dl
=\int_{\partial\Omega'}\sigma_{k-2}(\kappa_{\partial\Omega'})\left(\frac{|Df|}{w}\right)^{k}(\sin\theta)^{k-1}\cos\theta\,dl.
\end{align}
With \eqref{eq:formula 1}, \eqref{eq:formula 1-1}, and \eqref{eq:formula 1-1-1}, we complete the proof.
\end{proof}
We now check that if $k=2a$, then
\[
\lim_{R\to\infty}\bigg(\int_{\partial B_{R}^{+}}(T_{k-1}(S))_{j}^{i}\frac{f_{j}}{w}\frac{x^{i}}{|x|}\,d\sigma+\int_{\partial B_{R}^{n-1}}T_{k-2}(S^T)_{\alpha}^{i_{1}}\frac{f_{n}}{w}\frac{f_{\alpha}}{w}\frac{x^{i_{1}}}{|x|}\,dl\bigg)=\frac{2\mathfrak{m}_{a,B}(g)}{c(n,a)(2a-1)!},
\]
where we recall that
\[
\frac{\mathfrak{m}_{a,B}(g)}{c(n,a)}=\lim_{R\to\infty}\left(\int_{S_{R}^{+}}g_{jk,l}P_a^{ijkl}\nu_{i}\,d\sigma+\int_{S_{R}^{n-2}}g_{n\beta}P_a^{n\beta n\alpha}\mu_{\alpha}\,dl+(a-1)\int_{S_{R}^{n-2}}Q_{a,a-2}^{\eta\beta\alpha\gamma}g_{\beta\alpha,\gamma}\mu_{\eta}\,dl\right).
\]

Now we know that 
\begin{thm}
Assume as Theorem \ref{graph formula of mass}. 
\iffalse
Let $f:\mathbb{R}_{+}^{n}\backslash\Omega\rightarrow\mathbb{R}$ be
an asymptotically flat function over $\mathbb{R}_{+}^{n}\backslash\Omega$
of class $C^{2}$ up to boundary, with the assumption as Definition
\ref{def: assumption of f-1}. Sssume that $\kappa_f\in \Gamma_k^+$. Let $\left(\mathbb{R}_{+}^{n}\backslash\Omega,g\right)$
be the graph of $f$ and $f$ be constant on $\partial\Omega.$
\fi
We have
\begin{align*}
\frac{2\mathfrak{m}_{a,B}(g)}{(2a)!c(n,a)}= & \int_{\mathbb{R}_{+}^{n}\backslash\Omega}\frac{1}{(2a)!}L_{a}dx+\frac{1}{2a}\int_{\partial\Omega\cap\mathbb{R}_{+}^{n}}\sigma_{2a-1}(\kappa_{\partial\Omega})(\frac{|Df|}{w})^{2a}d\sigma\\
 & +\int_{\partial\mathbb{R}_{+}^{n}\backslash\Omega'}\sigma_{2a-1}(S^T)\frac{f_{n}}{w}d\sigma+\frac{1}{2a}\int_{\partial\Omega'}\sigma_{2a-2}(\kappa_{\partial\Omega'})(\frac{|Df|}{w})^{2a}(\sin\theta)^{2a-1}\cos\theta d\sigma,
\end{align*}
where $S^T$ is the tangential part of $S$ and $\theta$
is the angle formed by the tangent plane of $\partial\Omega$ at $\partial\Omega'$
and $x_{n}=0.$
\end{thm}

\begin{proof}
From the asymptotic behavior of $f$ and \eqref{RandL}, \eqref{expression of S in function}, we have
\begin{align}
 & \lim_{R\rightarrow\infty}\int_{S_{R}^{+}}g_{jk,l}P_a^{ijkl}\nu_{i}d\sigma\no\\
%= & \lim_{R\rightarrow\infty}\int_{S_{R}^{+}}(f_{jl}f_{k}+f_{j}f_{kl})\frac{1}{2^{a}}\delta_{j_{1}j_{2}\cdots j_{2a-3}j_{2a-2}j_{2a-1}j_{2a}}^{i_{1}i_{2}\cdots i_{2a-3}i_{2a-2}ij}R_{i_{1}i_{2}}{}^{j_{1}j_{2}}\cdots R_{i_{2a-3}i_{2a-2}}{}^{j_{2a-3}j_{2a-2}}g^{j_{2a-1}k}g^{j_{2a}l}\nu_{i}\nonumber \\
= & \lim_{R\rightarrow\infty}\int_{S_{R}^{+}}(f_{jl}f_{k}+f_{j}f_{kl})\frac{1}{2^{a}}\delta_{j_{1}j_{2}\cdots j_{2a-3}j_{2a-2}kl}^{i_{1}i_{2}\cdots i_{2a-3}i_{2a-2}ij}R_{i_{1}i_{2}}{}^{j_{1}j_{2}}\cdots R_{i_{2a-3}i_{2a-2}}{}^{j_{2a-3}j_{2a-2}}\nu_{i}\nonumber \\
= & \lim_{R\rightarrow\infty}\int_{S_{R}^{+}}f_{jl}f_{k}\frac{1}{2^{a}}\delta_{j_{1}j_{2}\cdots j_{2a-3}j_{2a-2}kl}^{i_{1}i_{2}\cdots i_{2a-3}i_{2a-2}ij}R_{i_{1}i_{2}}{}^{j_{1}j_{2}}\cdots R_{i_{2a-3}i_{2a-2}}{}^{j_{2a-3}j_{2a-2}}\nu_{i}\nonumber \\
= & \frac{1}{2}\lim_{R\rightarrow\infty}\int_{S_{R}^{+}}f_{jl}f_{k}\delta_{j_{1}j_{2}\cdots j_{2a-3}j_{2a-2}kl}^{i_{1}i_{2}\cdots i_{2a-3}i_{2a-2}ij}S_{i_{1}}^{j_{1}}\cdots S_{i_{2a-2}}^{j_{2a-2}}\nu_{i}d\sigma\nonumber \\
= & \frac{1}{2}\lim_{R\rightarrow\infty}\int_{S_{R}^{+}}f_{k}\delta_{j_{1}j_{2}\cdots j_{2a-3}j_{2a-2}kl}^{i_{1}i_{2}\cdots i_{2a-3}i_{2a-2}ij}S_{i_{1}}^{j_{1}}\cdots S_{i_{2a-2}}^{j_{2a-2}}S_{j}^{l}\nu_{i}d\sigma\nonumber \\[-6pt]
= & \frac{(2a-1)!}{2}\lim_{R\rightarrow\infty}\int_{S_{R}^{+}}f_{k}T_{2a-1}(S)_{k}^{i}\nu_{i}d\sigma,\label{eq:aysmptotic formula 1}
\end{align}
where the second equality holds due to the symmetry and anti-symmetry of $k$ and $l$ in $f_{kl}$ and Kronecker delta.

Also by \eqref{RandL},
\begin{align}&\lim_{R\rightarrow\infty}\int_{S_{R}^{n-2}}g_{n\beta}P_a^{n\beta n\alpha}\mu_{\alpha}dl\label{eq:asymptotic behavior 2}\\
= & \lim_{R\rightarrow\infty}\int_{S_{R}^{n-2}}f_{n}f_{\beta}P_a^{n\beta n\alpha}\mu_{\alpha}dl\nonumber \\
= & \lim_{R\rightarrow\infty}\int_{S_{R}^{n-2}}f_{n}f_{\beta}\mu^{\alpha}\frac{1}{2^{a}}\sum_{i,j=1}^{n}\delta_{j_{1}\cdots j_{2a-2}n\alpha}^{i_{1}\cdots i_{2a-2}n\beta}R_{i_{1}i_{2}}{}^{j_{1}j_{2}}\cdots R_{i_{2a-3}i_{2a-2}}{}^{j_{2a-3}j_{2a-2}}dl\nonumber \\
= & \lim_{R\rightarrow\infty}\int_{S_{R}^{n-2}}f_{n}f_{\beta}\mu^{\alpha}\frac{2^{a-1}}{2^{a}}\sum_{i,j=1}^{n-1}\delta_{j_{1}\cdots j_{2a-2}\alpha}^{i_{1}\cdots i_{2a-2}\beta}S_{i_{1}}^{j_{1}}\cdots S_{i_{2a-2}}^{j_{2a-2}}dl\nonumber \\
%= & \frac{(2a-2)!}{2}\lim_{R\rightarrow\infty}\int_{S_{R}^{n-2}}f_{n}f_{\beta}\mu_{\alpha}T_{2a-2}(S^T)_{\alpha}^{\beta}dl\nonumber \\
= & \frac{(2a-2)!}{2}\lim_{R\rightarrow\infty}\int_{S_{R}^{n-2}}f_{n}f_{\beta}\frac{x^{\alpha}}{|x'|}T_{2a-2}(S^T)_{\alpha}^{\beta}dl,\nonumber 
\end{align}
where 
\[
P_a^{n\beta n\alpha}=\frac{1}{2^{a}}\sum_{i,j=1}^{n}\delta_{j_{1}j_{2}\cdots j_{2a-3}j_{2a-2}j_{2a-1}j_{2a}}^{i_{1}i_{2}\cdots i_{2a-3}i_{2a-2}n\beta}R_{i_{1}i_{2}}{}^{j_{1}j_{2}}\cdots R_{i_{2a-3}i_{2a-2}}{}^{j_{2a-3}j_{2a-2}}g^{j_{2a-1}n}g^{j_{2a}\alpha}.
\] 
Also by using the asymptotic behavior of $f$ and \eqref{RandL}\eqref{expression of S in function}, it follows that 
\begin{align}
 & (a-1)\lim_{R\rightarrow\infty}\int_{S_{R}^{n-2}}Q_{a,a-2}^{\eta\beta\alpha\gamma}g_{\beta\alpha,\gamma}\mu_{\eta}dl\nonumber \\
= & \frac{(a-1)}{2^{a-2}}\lim_{R\rightarrow\infty}\int_{S_{R}^{n-2}}\sum_{i,j=1}^{n-1}\delta_{j_{2}\cdots j_{2a-3}j_{2a-2}\alpha\gamma}^{i_{2}\cdots i_{2a-3}i_{2a-2}\eta\beta}L_{i_{2}}^{j_{2}}\cdots R_{i_{2a-3}i_{2a-2}}{}^{j_{2a-3}j_{2a-2}}g_{\beta\alpha,\gamma}\mu^{\eta}dl\label{eq:asymptotic behavior 3}\\
= & \frac{a-1}{2^{a-2}}\lim_{R\rightarrow\infty}\int_{S_{R}^{n-2}}\sum_{i,j=1}^{n-1}\delta_{j_{2}\cdots j_{2a-3}j_{2a-2}\alpha\gamma}^{i_{2}\cdots i_{2a-3}i_{2a-2}\eta\beta}L_{i_{2}}^{j_{2}}\cdots R_{i_{2a-3}i_{2a-2}}{}^{j_{2a-3}j_{2a-2}}f_{\beta\gamma}f_{\alpha}\mu^{\eta}dl\nonumber \\
= & (a-1)\lim_{R\rightarrow\infty}\int_{S_{R}^{n-2}}(2a-2)!T_{2a-2}(S^T)_{\alpha}^{\eta}f_{\alpha}\frac{x^{\eta}}{|x'|}f_{n}dl,\nonumber 
\end{align}
where in the last equality, we have used $L_{\alpha\beta}=\frac{f_{\alpha\beta}f_{n}}{\sqrt{1+|\overline{D}f|^{2}}\sqrt{1+|Df|^{2}}}$ from Lemma~\ref{lem:second fundametal in graph} in the Appendix.

Therefore, from \eqref{eq:aysmptotic formula 1}, \eqref{eq:asymptotic behavior 2}, and \eqref{eq:asymptotic behavior 3}, we have
\begin{align*}
 & \lim_{R\to\infty}\bigg(\int_{S_{R}^{+}}g_{jk,l}P_a^{ijkl}\nu_{i}\,d\sigma+\int_{S_{R}^{n-2}}g_{n\beta}P_a^{n\beta n\alpha}\mu_{\alpha}\,dl+(a-1)\int_{S_{R}^{n-2}}Q_{a,a-2}^{\eta\beta\alpha\gamma}g_{\beta\alpha,\gamma}\mu^{\eta}\,dl\bigg)\\
= & \,\frac{(2a-1)!}{2}\lim_{R\to\infty}\left(\int_{S_{R}^{+}}f_{k}T_{2a-1}(S)_{k}^{i}\nu_{i}\,d\sigma+\int_{S_{R}^{n-2}}T_{2a-2}(S^T)_{\alpha}^{\eta}f_{\alpha}\frac{x^{\eta}}{|x'|}f_{n}\,dl\right).
\end{align*}
Combining with \eqref{graph mass1}, we conclude the theorem by the definition of  $\mathfrak{m}_{a,B}(g)$.
\end{proof}
We also obtain the following consequence.

\begin{thm}
Assume the hypotheses of Theorem~\ref{graph formula of mass} and that $|Df|\to\infty$ on $\partial\Omega$. Assume additionally that $\lambda(S)\in \overline{\Gamma}_{2a}^+$ and $f_n\ge 0$ on $\partial \mathbb{R}^n_+\backslash \Omega'$. Then
\begin{align*}
\frac{2\mathfrak{m}_{a,B}(g)}{c(n,a)(2a)!}
%= & \,\int_{\mathbb{R}_{+}^{n}\setminus\Omega}\frac{1}{(2a)!}L_{a}\,dx+\frac{1}{2a}\int_{\partial\Omega}\sigma_{2a-1}(\kappa_{\partial\Omega})\,d\sigma+\int_{\partial\mathbb{R}_{+}^{n}\setminus\Omega'}\sigma_{2a-1}(S^T)\frac{f_{n}}{w}\,dx\\
% & +\frac{1}{2a}\int_{\partial\Omega'}\sigma_{2a-2}(\kappa_{\partial\Omega'})(\sin\theta)^{2a-1}\cos\theta\,dl\\
\ge & \,\int_{\mathbb{R}_{+}^{n}\setminus\Omega}\frac{1}{(2a)!}L_{a}\,dx+\frac{1}{2a}\int_{\partial\Omega}\sigma_{2a-1}(\kappa_{\partial\Omega})\,d\sigma\\
  & +\frac{1}{2a}\int_{\partial\Omega'}\sigma_{2a-2}(\kappa_{\partial\Omega'})(\sin\theta)^{2a-1}\cos\theta\,dl.
\end{align*}
Here $\theta$ is the angle formed by the tangent plane of $\partial\Omega$ at $\partial\Omega'$ and $x_{n}=0$, and $\kappa_{\partial\Omega}$ and $\kappa_{\partial\Omega'}$ are the principal curvatures of $\partial\Omega$ and $\partial\Omega'$, respectively.
\end{thm}

Thus Theorem~\ref{thm:Main thm 2} and Corollary~\ref{cor:sharp lower bound for graph} follow from the above theorem.

\iffalse
First, one can choose the auxiliary function $G(x)=g(x)e^{\frac{x_n}{r}}\varphi(f)\log|\nabla f|$, where $g(x)=1-\frac{|x|^2}{r^2}$ and $\varphi(f)=1+\frac f M$, and $M=\sup_{B_R}|f|$. By maximum principle, one can obtain the boundness of the graph $f$.   Second, assuming that $|\nabla f|(0)\ge \delta>0$ for proper $\delta$,  one can choose $G(x)=g(x) \varphi(f) |\nabla f|$, where $g(x)=1-|x|^2 / r^2, \varphi(f)=$ $(1-f / M)^{-\alpha}, M=M_r=4 \sup \left\{|f(x)|, x \in B_r(0)\right\}$, and $\alpha \in(0,1)$, to conclude $\nabla f(0)=0$.
\fi

\section{Mass in locally conformally flat manifolds}

Let $(M^{n},g)$ be an $n$-dimensional Riemannian manifold with boundary $\partial M$. Near a boundary point we choose local coordinates $\{x^{i}\}_{i=1}^{n}$ such that $x^{n}$ is a defining function for $\partial M$ and $\{x^{\alpha}\}_{\alpha=1}^{n-1}$ are local coordinates on $\partial M$. In Fermi coordinates the metric takes the form
\[
g=(dx^{n})^{2}+g_{\alpha\beta}\,dx^{\alpha}dx^{\beta}.
\]

Throughout this section we assume that $(M,g)$ is locally conformally flat. Then the Riemann curvature tensor is determined by the Schouten tensor $A=A_g$ via
\begin{equation}\label{RandSchouten}
R_{ijkl}=A_{ik}g_{jl}+A_{jl}g_{ik}-A_{il}g_{jk}-A_{jk}g_{il}.
\end{equation}

Our first goal is to rewrite the bulk--boundary functional
$\int_{M}L_{a}\,dv_{g}+2a\int_{\partial M}\mathscr{B}^{a-1}\,d\sigma_{g}$
in terms of elementary symmetric functions of $A_g$ and a natural mixed boundary term. We then relate the boundary GBC mass to these conformal-geometric quantities, yielding in particular nonnegativity in the conformally flat setting.

We will frequently use the following contraction identity for generalized Kronecker symbols: if $1\le i,j\le m$, then
\begin{equation}\label{basci alge}
\sum_{s=1}^{m}\delta_{j_{1}\ldots j_{k}s_{1}\ldots s_{r}}^{i_{1}\ldots i_{k}s_{1}\ldots s_{r}}=\frac{(m-k)!}{(m-k-r)!}\delta_{j_{1}\ldots j_{k}}^{i_{1}\ldots i_{k}}.
\end{equation}

\begin{proposition}
Let $(M, g)$ be a locally conformally flat manifold. Assume that $2a<n$. Then
\begin{align}
&\int_{M}L_{a}dv_{g}+2a\int_{\partial M}\mathscr{B}^{a-1}d\sigma_{g}\no\\
=&\frac{2^{a}a!(n-a)!}{(n-2a)!}\left(\int_{M}\sigma_{a}(A_{g})dv_{g}+\frac{a}{n-a}\int_{\partial M}\sigma_{a,a-1}(A^{T}_g,  L_g)d\sigma_{g}\right),
\end{align}
where $A^T_{g,\alpha\beta}=A_{\alpha\beta}$ is the tangential part of the Schouten tensor $A_g$ on $\partial M$, and $L_g$ is the second fundamental form of $\partial M\subset M$. %and $$\sigma_{a,a-1}(A^{T},  L)=\frac{1}{a!}\sum_{\alpha,\beta=1}^{n-1}\binom{\alpha_{1}\cdots\alpha_{a}}{\beta_{1}\cdots\beta_{a}}A_{\alpha_{1}}^{\beta_{1}}\cdots A_{\alpha_{a-1}}^{\beta_{a-1}}L_{\alpha_{a}}^{\beta_{a}}.$$
\end{proposition}

\begin{proof}

By \eqref{RandSchouten}, 
\begin{align*}
 & \sum_{\alpha,\beta=1}^{n-1}\delta^{\alpha_{1}\cdots\alpha_{2a-1}}_{\beta_{1}\cdots\beta_{2a-1}}R_{\alpha_{1}\alpha_{2}}^{\quad\quad\beta_{1}\beta_{2}}\cdots R_{\alpha_{2a-3}\alpha_{2a-2}}^{\quad\quad\,\quad\beta_{2a-3}\beta_{2a-2}}L_{\alpha_{2a-1}}^{\beta_{2a-1}}\\
= & 2^{a-1}\sum_{\alpha,\beta=1}^{n-1}\delta^{\alpha_{1}\cdots\alpha_{2a-2}\alpha_{2a-1}}_{\beta_{1}\cdots\beta_{2a-2}\beta_{2a-1}}\left(A_{\alpha_{1}}^{\beta_{1}}g_{\alpha_{2}}^{\beta_{2}}+A_{\alpha_{2}}^{\beta_{2}}g_{\alpha_{1}}^{\beta_{1}}\right)\cdots\left(A_{\alpha_{2a-3}}^{\beta_{2a-3}}g_{\alpha_{2a-2}}^{\beta_{2a-2}}+A_{\alpha_{2a-2}}^{\beta_{2a-2}}g_{\alpha_{2a-3}}^{\beta_{2a-3}}\right)L_{\alpha_{2a-1}}^{\beta_{2a-1}}\\
= & 2^{2(a-1)}\sum_{\alpha,\beta=1}^{n-1}\delta^{\alpha_{1}\cdots\alpha_{2a-2}\alpha_{2a-1}}_{\beta_{1}\cdots\beta_{2a-2}\beta_{2a-1}}A_{\alpha_{1}}^{\beta_{1}}g_{\alpha_{2}}^{\beta_{2}}A_{\alpha_{3}}^{\beta_{3}}g_{\alpha_{4}}^{\beta_{4}}\cdots A_{\alpha_{2a-3}}^{\beta_{2a-3}}g_{\alpha_{2a-2}}^{\beta_{2a-2}}L_{\alpha_{2a-1}}^{\beta_{2a-1}}\\
%= & 2^{2(a-1)}\sum_{\alpha,\beta=1}^{n-1}\binom{\alpha_{1}\cdots\alpha_{2a-2}\alpha_{2a-1}}{\beta_{1}\cdots\beta_{2a-2}\beta_{2a-1}}A_{\alpha_{1}}^{\beta_{1}}g_{\alpha_{2}}^{\beta_{2}}A_{\alpha_{3}}^{\beta_{3}}g_{\alpha_{4}}^{\beta_{4}}\cdots A_{\alpha_{2a-2}}^{\beta_{2a-2}}g_{\alpha_{2a-3}}^{\beta_{2a-3}}L_{\alpha_{2a-1}}^{\beta_{2a-1}}\\
= & 2^{2(a-1)}\sum_{\alpha,\beta=1}^{n-1}\delta^{\alpha_{1}\cdots\alpha_{a-1}\alpha_{a}\cdots\alpha_{2a-1}}_{\alpha{}_{1}\cdots\alpha_{a-1}\beta_{a}\cdots\beta_{2a-1}}A_{\alpha_{a}}^{\beta_{a}}\cdots A_{\alpha_{2a-2}}^{\beta_{2a-2}}L_{\alpha_{2a-1}}^{\beta_{2a-1}}\\
= & 2^{2(a-1)}\frac{(n-1-a)!}{(n-1-(2a-1))!}\sum_{\alpha,\beta=1}^{n-1}\delta^{\alpha_{a}\cdots\alpha_{2a-1}}_{\beta_{a}\cdots\beta_{2a-1}}A_{\alpha_{a}}^{\beta_{a}}\cdots A_{\alpha_{2a-2}}^{\beta_{2a-2}}L_{\alpha_{2a-1}}^{\beta_{2a-1}}\\
%= & 2^{2(a-1)}\frac{(n-1-a)!}{(n-2a)!}\sum_{\alpha,\beta=1}^{n-1}\binom{\alpha_{a}\cdots\alpha_{2a-1}}{\beta_{a}\cdots\beta_{2a-1}}A_{\alpha_{a}}^{\beta_{a}}\cdots A_{\alpha_{2a-2}}^{\beta_{2a-2}}L_{\alpha_{2a-1}}^{\beta_{2a-1}}\\
= & 2^{2(a-1)}\frac{(n-1-a)!a!}{(n-2a)!}\sigma_{a,a-1}(A^{T},L),
\end{align*}
where the second-to-last equality follows from \eqref{basci alge}.

Thus,
\begin{align}\label{expressionof B in LCF}
\mathscr{B}^{a-1} & =2^{a-1}\frac{(n-1-a)!a!}{(n-2a)!}\sigma_{a,a-1}(A^{T},L).
\end{align}

By Proposition 2.2 in \cite{GWW2}, we know that
\[
L_{a}=2^{a}a!\frac{(n-a)!}{(n-2a)!}\sigma_{a}(A_{g}),
\]
and then, together with \eqref{expressionof B in LCF}, it follows that
\begin{align*}
 & \int_{M}L_{a}dv_{g}+2a\int_{\partial M}\mathscr{B}^{a-1}d\sigma_{g}\\
%= & \int_{M}2^{a}a!\frac{(n-a)!}{(n-2a)!}\sigma_{a}(A_{g})dv_{g}+2^{a}a\int_{\Sigma}\frac{(n-1-a)!a!}{(n-2a)!}\sigma_{a,a-1}(A^{T},L)d\sigma_{g}\\
= & \frac{2^{a}a!(n-a)!}{(n-2a)!}\left(\int_{M}\sigma_{a}(A_{g})dv_{g}+\frac{a}{n-a}\int_{\partial M}\sigma_{a,a-1}(A^{T},L)d\sigma_{g}\right).
\end{align*}
\end{proof}
For $n\ge 2a$, S.~Chen \cite{Chen} introduced a natural boundary curvature
\[
\mathcal{B}^a=\sum_{i=0}^{a-1} C_1(n,a,i)\,\sigma_{2a-i-1,\,i}\bigl(A_g^T,L_g\bigr),
\]
where
\[
C_1(n,a,i)=\frac{(2a-i-1)!(n-2a+i)!}{(n-a)!(2a-2i-1)!!\,i!}
\]
and $\sigma_{2a-i-1,\,i}(A_g^T,L_g)$ is the mixed curvature from Section~3. The quantity $\mathcal{B}^a$ naturally accompanies the functional
\[
\int_{M}\sigma_{a}(A_g)\,dv_{g}+\int_{\partial M}\mathcal{B}^a\,d\sigma_{g}.
\]
We refer to \cite{CC, CMW, CW18,CW2020, W, CLL2} for recent developments of $\mathcal{B}^a$ in the positive cone $\Gamma_a^+$. Note that $C_1(n,a,a-1)=\frac{a}{n-a}$, so the term $\frac{a}{n-a}\sigma_{a,a-1}(A^{T},L)$ appearing above is one of the summands in $\mathcal{B}^{a}$.

When $g=e^{-2u}g_{\mathbb{E}}$, along $\partial\mathbb{R}^{n}_{+}$ we have
\[
A^T_{\alpha\beta}=u_{\alpha\beta}+u_{\alpha}u_{\beta}-\frac{|Du|^2}{2}\,\delta_{\alpha\beta},
\qquad
L_{\alpha\beta}=e^{-u}u_n\,\delta_{\alpha\beta},
\qquad
h_{g}=e^u u_n.
\] 

\begin{lem}
Let $g=e^{-2u}g_{\mathbb{E}}$ locally. Then, in $M$,
\begin{equation}\label{P^sttm}
P_a^{sttm}=-2^{a-2}\frac{(n-a)!}{(n-2a)!}(a-1)!T_{a-1}(A)^{sm}e^{2u},
\end{equation}
and on $\partial M$,
\begin{equation}\label{P^sttm_L}
Q_{a,a-2}^{\eta\alpha\alpha\gamma}=-2^{a-2}\frac{(n-1-a)!}{(n-2a)!}(a-1)!T_{a-1,a-2}(A^{T},L)^{\eta\gamma}e^{2u},
\end{equation}
where
\[
T_{a-1,a-2}(A^{T},L)^{\eta}_\gamma=\frac{1}{(a-1)!}\sum_{i,j=1}^{n-1}\delta_{j_{1}\cdots j_{a-2}j\gamma}^{i_{1}\cdots i_{a-2}i\eta}A_{i_{1}}^{j_{1}}\cdots A_{i_{a-2}}^{j_{a-2}}L_{i}^{j}.
\]
\end{lem}

\begin{proof}
From the definition of $P_a^{sttm}$, with \eqref{RandSchouten} and \eqref{basci alge}, we obtain
\begin{align*}
P_a^{sttm} & =\frac{1}{2^{a}}\delta_{j_{1}j_{2}\cdots j_{2a-3}j_{2a-2}j_{2a-1}j_{2a}}^{i_{1}i_{2}\cdots i_{2a-3}i_{2a-2}st}R_{i_{1}i_{2}}{}^{j_{1}j_{2}}\cdots R_{i_{2a-3}i_{2a-2}}{}^{j_{2a-3}j_{2a-2}}g^{j_{2a-1}t}g^{j_{2a}m}\\
 & =\frac{1}{2^{a}}\delta_{j_{1}j_{2}\cdots j_{2a-3}j_{2a-2}tj_{2a}}^{i_{1}i_{2}\cdots i_{2a-3}i_{2a-2}st}R_{i_{1}i_{2}}{}^{j_{1}j_{2}}\cdots R_{i_{2a-3}i_{2a-2}}{}^{j_{2a-3}j_{2a-2}}g^{j_{2a}m}e^{2u}\\
 & =-\frac{1}{2^{a}}\delta_{j_{1}j_{2}\cdots j_{2a-3}j_{2a-2}j_{2a}t}^{i_{1}i_{2}\cdots i_{2a-3}i_{2a-2}st}R_{i_{1}i_{2}}{}^{j_{1}j_{2}}\cdots R_{i_{2a-3}i_{2a-2}}{}^{j_{2a-3}j_{2a-2}}g^{j_{2a}m}e^{2u}\\
 & =-\frac{1}{2^{a}}\frac{(n-2a+1)!}{(n-2a)!}\delta_{j_{1}j_{2}\cdots j_{2a-3}j_{2a-2}m}^{i_{1}i_{2}\cdots i_{2a-3}i_{2a-2}s}R_{i_{1}i_{2}}{}^{j_{1}j_{2}}\cdots R_{i_{2a-3}i_{2a-2}}{}^{j_{2a-3}j_{2a-2}}e^{4u}\\
 & =-\frac{2^{a-1}}{2^{a}}\frac{(n-2a+1)!}{(n-2a)!}\delta_{j_{1}j_{2}\cdots j_{2a-3}j_{2a-2}m}^{i_{1}i_{2}\cdots i_{2a-3}i_{2a-2}s}(A_{i_{1}}^{j_{1}}g_{i_{2}}^{j_{2}}+A_{i_{2}}^{j_{2}}g_{i_{1}}^{j_{1}})\cdots(A_{i_{2a-3}}^{j_{2a-3}}g_{i_{2a-2}}^{j_{2a-2}}+A_{i_{2a-2}}^{j_{2a-2}}g_{i_{2a-3}}^{j_{2a-3}})e^{4u}\\
 & =-\frac{2^{2(a-1)}}{2^{a}}\frac{(n-2a+1)!}{(n-2a)!}\delta_{i_{1}\cdots i_{a-1}j_{a}\cdots j_{2a-2}m}^{i_{1}\cdots i_{a-1}i_{a}\cdots i_{2a-2}s}A_{i_{a}}^{j_{a}}\cdots A_{i_{2a-2}}^{j_{2a-2}}e^{4u}\\
 & =-\frac{2^{a}}{4}\frac{(n-2a+1)!}{(n-2a)!}\frac{(n-a)!}{(n-2a+1)!}\delta_{j_{a}\cdots j_{2a-2}m}^{i_{a}\cdots i_{2a-2}s}A_{i_{a}}^{j_{a}}\cdots A_{i_{2a-2}}^{j_{2a-2}}e^{4u}\\
 & =-\frac{2^{a}}{4}\frac{(n-a)!}{(n-2a)!}(a-1)!T_{a-1}(A)_{m}^{s}e^{4u}\\
 & =-2^{a-2}\frac{(n-a)!}{(n-2a)!}(a-1)!T_{a-1}(A)^{sm}e^{2u}.
\end{align*}
Similarly, by \eqref{basci alge}, we have
\begin{align*}
Q_{a,a-2}^{\eta\alpha\alpha\gamma} & =\frac{1}{2^{a-2}}\sum_{i,j=1}^{n-1}\delta_{jj_{1}\cdots j_{2a-5}j_{2a-4}j_{2a-3}j_{2a-2}}^{ii_{1}\cdots i_{2a-5}i_{2a-4}\eta\alpha}L_{i}^{j}R_{i_{1}i_{2}}{}^{j_{1}j_{2}}\cdots R_{i_{2a-5}i_{2a-4}}{}^{j_{2a-5}j_{2a-4}}g^{j_{2a-3}\alpha}g^{j_{2a-2}\gamma}\\
 & =\frac{1}{2^{a-2}}\sum_{i,j=1}^{n-1}\delta_{jj_{1}\cdots j_{2a-5}j_{2a-4}\alpha\gamma}^{ii_{1}\cdots i_{2a-5}i_{2a-4}\eta\alpha}L_{i}^{j}R_{i_{1}i_{2}}{}^{j_{1}j_{2}}\cdots R_{i_{2a-5}i_{2a-4}}{}^{j_{2a-5}j_{2a-4}}e^{4u}\\
 & =\sum_{i,j=1}^{n-1}\delta_{jj_{1}\cdots j_{2a-5}j_{2a-4}\alpha\gamma}^{ii_{1}\cdots i_{2a-5}i_{2a-4}\eta\alpha}L_{i}^{j}(A_{i_{1}}^{j_{1}}g_{i_{2}}^{j_{2}}+A_{i_{2}}^{j_{2}}g_{i_{1}}^{j_{1}})\cdots(A_{i_{2a-5}}^{j_{2a-5}}g_{i_{2a-4}}^{j_{2a-4}}+A_{i_{2a-4}}^{j_{2a-4}}g_{i_{2a-5}}^{j_{2a-5}})e^{4u}\\
 & =-2^{a-2}\sum_{i,j=1}^{n-1}\delta_{ji_{1}\cdots i_{a-2}j_{a-1}\cdots j_{2a-5}j_{2a-4}\gamma\alpha}^{ii_{1}\cdots i_{a-2}i_{a-1}\cdots i_{2a-5}i_{2a-4}\eta\alpha}L_{i}^{j}A_{i_{a-1}}^{j_{a-1}}\cdots A_{i_{2a-4}}^{j_{2a-4}}e^{4u}\\
 & =-2^{a-2}\frac{(n-1-a)!}{(n-2a)!}\sum_{i,j=1}^{n-1}\delta_{jj_{a-1}\cdots j_{2a-5}j_{2a-4}\gamma}^{ii_{a-1}\cdots i_{2a-5}i_{2a-4}\eta}L_{i}^{j}A_{i_{a-1}}^{j_{a-1}}\cdots A_{i_{2a-4}}^{j_{2a-4}}e^{4u}\\
% & =-2^{a-2}\frac{(n-1-a)!}{(n-2a)!}(a-1)!T_{a-1,a-2}(A^{T},L)_{\gamma}^{\eta}e^{4u}\\
 & =-2^{a-2}\frac{(n-1-a)!}{(n-2a)!}(a-1)!T_{a-1,a-2}(A^{T},L)^{\eta\gamma}e^{2u}.
\end{align*}
\end{proof}
Recall the definition of $\mathfrak{m}_{a,B}(g)$. To prove that $\mathfrak{m}_{a,B}(g)$ is nonnegative on $(\mathbb{R}^n_+, e^{-2u}g_{\mathbb{E}})$, it suffices to show
\begin{align}
\lim_{R\to\infty}c(n,a)\bigg(\int_{S_{R}^{+}}g_{jk,l}P_a^{ijkl}\nu_{i}\,d\sigma+\int_{S_{R}^{n-2}}g_{n\beta}P_a^{n\beta n\alpha}\mu_{\alpha}\,dl
+(a-1)\int_{S_{R}^{n-2}}Q_{a,a-2}^{\eta\beta\alpha\gamma}g_{\beta\alpha,\gamma}\mu_{\eta}\,dl\bigg)\ge 0.\nonumber
\end{align}

We next compute the boundary integrals
\[
\int_{S_{R}^{+}}g_{jk,l}P_a^{ijkl}\nu_{i}\,d\sigma+\int_{S_{R}^{n-2}}g_{n\beta}P_a^{n\beta n\alpha}\mu_{\alpha}\,dl+(a-1)\int_{S_{R}^{n-2}}Q_{a,a-2}^{\eta\beta\alpha\gamma}g_{\beta\alpha,\gamma}\mu_{\eta}\,dl
\]
in the conformally flat setting.

\begin{proposition}\label{asym behavior}
On AF CF manifolds $(M, g)$ of order $\tau$, we have
\begin{align}
&\lim_{R\to\infty}\int_{S_{R}^{+}}g_{jk,l}P_a^{ijkl}\nu_{i}\,d\sigma
=2^{a-1}(a-1)!\frac{(n-a)!}{(n-2a)!}\lim_{R\to\infty}\int_{S_{R}^{+}}T_{a-1}(D^{2}u)^{il}u_{l}\nu_{i}\,d\sigma,
\end{align}
and
\begin{align}
&\lim_{R\to\infty}\int_{S_{R}^{n-2}}Q_{a,a-2}^{\eta\beta\alpha\gamma}g_{\beta\alpha,\gamma}\mu_{\eta}\,dl
=2^{a-1}(a-2)!\frac{(n-a)!}{(n-2a)!}\lim_{R\to\infty}\int_{S_{R}^{n-2}}u_{n}T_{a-2}(\overline{D}^{2}u)_{\gamma}^{\eta}\mu_{\eta}u^{\gamma}\,dl,
\end{align}
where $(\overline{D}^{2}u)_{\alpha \beta}=u_{\alpha \beta}=\frac{\partial ^2 u}{\partial x_\alpha \partial x_\beta}$.
\end{proposition}

\begin{proof}

Since $g=e^{-2u}g_{\mathbb{E}}$, by \eqref{P^sttm} we have
\begin{align*}
g_{jk,l}P_a^{ijkl}\nu_{i}
&=-2e^{-2u}P_a^{ijjl}u_{l}\nu_{i}\\
&=2^{a-1}\frac{(n-a)!}{(n-2a)!}(a-1)!T_{a-1}(A)^{il}u_{l}\nu_{i},
\end{align*}
and hence
\begin{align}
&\lim_{R\to\infty}\int_{S_{R}^{+}}g_{jk,l}P_a^{ijkl}\nu_{i}\,d\sigma\no\\
& =2^{a-1}(a-1)!\frac{(n-a)!}{(n-2a)!}\lim_{R\to\infty}\int_{S_{R}^{+}}T_{a-1}(A)^{il}u_{l}\nu_{i}\,d\sigma\no\\
 & =2^{a-1}(a-1)!\frac{(n-a)!}{(n-2a)!}\lim_{R\to\infty}\int_{S_{R}^{+}}T_{a-1}(D^{2}u)^{il}u_{l}\nu_{i}\,d\sigma,\label{expression of inter contribution}
\end{align}
where the last equality holds because $A=D^2 u+d u \otimes d u-\frac{|D u|^2}{2}\mathbb{I}$ and by the decay assumptions on $u$ at infinity.

By \eqref{P^sttm_L},
\begin{align*}
Q_{a,a-2}^{\eta\beta\alpha\gamma}g_{\beta\alpha,\gamma}\mu_{\eta}
&=-2Q_{a,a-2}^{\eta\alpha\alpha\gamma}\mu_{\eta}u_{\gamma}e^{-2u}\\
&=2^{a-1}\frac{(n-1-a)!}{(n-2a)!}(a-1)!T_{a-1,a-2}(A^{T},L)^{\eta\gamma}\mu_{\eta}u_{\gamma}.
\end{align*}

Then, on $\partial\mathbb{R}_{+}^{n}$, it holds that
\begin{align}
\lim_{R\to\infty}\int_{S_{R}^{n-2}}Q_{a,a-2}^{\eta\beta\alpha\gamma}g_{\beta\alpha,\gamma}\mu_{\eta}\,dl\no
 & =2^{a-1}\frac{(n-1-a)!}{(n-2a)!}(a-1)!\lim_{R\to\infty}\int_{S_{R}^{n-2}}T_{a-1,a-2}(A^{T},L)^{\eta\gamma}\mu_{\eta}u_{\gamma}\,dl\no\\
 & =2^{a-1}\frac{(n-1-a)!}{(n-2a)!}(a-1)!\lim_{R\to\infty}\int_{S_{R}^{n-2}}T_{a-1,a-2}(\overline D^{2}u,u_{n}\mathbb{I})_{\gamma}^{\eta}\mu_{\eta}u^{\gamma}\,dl\no\\
 &=2^{a-1}\frac{(n-a)!}{(n-2a)!}(a-2)!\lim_{R\to\infty}\int_{S_{R}^{n-2}}u_{n}T_{a-2}(\overline{D}^{2}u)_{\gamma}^{\eta}\mu_{\eta}u^{\gamma}\,dl,\label{expression of boundary contribution}
\end{align}
where $L_{\alpha\beta}=e^{-u}u_n\delta_{\alpha\beta}$ on $\partial \mathbb{R}^n_+$, and
\begin{align*}
T_{a-1,a-2}(\overline D^{2}u,u_{n}\mathbb{I})_{\gamma}^{\eta}
& =\frac{u_{n}}{(a-1)!}\delta_{\beta_{1}\cdots\beta_{a-2}\alpha_{a-1}\gamma}^{\alpha_{1}\cdots\alpha_{a-2}\alpha_{a-1}\eta}u_{\alpha_{1}}^{\beta_{1}}\cdots u_{\alpha_{a-2}}^{\beta_{a-2}}\\
& =\frac{(n-a)u_{n}}{a-1}T_{a-2}(\overline{D}^{2}u)_{\gamma}^{\eta}.
\end{align*}

\end{proof}

\begin{thm}\label{one expression of mass}
On AF CF manifolds $(M, g)$ of order $\tau$, assume that $g\in \overline\Gamma_a^+$ and $\int_{M}L_{a}dv_{g}+2a\int_{\partial M}\mathscr{B}^{a-1}d\sigma_{g}$ is finite. Then
\begin{align}
\mathfrak{m}_{a,B}(g)=&\frac{(a-1)!(n-a)!}{2(n-1)!\omega_{n-1}}\bigg(a\int_{\mathbb{R}^n_+\setminus\Omega}\sigma_{a}(D^{2}u)\,dx+a\int_{\mathbb{R}^{n-1}\setminus\Omega'}\sigma_{a-1}(\overline{D}^{2}u)u_{n}\,dx\no\\
 &+\int_{\partial\Omega'}T_{a-2}(\overline{D}^{2}u)^{\alpha\beta}u_{n}u_{\alpha}\mu_{\beta}\,dl+\int_{\partial\Omega}T_{a-1}(D^{2}u)^{ij}u_{i}\nu_{j}\,d\sigma\bigg).
\end{align}
\end{thm}
\begin{proof}
Since $\partial_{\beta}T_{a-2}(\overline{D}^{2}u)^{\alpha\beta}=0$,
\begin{align}
    T_{a-2}(\overline{D}^{2}u)^{\alpha\beta}u_{\beta n}u_{\alpha}\no
=&\partial_{\beta}\left(T_{a-2}(\overline{D}^{2}u)^{\alpha\beta}u_{ n}u_{\alpha}\right)-\partial_{\beta}\left(T_{a-2}(\overline{D}^{2}u)^{\alpha\beta}u_{\alpha}\right)u_{ n}\no\\
=&\partial_{\beta}\left(T_{a-2}(\overline{D}^{2}u)^{\alpha\beta}u_{ n}u_{\alpha}\right)-T_{a-2}(\overline{D}^{2}u)^{\alpha\beta}u_{\alpha\beta}u_n\no\\
=&\partial_{\beta}\left(T_{a-2}(\overline{D}^{2}u)^{\alpha\beta}u_{ n}u_{\alpha}\right)-(a-1)\sigma_{a-1}(\overline{D}^{2}u)u_{n}.
\end{align}
Then, 
\begin{align}
 \int_{B_{R}^{n-1}\backslash\Omega'}T_{a-1}(D^{2}u)^{in}u_{i}\no
= & \int_{B_{R}^{n-1}\backslash\Omega'}T_{a-1}(D^{2}u)^{\alpha n}u_{\alpha}+\int_{B_{R}^{n-1}\backslash\Omega'}\sigma_{a-1}(\overline{D}^{2}u)u_{n}\no\\
= & -\int_{B_{R}^{n-1}\backslash\Omega'}T_{a-2}(\overline{D}^{2}u)^{\alpha\beta}u_{\beta n}u_{\alpha}+\int_{B_{R}^{n-1}\backslash\Omega'}\sigma_{a-1}(\overline{D}^{2}u)u_{n}\no\\
%= & -\int_{\partial B_{R}^{n-1}}T_{a-2}(\overline{D}^{2}u)^{\alpha\beta}u_{n}u_{\alpha}\mu_{\beta}+\int_{\partial\Omega'}T_{a-2}(\overline{D}^{2}u)^{\alpha\beta}u_{n}u_{\alpha}\mu_{\beta}\no\\
% & +\int_{B_{R}^{n-1}\backslash\Omega'}T_{a-2}(\overline{D}^{2}u)^{\alpha\beta}u_{n}u_{\alpha\beta}+\int_{B_{R}^{n-1}\backslash\Omega'}\sigma_{a-1}(\overline{D}^{2}u)u_{n}\no\\
= & -\int_{\partial B_{R}^{n-1}}T_{a-2}(\overline{D}^{2}u)^{\alpha\beta}u_{n}u_{\alpha}\mu_{\beta}+\int_{\partial\Omega'}T_{a-2}(\overline{D}^{2}u)^{\alpha\beta}u_{n}u_{\alpha}\mu_{\beta}\no\\&+a\int_{B_{R}^{n-1}\backslash\Omega'}\sigma_{a-1}(\overline{D}^{2}u)u_{n},\label{boundary term}
\end{align}
where $T_{a-1}(D^2u)^{\alpha n}=-T_{a-2}(\overline D^2u)^{\alpha \beta}u_{\beta n}$ by the definition of $T_{a-1}(D^2u)^{\alpha n}$.
From the divergence structure
\[
\sigma_{a}(D^{2}u)=\frac{1}{a}\bigl(T_{a-1}(D^2 u)^{ij}u_j\bigr)_i
\]
and \eqref{boundary term}, we get
\begin{align*}
  a\int_{B_{R}^{+}\backslash\Omega}\sigma_{a}(D^{2}u)
= & \int_{S_{R}^{+}}T_{a-1}(D^{2}u)^{ij}u_{i}\nu_{j}-\int_{\partial\Omega}T_{a-1}(D^{2}u)^{ij}u_{i}\nu_{j} -\int_{B_{R}^{n-1}\backslash\Omega'}T_{a-1}(D^{2}u)^{in}u_{i}\\
= & \int_{S_{R}^{+}}T_{a-1}(D^{2}u)^{ij}u_{i}\nu_{j}-\int_{\partial\Omega}T_{a-1}(D^{2}u)^{ij}u_{i}\nu_{j}-a\int_{B_{R}^{n-1}\backslash\Omega'}\sigma_{a-1}(\overline{D}^{2}u)u_{n}\\
 & +\int_{\partial B_{R}^{n-1}}T_{a-2}(\overline{D}^{2}u)^{\alpha\beta}u_{n}u_{\alpha}\mu_{\beta}-\int_{\partial\Omega'}T_{a-2}(\overline{D}^{2}u)^{\alpha\beta}u_{n}u_{\alpha}\mu_{\beta},
\end{align*}
where $\nu$ on $\partial \Omega$ and $\mu$ on $\partial \Omega'$ are the  normal vector  pointing outward of $\Omega$ and $\Omega'$, respectively.
Then, 
\begin{align*}
 a\int_{B_{R}^{+}\backslash\Omega}\sigma_{a}(D^{2}u)+a\int_{B_{R}^{n-1}\backslash\Omega'}\sigma_{a-1}(\overline{D}^{2}u)u_{n}\no
 &+\int_{\partial\Omega'}T_{a-2}(\overline{D}^{2}u)^{\alpha\beta}u_{n}u_{\alpha}\mu_{\beta} +\int_{\partial\Omega}T_{a-1}(D^{2}u)^{ij}u_{i}\nu_{j}\\
= & \int_{S_{R}^{+}}T_{a-1}(D^{2}u)^{ij}u_{i}\nu_{j}
 +\int_{\partial B_{R}^{n-1}}T_{a-2}(\overline{D}^{2}u)^{\alpha\beta}u_{n}u_{\alpha}\mu_{\beta}.
\end{align*}

By Proposition~\ref{asym behavior}, we have
\begin{align}
&\lim_{R\to\infty}\bigg(\int_{S_{R}^{+}}g_{jk,l}P_a^{ijkl}\nu_{i}\,d\sigma
+(a-1)\int_{S_{R}^{n-2}}Q_{a,a-2}^{\eta\beta\alpha\gamma}g_{\beta\alpha,\gamma}\mu_{\eta}\,dl\bigg)\no\\
=&\,2^{a-1}(a-1)!\frac{(n-a)!}{(n-2a)!}\bigg(a\int_{B_{R}^{+}\setminus\Omega}\sigma_{a}(D^{2}u)+a\int_{B_{R}^{n-1}\setminus\Omega'}\sigma_{a-1}(\overline{D}^{2}u)u_{n}\no\\
 &+\int_{\partial\Omega'}T_{a-2}(\overline{D}^{2}u)^{\alpha\beta}u_{n}u_{\alpha}\mu_{\beta} +\int_{\partial\Omega}T_{a-1}(D^{2}u)^{ij}u_{i}\nu_{j}\bigg).\label{intermediate expression}
\end{align}
\end{proof}
\iffalse
\begin{thm}\label{lower bound 1}
On AF CF manifolds $(M, g)$ of order $\tau$, assume that $g\in \overline\Gamma_a^+$ and $\int_{M}L_{a}dv_{g}+2a\int_{\partial M}\mathscr{B}^{a-1}d\sigma_{g}$ is finite. Let $u$ be constant on $\partial \Omega$. Then,
\begin{align*}
\mathfrak{m}_{a,B}(g)&\ge \frac{a!(n-a)!}{2(n-1)!\omega_{n-1}}\bigg(\int_{\mathbb{R}^n_+\backslash\Omega}\sigma_a(A)+\int_{\mathbb{R}^{n-1}\backslash\Omega'}\sum_{k=0}^{a-1}C_2(n,a,k)\sigma_k(A^T)u_n^{2a-2k-1}
\bigg)\no\\
&+\frac{(a-1)!(n-a)!}{2(n-1)!\omega_{n-1}}\bigg(\int_{\partial\Omega'} \sigma_{a-2}(\kappa_{\partial \Omega'})|\overline{D}u|^{a-2}\langle\overline{D} u,\mu\rangle u_n dl \no\\
&+\int_{\partial\Omega} \sigma_{a-1}(\kappa_{\partial \Omega})|Du|^{a-1}\langle D u,\nu\rangle d\sigma\bigg)\no\\
&+\frac{(n-2 a)}{2^{a+1} \omega_{n-1}}\int_{\mathbb{R}^n_+\backslash\Omega}|Du|^{2 a}+\frac{a(n-2a+1)}{2^{a}(n-1)\omega_{n-1}}\int_{\mathbb{R}^{n-1}\backslash\Omega'}|\overline{D} u|^{2 (a-1)}u_n\no
\end{align*}
where $C_2(n,a,k)=\frac{(n-1-k)!}{(2a-2k-2)!!(n-a)!}$, and $\nu$ and $\mu$ are unit outer normal vectors of $\partial \Omega$ and $\partial \Omega'$ respectively.
\end{thm}

\fi

We now provide a Penrose-type inequality.
\begin{thm}\label{lower bound of mass}
Assume the hypotheses of Theorem~\ref{thm on CF}. Suppose in addition that $u$ is constant on $\partial\Omega$. Then
\begin{align}
\mathfrak{m}_{a,B}(g)
\ge& \,\frac{a!(n-a)!}{2(n-1)!\omega_{n-1}}\bigg(\int_{\mathbb{R}^n_+\setminus\Omega}\sigma_a(g^{-1}A_g)e^{(n-2a)u}dv_{g}\no\\
&+\int_{\mathbb{R}^{n-1}\setminus\Omega'}\sum_{k=0}^{a-1}C_2(n,a,k)\sigma_k(g^{-1}A^T_g)h_g^{2a-2k-1}e^{(n-2a)u}d\sigma_{g}\bigg)
\no\\
&+\frac{(a-1)!(n-a)!}{2(n-1)!\omega_{n-1}}\bigg(\int_{\partial\Omega} \sigma_{a-1}(\kappa_{\partial \Omega})\left(\frac{\sigma_1(\kappa_{\partial \Omega})}{n-1}\right)^a d\sigma\no\\
&+\int_{\partial\Omega'} \sigma_{a-2}(\kappa_{\partial \Omega'})\left(\frac{\sigma_1(\kappa_{\partial \Omega})}{n-1}\right)^a\sin^{a-1} \theta\cos \theta \,dl\bigg)\no \\
&+\frac{(n-2 a)}{2^{a+1} \omega_{n-1}}\int_{\mathbb{R}^n_+\setminus\Omega}|\nabla_g u|^{2 a}e^{(n-2a)u}dv_{g}\\
&+\frac{a(n-2a+1)}{2^{a}(n-1)\omega_{n-1}}\int_{\mathbb{R}^{n-1}\setminus\Omega'}|\overline{\nabla}_g u|^{2 (a-1)}h_ge^{(n-2a)u} d\sigma_{g}.\no
\end{align}
Here $C_2(n,a,k)=\frac{(n-1-k)!}{(2a-2k-2)!!(n-a)!}$ and $\theta$ is the angle formed by the tangent plane of $\partial\Omega$ at $\partial\Omega'$ and $x_n=0$, satisfying $0<\theta<\pi/2$.
\end{thm}

\begin{proof}
Now we utilize the Schouten tensor $A$ and the tangential part $A^T$ to obtain the lower bound of $a\int_{B_{R}^{+}\backslash\Omega}\sigma_{a}(\nabla^{2}u)+a\int_{B_{R}^{n-1}\backslash\Omega'}\sigma_{a-1}(\overline{D}^{2}u)u_{n}$.

Denote
\[
B=\frac{|Du|^2}{2} I_{n\times n}-{D} u \otimes {D} u.
\]
Then $B\in \Gamma_k^{+}$ for any $k<n/2$, and
\[
\sigma_j(B)=\frac{(n-1)!(n-2 j)}{2^j j!(n-j)!}|Du|^{2 j}\qquad \text{for any } j<n/2.
\]
Thus, by $A\in \Gamma_a^{+} $ and $B\in \Gamma_a^{+}$, we obtain
\begin{align}\label{lower bound of sigma Hessian}
\sigma_a\left(D^2 u\right) & =\sigma_a(A+B)  \geq \sigma_a(A)+\sigma_a(B)=\sigma_a(A)+\frac{(n-1)!(n-2 a)}{2^a a!(n-a)!}|Du|^{2 a}.
\end{align}
Denote
\[
\overline{B}=\frac{|\overline{D} u|^2}{2} I_{(n-1)\times (n-1)}-\overline{D} u \otimes \overline D u,
\]
and note that $\overline{B}\in \Gamma_k^{+}$ for any $k< \frac{n-1}{2}$.

It follows that
$$A^T=\overline{D}^{2}u-\overline{B}-\frac{u_n^2}{2}I_{(n-1)\times (n-1)}$$ and
\begin{align}
\sigma_{a-1}\left(\overline{D}^{2}u\right)
 & =\sigma_{a-1}(A^T+\overline{B}+\frac{u_n^2}{2}I_{(n-1)\times (n-1)}) \no\\
& \geq \sigma_{a-1}(A^T+\frac{u_n^2}{2}I_{(n-1)\times (n-1)})+\sigma_{a-1}(\overline{B})\no\\
&=\sigma_{a-1}(A^T+\frac{u_n^2}{2}I_{(n-1)\times (n-1)})+\frac{(n-2)!(n-2a+1)}{2^{a-1}(a-1)!(n-a)!}|\overline{D} u|^{2 (a-1)}.\label{sigma{a-1}nabla}
\end{align}
Note that
\begin{align}\label{lower bound of sigma-1 tangential Hessian}
\sigma_{a-1}(A^T+\frac{u_n^2}{2}I_{(n-1)\times (n-1)})
=&\frac{1}{(a-1)!}\delta^{\alpha_1\cdots \alpha_{a-1}}_{\beta_1\cdots \beta_{a-1}}(A_{\alpha_ 1}^{\beta_1}+\frac{u_n^2}{2}\delta_{\alpha_1}^{\beta_1})\cdots(A_{\alpha_{a-1}}^{\beta_{a-1}}+\frac{u_n^2}{2}\delta_{\alpha_{a-1}}^{\beta_{a-1}})\no\\
=&\sum_{k=0}^{a-1}\frac{C_{a-1}^k}{(a-1)!}\delta^{\alpha_1\cdots \alpha_{a-1}}_{\beta_1\cdots \beta_{a-1}}A_{\alpha_ 1}^{\beta_1}\cdots A_{\alpha_ {k}}^{\beta_k}\delta_{\alpha_{k+1}}^{\beta_{k+1}}\cdots \delta_{\alpha_{a-1}}^{\beta_{a-1}}(\frac{u_n^2}{2})^{a-1-k}\no\\
=&\sum_{k=0}^{a-1}\frac{C_{a-1}^k}{(a-1)!}(\frac{u_n^2}{2})^{a-1-k}\delta^{\alpha_1\cdots\alpha_k\alpha_{k+1}\cdots \alpha_{a-1}}_{\beta_1\cdots \beta_k\alpha_{k+1}\cdots\alpha_{a-1}}A_{\alpha_ 1}^{\beta_1}\cdots A_{\alpha_ {k}}^{\beta_k}\no\\
=&\sum_{k=0}^{a-1}\frac{C_{a-1}^k}{(a-1)!}(\frac{u_n^2}{2})^{a-1-k}\frac{(n-1-k)!}{(n-a)!}k!\sigma_k(A^T)\no\\
=&\sum_{k=0}^{a-1}\frac{(n-1-k)!}{(2a-2k-2)!!(n-a)!}\sigma_k(A^T)u_n^{2(a-1-k)}.
\end{align}
Substituting \eqref{lower bound of sigma-1 tangential Hessian} into \eqref{sigma{a-1}nabla}, we get
\begin{align}
\sigma_{a-1}\left(\overline{D}^{2}u\right)
\ge &\sum_{k=0}^{a-1}\frac{(n-1-k)!}{(2a-2k-2)!!(n-a)!}\sigma_k(A^T)u_n^{2(a-1-k)}+\frac{(n-2)!(n-2a+1)}{2^{a-1}(a-1)!(n-a)!}|\overline{D} u|^{2 (a-1)}.\label{lower bound of tangential hessian u}
\end{align}
Let $\nu$ be the unit outer normal vector of $\partial \Omega$ pointing into $\mathbb{R}^n\setminus\Omega$, and let $\mu$ be the unit outer normal vector of $\partial \Omega'$ pointing into $\mathbb{R}^{n-1}\setminus\Omega'.$
Recall that $h_{\alpha \beta}$ is the second fundamental form of $\partial \Omega\subset\mathbb{R}^n$, i.e., $h_{\alpha \beta}=\langle D_{e_\alpha}e_\beta, -\nu\rangle$, where $e_\alpha,e_\beta$ are unit tangent vectors on $\partial \Omega$. We denote its principal curvatures by $\kappa_{\partial \Omega}$ (so $\sigma_1(\kappa_{\partial \Omega})=\sigma_1(h_{\alpha \beta})$). Likewise, let $h'_{\alpha \beta}=\langle D_{e'_\alpha}e'_\beta, -\mu\rangle$ be the second fundamental form of $\partial \Omega'\subset\mathbb{R}^{n-1}$, where $e'_\alpha,e'_\beta$ are unit tangent vectors on $\partial \Omega'$, and denote its principal curvatures by $\kappa_{\partial \Omega'}$.

Since $\partial M\setminus\partial \mathbb{R}^n_+$ is a horizon of $M$, the mean curvature satisfies $H_{\partial M}=0$ on $\partial M\setminus\partial \mathbb{R}^n_+$. Under the conformal change this gives
\[
\sigma_1(\kappa_{\partial \Omega})-(n-1)\langle D u, \nu\rangle=0\qquad\text{on }\partial \Omega.
\]
Since $\sigma_1(\kappa_{\partial \Omega})>0$, we have $\langle D u, \nu\rangle>0$ on $\partial \Omega$, and hence $\nu=\frac{Du}{|Du|}$ when $u$ is constant on $\partial \Omega$.

Along $\partial \Omega'$, let $\theta$ be the angle formed by the tangent plane of $\partial\Omega$ at $\partial\Omega'$ and $x_{n}=0$. Then
\[
\nu= \mu \sin\theta+e_n \cos \theta.
\]
In particular, on $\partial \Omega'$ we compute
\begin{align*}
\langle {D} u, \nu\rangle
&=\langle \overline{D}u+u_n e_n, \nu\rangle
=\sin\theta\langle \overline{D}u,\mu\rangle+u_n\cos \theta.
\end{align*}
Since $u$ is constant on $\partial \Omega$, we have $u_n=u_{\nu}\cos\theta=|Du|\cos \theta$ on $\partial \Omega'$ and
\[
|Du|=\langle {D} u, \nu\rangle=\sin\theta\langle \overline{D}u,\mu\rangle+|Du|(\cos \theta)^2.
\]
It follows that $\langle\overline{D}u,\mu\rangle=|Du|\sin \theta$ and hence
$|\overline D u|=|Du|\,|\sin \theta|=\frac{\sigma_1(\kappa_{\partial \Omega})}{n-1}|\sin \theta|$ on $\partial \Omega'$.

Since the mean curvature $h_g$ of $\partial M$ is nonnegative, we have $u_n\ge 0$, and since $u_n=|Du|\cos\theta$, it follows that $0< \theta\le \pi/2$. 
On $\partial \Omega$, $D^2u(e_\alpha,e_\beta)=h_{\alpha\beta} u_{\nu}$.
Since $g\in \overline \Gamma_a^+$, we know that $A_{\alpha \beta}=D^2u(e_\alpha,e_\beta)-\frac{|Du|^2}{2}\delta_{\alpha\beta}\in \overline \Gamma_{a-1}^+$ and then $h_{\alpha\beta} u_{\nu}\in \overline \Gamma_{a-1}^+$. Due to $u_\nu>0$ on $\partial \Omega$, $h_{\alpha\beta}\in \overline\Gamma_{a-1}^+$ and then $\sigma_{a-1}(\kappa_{\partial \Omega})\ge 0$. And as $\langle \overline Du,\mu\rangle=|Du|\sin\theta>0$, we have $\mu=\frac{\overline D u}{|\overline D u|}$ on $\partial \Omega'$. Since $\{A_g(\frac{\partial}{\partial x^{\alpha}},\frac{\partial}{\partial x^{\beta}})\}_{(n-1)\times(n-1)}\in \overline\Gamma_{a-1}^+$ on $\partial \mathbb{R}^n_+$, by similar argument as $\partial \Omega$, we know that $h'_{\alpha\beta}\in \overline\Gamma_{a-2}^+$ on $\partial \Omega'$, where $h'_{\alpha\beta}$ is the second fundamental form of $\partial \Omega'$ in $\partial \mathbb{R}^n.$

Then, with $\nu=\frac{Du}{|Du|}$ and $\mu=\frac{\overline D u}{|\overline Du|}$, by arguments similar to \eqref{eq:formula 1-1} and \eqref{eq:formula 1-1-1}, we obtain
\begin{equation}\label{boundary curvature}
\int_{\partial\Omega}T_{a-1}(D^{2}u)^{ij}u_{i}\nu_{j}=\int_{\partial\Omega} \sigma_{a-1}(\kappa_{\partial \Omega})|Du|^{a-1}\langle D u,\nu\rangle \,d\sigma.
\end{equation}
and 
\begin{equation}\label{small boundary curvature}
\int_{\partial\Omega'}T_{a-2}(\overline{D}^{2}u)^{\alpha\beta}u_{n}u_{\alpha}\mu_{\beta}=\int_{\partial\Omega'} \sigma_{a-2}(\kappa_{\partial \Omega'})|\overline{D}u|^{a-2}\langle\overline{D} u,\mu\rangle u_n dl.
\end{equation} 

By Theorem~\ref{one expression of mass}, together with \eqref{lower bound of sigma Hessian}, \eqref{lower bound of tangential hessian u}, \eqref{boundary curvature}, and \eqref{small boundary curvature}, we conclude
\begin{align*}
\mathfrak{m}_{a,B}(g)
\ge &\frac{a!(n-a)!}{2(n-1)!\omega_{n-1}}\bigg(\int_{\mathbb{R}^n_+\backslash\Omega}\sigma_a(A)+\int_{\mathbb{R}^{n-1}\backslash\Omega'}\sum_{k=0}^{a-1}C_2(n,a,k)\sigma_k(A^T)u_n^{2a-2k-1}
\bigg)\no\\
&+\frac{(a-1)!(n-a)!}{2(n-1)!\omega_{n-1}}\bigg(\int_{\partial\Omega'} \sigma_{a-2}(\kappa_{\partial \Omega'})|\overline{D}u|^{a-2}\langle\overline{D} u,\mu\rangle u_n dl \no\\
&+\int_{\partial\Omega} \sigma_{a-1}(\kappa_{\partial \Omega})|Du|^{a-1}\langle D u,\nu\rangle d\sigma\bigg)\no\\
&+\frac{(n-2 a)}{2^{a+1} \omega_{n-1}}\int_{\mathbb{R}^n_+\backslash\Omega}|Du|^{2 a}+\frac{a(n-2a+1)}{2^{a}(n-1)\omega_{n-1}}\int_{\mathbb{R}^{n-1}\backslash\Omega'}|\overline{D} u|^{2 (a-1)}u_n.\no
\end{align*}
With
\begin{align}
&\int_{\partial\Omega'} \sigma_{a-2}(\kappa_{\partial \Omega'})|\overline{D}u|^{a-2}\langle\overline{D} u,\mu\rangle u_n\,dl\no\\
=&\int_{\partial\Omega'} \sigma_{a-2}(\kappa_{\partial \Omega'})\left(\frac{\sigma_1(\kappa_{\partial \Omega})}{n-1}\right)^a|\sin \theta|^{a-2}\cos \theta\sin\theta\,dl,
\end{align}
we can rewrite the quantities under the metric $g=e^{-2u}g_{\mathbb{E}}$ and obtain the desired inequality.
\end{proof}
Theorem~\ref{thm on CF} follows from Theorem~\ref{lower bound of mass}.

\iffalse
Theorem \ref{thm1 on CF} and Theorem \ref{thm on CF} can be deduced by Theorem \ref{lower bound of mass} with the following argument:
\begin{proof}[Proof of Theorem \ref{lower bound of mass}]
Since the mean curvature $h_g$ of $\partial M$  is non-negative, we have $u_n\ge 0$ and due to $u_n=|Du|\cos\theta$, we get that $0\le \theta\le \pi/2$. 
On $\partial \Omega$, $D^2u(e_\alpha,e_\beta)=h_{\alpha\beta} u_{\nu}$, where $h_{\alpha\beta}$ is the second fundamental form of $\partial \Omega$ in $\mathbb{R}^n$ and $e_\alpha,e_\beta$ are unit tangential vectors on $\partial \Omega$.
Since $g\in \overline \Gamma_a^+$, we know that $A_{\alpha \beta}=D^2u(e_\alpha,e_\beta)-\frac{|Du|^2}{2}\delta_{\alpha\beta}\in \overline \Gamma_{a-1}^+$ and then $h_{\alpha\beta} u_{\nu}\in \overline \Gamma_{a-1}^+$. Due to $u_\nu>0$ on $\partial \Omega$, $h_{\alpha\beta}\in \overline\Gamma_{a-1}^+$ and then $\sigma_{a-1}(\kappa_{\partial \Omega})\ge 0$. Since $\{A_g(\frac{\partial}{\partial x^{\alpha}},\frac{\partial}{\partial x^{\beta}})\}_{(n-1)\times(n-1)}\in \overline\Gamma_{a-1}^+$ on $\partial \mathbb{R}^n_+$ and $\langle \overline Du,\mu\rangle=|Du|\sin\theta>0$, by similar argument for $\partial \Omega$, we know that $h'_{\alpha\beta}\in \overline\Gamma_{a-2}^+$ on $\partial \Omega'$, where $h'_{\alpha\beta}$ is the second fundamental form of $\partial \Omega'$ in $\partial \mathbb{R}^n.$
\end{proof}

\fi

\appendix
\section{Appendix}

\subsection{Second fundamental form of the boundary of a graph}
Let $f\colon\mathbb{R}_{+}^{n}\setminus\Omega\to\mathbb{R}$ be a $C^{2}$ function up to the boundary. Let $M=(x,f(x))\subset\mathbb{R}^{n+1}$ be its graph, equipped with the induced metric $g_{ij}=\delta_{ij}+f_i f_j$. Then
\[
\partial M=\bigl\{(x',0,f(x',0))\bigr\}.
\]
The notation in this subsection is consistent with Section~3.
\begin{lem}\label{lem:second fundametal in graph}
The second fundamental form of $\partial M$ (as a hypersurface of $(M,g)$) is given by
\[
L_{\alpha\beta}=h_{\alpha\beta}^{M,\partial M}
=\frac{f_{\alpha\beta}f_{n}}{\sqrt{(1+|Df|^{2})(1+|\overline{D}f|^{2})}},
\qquad 1\le\alpha,\beta\le n-1,
\]
where $f_i=\partial_i f$, $f_{\alpha\beta}=\partial_{\alpha}\partial_{\beta}f$, and $Df=(f_1,\dots,f_n)=:(\overline{D}f,f_n)$.
\end{lem}

\begin{proof}
Let $\nu^{*}$ be the unit normal to $\partial M$ in $(M,g)$. Using Gram--Schmidt, one may write $\nu^{*}$ as the following formal determinant (interpreted by cofactor expansion along the vector row):
\[
\nu^{*}=\frac{1}{\sqrt{(1+|Df|^{2})(1+|\overline{D}f|^{2})}}\left|\begin{array}{ccccc}
\langle\vec{v}_{1}, \vec v_{1}\rangle &  \langle\vec v_{2},\vec v_{1} \rangle & \cdots &  & \langle\vec v_{n},\vec v_{1}\rangle\\
\langle\vec v_{1}, \vec v_{2}\rangle & \langle\vec v_{2},\vec v_{2}\rangle &  &  & \langle\vec v_{n},\vec v_{2}\rangle\\
\vdots & \vdots &  &  & \vdots\\
\langle\vec v_{1},\vec v_{n-1}\rangle &  &  &  & \langle\vec v_{n},\vec v_{n-1}\rangle\\
\vec v_{1} &\vec  v_{2} & \cdots &  & \vec v_{n}
\end{array}\right|,
\]
where $\vec v_i=(e_i,f_i)\in\mathbb{R}^{n+1}$, $e_i$ is the $i$th standard basis vector in $\mathbb{R}^{n}$, and
$\langle \vec v_i,\vec v_j\rangle:=\delta_{ij}+f_i f_j$.
A direct computation gives
\begin{align*}
\nu^{*}
&=\frac{-\sum_{i=1}^{n-1}f_{i}f_{n}\vec v_{i}+(1+|\bar{D}f|^{2})\vec v_{n}}{\sqrt{(1+|Df|^{2})(1+|\overline{D}f|^{2})}}\\
&=\frac{(-f_{n}\bar{D}f,\,1+|\bar{D}f|^{2},\,f_{n})}{\sqrt{(1+|Df|^{2})(1+|\overline{D}f|^{2})}},
\end{align*}
so that $\langle \nu^{*},\nu^{*}\rangle=1$ and $\langle \vec v_i,\nu^{*}\rangle=0$ for $i=1,\dots,n-1$.

For $i,j\in\{1,\dots,n-1\}$ we have
\[
\sqrt{(1+|Df|^{2})(1+|\overline{D}f|^{2})}\,\langle\nu^{*},(0,\dots,0,f_{ij})\rangle=f_{n}f_{ij},
\]
and hence
\begin{align*}
L_{ij}
&=\langle \nu^{*},D_{\vec v_{i}}\vec v_{j}\rangle
=\langle\nu^{*},(0,\dots,0,f_{ij})\rangle
=\frac{f_{n}f_{ij}}{\sqrt{(1+|Df|^{2})(1+|\overline{D}f|^{2})}}.
\end{align*}
\end{proof}

\subsection{A variational functional for $L_a$ with boundary}
Let $(M^{n},g)$ be an $n$-dimensional manifold with boundary $\partial M$. Near $\partial M$ we use local coordinates $\{x^{i}\}_{i=1}^{n}$ in $M$ such that $x^{n}$ is a defining function for $\partial M$ and $\{x^{\alpha}\}_{\alpha=1}^{n-1}$ are local coordinates on $\partial M$. Let $\nabla$ be the Levi--Civita connection of $g$, and let $\bar\nabla$ be the Levi--Civita connection of the induced metric $\bar g$ on $\partial M$. We write $R_{ijkl}$ for the Riemann curvature tensor of $g$ and $\bar R_{\alpha\beta\gamma\eta}$ for the Riemann curvature tensor of $\bar g$.

In Fermi coordinates, the metric takes the form
\[
g=(dx^{n})^{2}+g_{\alpha\beta}\,dx^{\alpha}dx^{\beta}.
\]
Near a boundary point, Fermi coordinates satisfy $\Gamma_{\alpha\beta}^{\gamma}=0$ and $\Gamma_{\alpha\beta}^{n}=L_{\alpha\beta}$, where
\[
L_{\alpha\beta}=\bigl\langle\nabla_{\partial_{\alpha}}\partial_{\beta},\partial_{n}\bigr\rangle,
\]
and $\partial_{n}=\frac{\partial}{\partial x^{n}}$ is the inward unit normal along $\partial M$. On $\partial M$ we have $g_{\alpha\beta}=\bar g_{\alpha\beta}$, and the Gauss equation yields
\[
R_{\alpha \beta}^{\ \gamma \eta}=\overline{R}_{\alpha\beta }^{\ \gamma \eta}-L_{\alpha}^{\gamma}L_{\beta}^{\eta}+L_{\alpha}^{\eta}L_{\beta}^{\gamma}.
\]

We study the first variation of the relevant bulk--boundary functional under a metric variation $g(t)$ with
$h_{ij}=\frac{\partial g_{ij}}{\partial t}(0)$. We will use the standard variation formulas (see \cite{Ricciflow}):
\begin{equation}\label{derivative of inverse g}
\frac{\partial}{\partial t}\Big|_{t=0}g^{ij}=-g^{ik}h_{kl}g^{lj},
\end{equation}
\begin{equation}\label{dGamma}
\frac{\partial}{\partial t}\Big|_{t=0}\Gamma(g)_{ij}^{k}=\frac{1}{2}g^{kl}\left(\nabla_{i}h_{jl}+\nabla_{j}h_{il}-\nabla_{l}h_{ij}\right),
\end{equation}
and
\begin{align}
\frac{\partial}{\partial t}R_{ij}^{\quad kl}
=&-g^{lr}h_{rs}g^{sp}R_{ij\quad p}^{\quad k}
+\frac{1}{2}g^{lp}g^{kq}\bigg(\nabla_{i}\nabla_{j}h_{pq}-\nabla_{j}\nabla_{i}h_{pq}\nonumber \\ 
&\qquad +\nabla_{i}\nabla_{p}h_{jq}-\nabla_{j}\nabla_{p}h_{iq}-\nabla_{i}\nabla_{q}h_{jp}+\nabla_{j}\nabla_{q}h_{ip}\bigg).\no
\end{align}

For the reader's convenience, we record the variation of the second fundamental form (compare, e.g., \cite{Lott} for variations of the Gibbons--Hawking--York functional).
\begin{lem}\label{derivative of L}
Let $h_{ij}=\frac{\partial g_{ij}}{\partial t}(0)$. Along $\partial M$, the first variation of the second fundamental form satisfies
\[
\frac{d}{dt}L_{\alpha\beta}\Big|_{t=0}=\frac{1}{2}\left(\nabla_{\alpha}^{g}h_{\beta n}+\nabla_{\beta}^{g}h_{\alpha n}-\nabla_{n}^{g}h_{\alpha\beta}\right)+\frac{h_{nn}}{2}L_{\alpha\beta},
\]
where $L_{\alpha\beta}=\langle \nabla_{\partial_{\alpha}}\partial_{\beta},e_n\rangle$ and $e_n$ denotes the inward unit normal for $g(t)$.
\end{lem}

\begin{proof}
Along $\partial M$ we have $g(e_n,e_n)=1$ and $g(e_n,\partial_{\alpha})=0$ for $1\le\alpha\le n-1$, and at $t=0$ we may take $e_n=\partial_n$. Differentiating these identities gives
\[
0=\frac{d}{dt}g(e_n,e_n)=h(e_n,e_n)+2g\Bigl(\frac{d e_n}{dt},e_n\Bigr)
\]
and
\[
0=\frac{d}{dt}g(e_n,\partial_{\alpha})=h(e_n,\partial_{\alpha})+g\Bigl(\frac{d e_n}{dt},\partial_{\alpha}\Bigr),
\]
where we used that $\partial_{\alpha}$ is fixed in the coordinate chart. Writing
\[
\frac{d e_n}{dt}=b_{n}e_n+\sum_{\alpha}a_{\alpha}\partial_{\alpha},
\]
we obtain $b_{n}=-\frac{h_{nn}}{2}$ and $a_{\alpha}=-g^{\alpha\beta}h_{n\beta}$. Using \eqref{dGamma} and $\Gamma_{\alpha\beta}^{\gamma}=0$ at $t=0$, we compute
\begin{align*}
\frac{d}{dt}L_{\alpha\beta}
&=\frac{d}{dt}\Gamma_{\alpha\beta}^{k}g_{kn}+\Gamma_{\alpha\beta}^{k}h_{kn}+\Gamma_{\alpha\beta}^{k}g\Bigl(\partial_{k},\frac{d e_n}{dt}\Bigr)\\
&=\frac{1}{2}\left(\nabla_{\alpha}^{g}h_{\beta n}+\nabla_{\beta}^{g}h_{\alpha n}-\nabla_{n}^{g}h_{\alpha\beta}\right)+\frac{h_{nn}}{2}L_{\alpha\beta},
\end{align*}
which proves the claim.
\end{proof}

\begin{cor}\label{cor:La^b derivative}
On $\partial M$, the mixed form $L^{\alpha}_{\beta}$ satisfies
\begin{equation}
\frac{d}{dt}L_{\beta}^{\alpha}=\frac{1}{2}\left(\nabla^{\alpha}h_{\beta n}+\nabla^{\beta}h_{\alpha n}-\nabla_{n}h_{\beta}^{\alpha}\right)+\frac{h_{nn}}{2}L_{\beta}^{\alpha}-h_{\gamma}^{\alpha}L_{\beta}^{\gamma}.
\end{equation}
\end{cor}

\begin{proof}
This follows by differentiating $L_{\beta}^{\alpha}=g^{\alpha q}L_{\beta q}$ and using \eqref{derivative of inverse g}.
\end{proof}

We now introduce the generalized boundary curvatures associated with the Lovelock curvature.

For $0\le k\le a-1$, on $\partial M$, we define
\begin{align}
B_{a}^{k}
=&\frac{1}{2^{k}}\sum_{i,j=1}^{n-1}\delta_{j_{1}j_{2}\cdots j_{2a-3}j_{2a-2}j_{2a-1}}^{i_{1}i_{2}\cdots i_{2a-3}i_{2a-2}i_{2a-1}}(\overline{R}_{i_{1}i_{2}}^{\quad j_{1}j_{2}}-L_{i_{1}}^{j_{1}}L_{i_{2}}^{j_{2}}+L_{i_{1}}^{j_{2}}L_{i_{2}}^{j_{1}})\no\\
&\cdots(\overline{R}_{i_{2k-1}i_{2k}}^{\quad j_{2k-1}j_{2k}}-L_{i_{2k-1}}^{j_{2k-1}}L_{i_{2k}}^{j_{2k}}+L_{i_{2k-1}}^{j_{2k}}L_{i_{2k}}^{j_{2k-1}})L_{i_{2k+1}}^{j_{2k+1}}\cdots L_{i_{2a-1}}^{j_{2a-1}},
\end{align}
\iffalse
\[
B_{a}^{k}=\frac{1}{2^{k}}\sum_{i,j=1}^{n-1}\delta_{j_{1}\cdots j_{2a-1}}^{i_{1}\cdots i_{2a-1}}R_{i_{1}i_{2}}^{\quad j_{1}j_{2}}\cdots R_{i_{2k-1}i_{2k}}^{\quad j_{2k-1}j_{2k}}L_{i_{2k+1}}^{j_{2k+1}}\cdots L_{i_{2a-1}}^{j_{2a-1}},
\]
\fi
and for $1\le i,j\le n-1$,
\begin{align}
(\bar Q_{a,k-1})_{j_{1}j_{2}}^{i_{1}i_{2}}
=&\frac{1}{2^{k}}\sum_{i,j=1}^{n-1}\delta_{j_{1}\cdots j_{2a-1}}^{i_{1}\cdots i_{2a-1}}(\overline{R}_{i_{3}i_{4}}^{\quad j_{3}j_{4}}-L_{i_{3}}^{j_{3}}L_{i_{4}}^{j_{4}}+L_{i_{3}}^{j_{4}}L_{i_{4}}^{j_{3}})\no\\
&\cdots(\overline{R}_{i_{2k-1}i_{2k}}^{\quad j_{2k-1}j_{2k}}-L_{i_{2k-1}}^{j_{2k-1}}L_{i_{2k}}^{j_{2k}}+L_{i_{2k-1}}^{j_{2k}}L_{i_{2k}}^{j_{2k-1}})L_{i_{2k+1}}^{j_{2k+1}}\cdots L_{i_{2a-1}}^{j_{2a-1}}.
\end{align}

\iffalse
\[
(Q_{a, k-1})_{j_{1}j_{2}}^{i_{1}i_{2}}=\frac{1}{2^{k}}\sum_{i,j=1}^{n-1}\delta_{j_{1}\cdots j_{2a-1}}^{i_{1}\cdots i_{2a-1}}R_{i_{3}i_{4}}^{\quad j_{3}j_{4}}\cdots R_{i_{2k-1}i_{2k}}^{\quad j_{2k-1}j_{2k}}L_{i_{2k+1}}^{j_{2k+1}}\cdots L_{i_{2a-1}}^{j_{2a-1}}.
\]
\fi

On $\partial M$, we also define the tensor
\begin{align}\label{eq:Pa-k-1-3}
(\bar Q_{a,k})_{\beta_{1}\beta_{2}\beta_{3}}^{\alpha_1\alpha_2\alpha_3}
:=&\frac{1}{2^{k}}\sum_{i,j=1}^{n-1}\delta_{\beta_{1}\beta_{2}\beta_{3}j_1\cdots j_{2a-4}}^{\alpha_1\alpha_2\alpha_3i_{1}\cdots i_{2a-4}}(\overline{R}_{i_{1}i_{2}}^{\quad j_{1}j_{2}}-L_{i_{1}}^{j_{1}}L_{i_{2}}^{j_{2}}+L_{i_{1}}^{j_{2}}L_{i_{2}}^{j_{1}})\no\\
&\cdots(\overline{R}_{i_{2k-1}i_{2k}}^{\quad j_{2k-1}j_{2k}}-L_{i_{2k-1}}^{j_{2k-1}}L_{i_{2k}}^{j_{2k}}+L_{i_{2k-1}}^{j_{2k}}L_{i_{2k}}^{j_{2k-1}})\,L_{i_{2k+1}}^{j_{2k+1}}\cdots L_{i_{2a-4}}^{j_{2a-4}}.
\end{align}
We adopt the convention that $(\bar Q_{a,k})_{j_{1}j_{2}j_{3}}^{i_{1}i_{2}i_{3}}=0$ when $k<0$.

For $0\le k\le a-1$, we introduce the boundary Einstein-type tensor
\begin{align}\label{expression of Eak}
({E}_{a}^{k})_{\alpha \beta}
:=&-\frac{1}{2^{k+1}}\sum_{i,j,l=1}^{n-1}g_{\alpha l}\delta_{\beta j_{1}\cdots j_{2a-1}}^{li_{1}\cdots i_{2a-1}}(\overline{R}_{i_{1}i_{2}}^{\quad j_{1}j_{2}}-L_{i_{1}}^{j_{1}}L_{i_{2}}^{j_{2}}+L_{i_{1}}^{j_{2}}L_{i_{2}}^{j_{1}})\no\\
&\cdots(\overline{R}_{i_{2k-1}i_{2k}}^{\quad j_{2k-1}j_{2k}}-L_{i_{2k-1}}^{j_{2k-1}}L_{i_{2k}}^{j_{2k}}+L_{i_{2k-1}}^{j_{2k}}L_{i_{2k}}^{j_{2k-1}})\,L_{i_{2k+1}}^{j_{2k+1}}\cdots L_{i_{2a-1}}^{j_{2a-1}},
\end{align}
where $1\le \alpha,\beta\le n-1$.

In particular,
\[
(E_{a}^{a-1})_{j}^i
=-\frac{1}{2}B_{a}^{a-1}\delta_j^i+(a-1)(\bar Q_{a,a-2})_{j\,\gamma}^{\alpha \beta}R_{\alpha \beta}^{i\gamma}+(P_a)_{nj}^{n\beta}L_{\beta}^{i},
\]
and for $1\le k\le a-2$,
\[
({E}_{a}^{k})^{i}_j
=-\frac{1}{2}B_{a}^{k}\delta_{j}^{i}+k(\bar Q_{a,k-1})_{j\,\gamma}^{\alpha\beta}R_{\alpha \beta}^{i\gamma}+(2a-2k-1)(\bar Q_{a,k})_{j\gamma}^{\alpha\beta}L_{\alpha}^{i}L_{\beta}^{\gamma}.
\]

We now consider the functional
\[
{\bf F}_{a}[g]=\int_{M}L_{a}\,dv_{g}+2^{a}a!\sum_{k=0}^{a-1}C_{3}(a,k)\int_{\partial M} B_{a}^{k}\,d\sigma_{g},
\]
where
\[
C_{3}(a,k)=\frac{1}{k!(2a-2k-1)!!\,2^{k}}.
\]

\begin{thm}\label{Main functional}
Along a variation $h_{ij}=\frac{\partial g_{ij}}{\partial t}$, we have
\begin{align*}
 \frac{d}{dt}{\bf F}_{a}[g]= -\int_{M}E_{rs}^{(a)}h^{rs}dv_{g}-2^{a}a!\sum_{k=0}^{a-1}C_{3}(a,k)\int _{\partial M}({E}_{a}^{k})_{\alpha}^{\beta}h^{\alpha}_\beta \,d\sigma_{g},
\end{align*}
where
\[
E_{rs}^{(a)}=-\frac{1}{2^{a+1}}\sum_{i,j,l=1}^{n}g_{rl}\delta_{sj_{1}\cdots j_{2a}}^{li_{1}\cdots i_{2a}}R_{i_{1}i_{2}}^{\quad j_{1}j_{2}}\cdots R_{i_{2a-1}i_{2a}}^{\quad j_{2a-1}j_{2a}}.
\]
The critical points of ${\bf F}_{a}[g]$ satisfy
\begin{equation}
    \begin{cases}
E^{(a)}=0\quad &\text{in }M,\\
\sum_{k=0}^{a-1}C_{3}(a,k){E}_{a}^{k}=0\quad &\text{on }\partial M.
\end{cases} 
\end{equation}
\end{thm}

We prove Theorem~\ref{Main functional} using Lemma~\ref{sumation of dBa^k} and Lemma~\ref{eq:La-derivative final} below. To this end, we first record several auxiliary identities.

\begin{lem}\label{lem:Pnanblab}
On $\partial M,$
\iffalse
\begin{equation}
(Q_{a,k-1})_{j_{1}j_{2}}^{i_{1}i_{2}}R_{i_{1}i_{2}}^{\quad j_{1}j_{2}}=B_{a}^{k}\label{eq:basic properties 1}
\end{equation}
\begin{equation}
(Q_{a,k-1})_{j_{1}j_{2}}^{i_{1}i_{2}}L_{i_{1}}^{j_{1}}L_{i_{2}}^{j_{2}}=\frac{1}{2}B_{a}^{k-1}\label{eq:basic properties 2}
\end{equation}
\fi
\begin{equation}
(\bar Q_{a,k-1})_{j_{1}j_{2}}^{i_{1}i_{2}}R_{i_{1}i_{2}}^{\quad j_{1}j_{2}}=B_{a}^{k},\label{eq:basic properties 1}
\end{equation}
\begin{equation}
(\bar Q_{a,k-1})_{j_{1}j_{2}}^{i_{1}i_{2}}L_{i_{1}}^{j_{1}}L_{i_{2}}^{j_{2}}=\frac{1}{2}B_{a}^{k-1},\label{eq:basic properties 2}
\end{equation}
\begin{equation}
(P_a)_{n\alpha}^{n\beta}L_{\beta}^{\alpha}=\frac{1}{2}B_{a}^{a-1}\label{eq:basic properties 2-1},
\end{equation}
and 
\begin{align}
(P_a)_{n\gamma}^{\alpha\beta} & =-(2a-2)(\bar Q_{a,a-2})_{j_{1}\gamma j_{2}}^{\alpha\beta i_{2}}L_{i_{2}}^{j_{2},\,j_{1}}.\label{eq:for boundary third}
\end{align}
\end{lem}

\begin{proof}
The formulas (\ref{eq:basic properties 1})(\ref{eq:basic properties 2})(\ref{eq:basic properties 2-1})
can be obtained from definitions of $B_{a}^{k}$ and $\bar Q_{a,k-1}.$
Since $R_{ni_2}^{\quad j_1j_2}=L^{j_2,j_1}_{i_2}-L^{j_1, j_2}_{i_2}$ for $1\le i_1,j_1, j_2\le n-1$, we get
\begin{align}\label{eq:Pn=00005Cgamma=00005Calpha=00005Cbeta}
(P_a)_{n\gamma}^{\alpha\beta} & =\frac{2a-2}{2^{a}}\delta_{j_{1}j_{2}\cdots j_{2a-3}j_{2a-2}n\gamma}^{ni_{2}\cdots i_{2a-3}i_{2a-2}\alpha\beta}R_{ni_{2}}{}^{j_{1}j_{2}}\cdots R_{i_{2a-3}i_{2a-2}}{}^{j_{2a-3}j_{2a-2}}\nonumber \\
 & =\frac{2a-2}{2^{a}}\delta_{j_{1}j_{2}\cdots j_{2a-3}j_{2a-2}n\gamma}^{ni_{2}\cdots i_{2a-3}i_{2a-2}\alpha\beta}(L_{i_{2}}^{j_{2},\,j_{1}}-L_{i_{2}}^{j_{1},j_{2}})R_{i_{3}i_{4}}^{\quad j_{3}j_{4}}\cdots R_{i_{2a-3}i_{2a-2}}{}^{j_{2a-3}j_{2a-2}}\nonumber \\
% & =\frac{2a-2}{2^{a-1}}\delta_{j_{1}j_{2}\cdots j_{2a-3}j_{2a-2}n\gamma}^{ni_{2}\cdots i_{2a-3}i_{2a-2}\alpha\beta}L_{i_{2}}^{j_{2},\,j_{1}}R_{i_{3}i_{4}}^{\quad j_{3}j_{4}}\cdots R_{i_{2a-3}i_{2a-2}}{}^{j_{2a-3}j_{2a-2}}\nonumber \\
 & =-\frac{2a-2}{2^{a-1}}\delta_{nj_{2}\cdots j_{2a-3}j_{2a-2}j_{1}\gamma}^{ni_{2}\cdots i_{2a-3}i_{2a-2}\alpha\beta}L_{i_{2}}^{j_{2},\,j_{1}}R_{i_{3}i_{4}}^{\quad j_{3}j_{4}}\cdots R_{i_{2a-3}i_{2a-2}}{}^{j_{2a-3}j_{2a-2}}. 
% & =-\frac{2a-2}{2^{a-1}}\sum_{i,j=1}^{n-1}\delta_{j_{2}\cdots j_{2a-3}j_{2a-2}j_{1}\gamma}^{i_{2}\cdots i_{2a-3}i_{2a-2}\alpha\beta}L_{i_{2}}^{j_{2},\,j_{1}}R_{i_{3}i_{4}}^{\quad j_{3}j_{4}}\cdots R_{i_{2a-3}i_{2a-2}}{}^{j_{2a-3}j_{2a-2}}.\no
\end{align}
We conclude this lemma by the definition in \eqref{eq:Pa-k-1-3}.

\end{proof}

We next record an induction identity for the tensors $\bar Q_{a,k}$.
\begin{lem}
On $\partial M$,
\begin{equation}\label{eq:induction formula}
\bar{\nabla}_{i_{1}}(\bar Q_{a,k-1})_{j_{1}j_{2}}^{i_{1}i_{2}}=(2a-2k-1)(\bar Q_{a,k-1})_{j_{1}j_{2}j_{3}}^{i_{1}i_{2}i_{3}}\,\bar \nabla_{i_{1}}L_{i_{3}}^{j_{3}}-2(k-1)(\bar Q_{a,k-2})_{j_{1}j_{2}j_{3}}^{i_{1}i_{2}i_{3}}\,\bar \nabla_{i_{1}}L_{i_{3}}^{j_{3}}.
\end{equation}
\end{lem}

\begin{proof}
Applying the second Bianchi identity for $\bar R$ and using the symmetries of the generalized Kronecker delta, we obtain
\begin{equation}\label{eq:bianchi}
\sum_{i,j=1}^{n-1}\delta_{j_{1}\cdots j_{2a-1}}^{i_{1}\cdots i_{2a-1}}\bar \nabla_{i_{1}}(\bar R_{i_{3}i_{4}}^{\quad j_{3}j_{4}})\cdots R_{i_{2k-1}i_{2k}}^{\quad j_{2k-1}j_{2k}}L_{i_{2k+1}}^{j_{2k+1}}\cdots L_{i_{2a-1}}^{j_{2a-1}}=0.
\end{equation}
Consequently,
\begin{align}
 & \bar \nabla_{i_{1}}(\bar Q_{a,k-1})_{j_{1}j_{2}}^{i_{1}i_{2}}\nonumber \\
= & -\frac{4(k-1)}{2^{k}}\sum_{i,j=1}^{n-1}\delta_{j_{1}\cdots j_{2a-1}}^{i_{1}\cdots i_{2a-1}}\bar {\nabla}_{i_{1}}L_{i_{3}}^{j_{3}}L_{i_{4}}^{j_{4}}R_{i_{5}i_{6}}^{\quad j_{5}j_{6}}\cdots R_{i_{2k-1}i_{2k}}^{\quad\quad j_{2k-1}j_{2k}}L_{i_{2k+1}}^{j_{2k+1}}\cdots L_{i_{2a-1}}^{j_{2a-1}}\nonumber \\
 & +\frac{2a-2k-1}{2^{k}}\sum_{i,j=1}^{n-1}\delta_{j_{1}\cdots j_{2a-1}}^{i_{1}\cdots i_{2a-1}}R_{i_{3}i_{4}}^{\quad j_{3}j_{4}}\cdots R_{i_{2k-1}i_{2k}}^{\quad\quad j_{2k-1}j_{2k}}\bar {\nabla}_{i_{1}}L_{i_{2k+1}}^{j_{2k+1}}\cdots L_{i_{2a-1}}^{j_{2a-1}}\no\\
 =&(2a-2k-1)(\bar Q_{a,k-1})_{j_{1}j_{2}j_{3}}^{i_{1}i_{2}i_{3}}\nabla_{i_{1}}L_{i_{3}}^{j_{3}}-2(k-1)(\bar Q_{a,k-2})_{j_{1}j_{2}j_{3}}^{i_{1}i_{2}i_{3}}\nabla_{i_{1}}L_{i_{3}}^{j_{3}}.\no
\end{align}
This proves \eqref{eq:induction formula}.
\end{proof}

The following elementary algebraic identity will be used to simplify the resulting summations.
For any $A_{1},\dots,A_{a-1}$, it holds that
\begin{align}\label{lem:algerbic oberservation}
 & \sum_{k=0}^{a-1}\frac{2}{k!(2a-2k-3)!!2^{k}}A_{k}-\sum_{k=0}^{a-1}\frac{2}{(k-1)!(2a-2k-1)!!2^{k-1}}A_{k-1}\no\\
= & \frac{2}{(a-1)!2^{a-1}}A_{a-1}.
\end{align}

\begin{lem}\label{induction}
On $\partial M$,
\begin{align}
 & \sum_{k=0}^{a-1}C_3(n,k)\left(-4k(\bar{Q}_{a,k-1})_{j_{1}j_{2}}^{i_{1}i_{2}}L_{i_{2}}^{j_{2}}\frac{d}{dt}L_{i_{1}}^{j_{1}}+2(2a-2k-1)(\bar{Q}_{a,k})_{j_{1}j_{2}}^{i_{1}i_{2}}L_{i_{2}}^{j_{2}}\frac{d}{dt}L_{i_{1}}^{j_{1}}\right)\no\\
= & \frac{2}{(a-1)!2^{a-1}}(P_a)_{n\beta}^{n\alpha}\frac{d}{dt}L_{\alpha}^{\beta},\nonumber 
\end{align}
and 
\begin{align}
&\sum_{k=0}^{a-1}C_3(n,k)2k\left((2a-2k-1)(\bar{Q}_{a,k-1})_{j_{1}j_{2}j_{3}}^{i_{1}i_{2}i_{3}}\nabla_{i_{1}}L_{i_{3}}^{j_{3}}-2(k-1)(\bar{Q}_{a,k-2})_{j_{1}j_{2}j_{3}}^{i_{1}i_{2}i_{3}}\nabla_{i_{1}}L_{i_{3}}^{j_{3}}\right)\no\\
=&\frac{1}{(a-2)!2^{a-2}}(\bar Q_{a,a-2})_{j_{1}j_{2}j_{3}}^{i_{1}i_{2}i_{3}}\nabla_{i_{1}}L_{i_{3}}^{j_{3}},\no
\end{align}
where by abuse of notations, we denote
\begin{equation}
(P_a)_{n\beta}^{n\alpha}\frac{d}{dt}L_{\alpha}^{\beta}=:(Q_{a,a-1})_{j_{1}j_{2}}^{i_{1}i_{2}}L_{i_{2}}^{j_{2}}\frac{d}{dt}L_{i_{1}}^{j_{1}}.\label{eq:some formula-1}
\end{equation}
\end{lem}

\begin{proof}
Recall \[
C_{3}(a,k)=\frac{1}{k!(2a-2k-1)!!2^{k}}.
\]
We have
\begin{align*}
 & \sum_{k=0}^{a-1}\frac{2k}{k!(2a-2k-1)!!2^{k}}\left((2a-2k-1)(\bar{Q}_{a,k-1})_{j_{1}j_{2}j_{3}}^{i_{1}i_{2}i_{3}}\nabla_{i_{1}}L_{i_{3}}^{j_{3}}-2(k-1)(\bar{Q}_{a,k-2})_{j_{1}j_{2}j_{3}}^{i_{1}i_{2}i_{3}}\nabla_{i_{1}}L_{i_{3}}^{j_{3}}\right)\\
= & \sum_{k=1}^{a-1}\frac{2}{(k-1)!(2a-2k-3)!!2^{k}}(\bar Q_{a,k-1})_{j_{1}j_{2}j_{3}}^{i_{1}i_{2}i_{3}}\nabla_{i_{1}}L_{i_{3}}^{j_{3}}\\
 & -\sum_{k=2}^{a-1}\frac{2}{(k-2)!(2a-2k-1)!!2^{k-1}}(\bar Q_{a,k-2})_{j_{1}j_{2}j_{3}}^{i_{1}i_{2}i_{3}}\nabla_{i_{1}}L_{i_{3}}^{j_{3}}\\
= & \frac{1}{(a-2)!2^{a-2}}(\bar Q_{a,a-2})_{j_{1}j_{2}j_{3}}^{i_{1}i_{2}i_{3}}\nabla_{i_{1}}L_{i_{3}}^{j_{3}}.
\end{align*}
\end{proof}
We now compute the first variation of $L_a$ in $M$ and identify the corresponding boundary terms.

\begin{lem}\label{lem:La-t}
In $M$, we have
\[
\frac{\partial}{\partial t}L_{a}=-a\sum_{i,l,r,s=1}^n P_a^{i_{1}i_{2}lr}h_{rs}R_{i_{1}i_{2}l}^{\quad\,\,\,\, s}+2a\sum_{i,q,p=1}^n\nabla_{i_{1}}\bigl(P_a^{i_{1}i_{2}qp}\nabla_{p}h_{i_{2}q}\bigr).
\]
\end{lem}

\begin{proof}
Differentiating the definition of $L_a$ and using the curvature variation formula yields
\begin{align*}
\frac{\partial}{\partial t}L_{a}= & \frac{a}{2^{a}}\sum_{i,j=1}^n\delta_{j_{1}j_{2}\cdots j_{2a-1}j_{2a}}^{i_{1}i_{2}\cdots i_{2a-1}i_{2a}}\frac{\partial}{\partial t}R_{i_{1}i_{2}}{}^{j_{1}j_{2}}\cdots R_{i_{2a-1}i_{2a}}{}^{j_{2a-1}j_{2a}}\\
= & \frac{a}{2^{a}}\delta_{j_{1}j_{2}\cdots j_{2a-1}j_{2a}}^{i_{1}i_{2}\cdots i_{2a-1}i_{2a}}(-g^{j_{2}r}h_{rs}g^{sp}R_{i_{1}i_{2}\,\,\,p}^{\quad\,\, j_{1}})R_{i_{3}i_{4}}^{\quad\,j_{3}j_{4}}\cdots R_{i_{2a-1}i_{2a}}{}^{j_{2a-1}j_{2a}}\\
 & +\frac{a}{2^{a}}\delta_{j_{1}j_{2}\cdots j_{2a-1}j_{2a}}^{i_{1}i_{2}\cdots i_{2a-1}i_{2a}}(\frac{1}{2}g^{j_{2}p}g^{j_{1}q})(\nabla_{i_{1}}\nabla_{i_{2}}h_{pq}-\nabla_{i_{2}}\nabla_{i_{1}}h_{pq}+\nabla_{i_{1}}\nabla_{p}h_{i_{2}q}\\
 & -\nabla_{i_{2}}\nabla_{p}h_{i_{1}q}-\nabla_{i_{1}}\nabla_{q}h_{i_{2}p}+\nabla_{i_{2}}\nabla_{q}h_{i_{1}p})\cdot R_{i_{3}i_{4}}^{\quad\,j_{3}j_{4}}\cdots R_{i_{2a-1}i_{2a}}{}^{j_{2a-1}j_{2a}}\\
= & -aP_a^{i_{1}i_{2}lr}h_{rs}g^{sp}R_{i_{1}i_{2}lp}+\frac{a}{2}P_a^{i_{1}i_{2}qp}(\nabla_{i_{1}}\nabla_{i_{2}}h_{pq}-\nabla_{i_{2}}\nabla_{i_{1}}h_{pq})\\
 & +\frac{a}{2}P_a^{i_{1}i_{2}qp}(\nabla_{i_{1}}\nabla_{p}h_{i_{2}q}-\nabla_{i_{2}}\nabla_{p}h_{i_{1}q})+\frac{a}{2}P_a^{i_{1}i_{2}qp}(-\nabla_{i_{1}}\nabla_{q}h_{i_{2}p}+\nabla_{i_{2}}\nabla_{q}h_{i_{1}p})\\
= & -aP_a^{i_{1}i_{2}lr}h_{rs}g^{sp}R_{i_{1}i_{2}lp}+aP_a^{i_{1}i_{2}qp}\nabla_{i_{1}}\nabla_{p}h_{i_{2}q}-aP_a^{i_{1}i_{2}qp}\nabla_{i_{1}}\nabla_{q}h_{i_{2}p}\\
= & -aP_a^{i_{1}i_{2}lr}h_{rs}g^{sp}R_{i_{1}i_{2}lp}+2a\nabla_{i_{1}}\bigl(P_a^{i_{1}i_{2}qp}\nabla_{p}h_{i_{2}q}\bigr),
\end{align*}
where, by the symmetry/antisymmetry in $p,q$, we have
\[
P_a^{i_{1}i_{2}qp}(\nabla_{i_{1}}\nabla_{i_{2}}h_{pq}-\nabla_{i_{2}}\nabla_{i_{1}}h_{pq})=0.
\]

We next analyze the boundary contribution.
\end{proof}
\begin{lem}
\label{thm: boundary terms detail}
On $\partial M$, we have
\begin{align*}
 & \sum_{i_{2},p,q=1}^{n}P_a^{ni_{2}qp}\nabla_{p}h_{i_{2}q}\\
= & 2(P_a)_{n\gamma }^{n\beta}\frac{\partial}{\partial t}L_{\beta}^{\gamma}+(P_a)_{n\gamma}^{n\beta}h_{r}^{\gamma}L_{\beta}^{r}-\bar{\nabla}_{\alpha}(h_{\beta n}P_a^{n\beta n\alpha})+P_a^{n\gamma\alpha\beta}\bar{\nabla}_{\beta}h_{\gamma\alpha}.
\end{align*}
\end{lem}

\begin{proof}
Splitting the indices into tangential and normal components, we compute
\begin{align*}
 & P_a^{ni_{2}qp}\nabla_{p}h_{i_{2}q}\\
= & P_a^{ni_{2}np}\nabla_{p}h_{i_{2}n}+P_a^{ni_{2}\alpha p}\nabla_{p}h_{i_{2}\alpha}\\
= & P_a^{ni_{2}n\alpha}\nabla_{\alpha}h_{i_{2}n}+P_a^{ni_{2}\alpha\beta}\nabla_{\beta}h_{i_{2}\alpha}+P_a^{ni_{2}\alpha n}\nabla_{n}h_{i_{2}\alpha}\\
= & P_a^{n \beta n\alpha}(\nabla_{\alpha}h_{\beta  n}+\nabla_{\beta}h_{\alpha n}-\nabla_{n}h_{\alpha \beta })-\nabla_{\alpha}(h_{i_{2}n}P_a^{ni_{2}n\alpha})+P_a^{ni_{2}\alpha\beta}\nabla_{\beta}h_{i_{2}\alpha}.
\end{align*}
Using Lemma~\ref{derivative of L}, we obtain
\begin{align}
 & P_a^{ni_{2}qp}\nabla_{p}h_{i_{2}q}\nonumber \\
 =& 2P_a^{n\beta n\alpha}(-\tfrac{1}{2}h_{nn}L_{\alpha\beta}+\tfrac{\partial}{\partial t}L_{\alpha\beta})-\nabla_{\alpha}(h_{i_{2}n}P_a^{ni_{2}n\alpha})+P_a^{ni_{2}\alpha\beta}\nabla_{\beta}h_{i_{2}\alpha}\nonumber \\
 =& 2P_a^{n\beta n\alpha}\frac{\partial}{\partial t}L_{\alpha\beta}-h_{nn}L_{\alpha\beta}P_a^{n\beta n\alpha}-\nabla_{\alpha}(h_{\beta n}P_a^{n\beta n\alpha})+P_a^{n\gamma\alpha\beta}\nabla_{\beta}h_{\gamma\alpha}.\label{eq:boundary terms}
\end{align}
Next, we convert covariant derivatives to tangential ones. From
\[
\nabla_{\alpha}h_{\beta n}=\bar{\nabla}_{\alpha}h_{\beta n}-L_{\alpha\beta}h_{nn}+L_{\alpha}^{\gamma}h_{\beta\gamma}
\qquad\text{and}\qquad
\nabla_{\alpha}(P_a)^{n\beta n\alpha}=\bar {\nabla}_{\alpha}(P_a)^{n\beta n\alpha}-L_{\alpha}^{\gamma}(P_a)_{n\gamma}^{\beta\alpha},
\]
we obtain
\begin{align}\label{third derivative 1}
\nabla_{\alpha}(h_{\beta n}P_a^{n\beta n\alpha})
= & \,\bar{\nabla}_{\alpha}(h_{\beta n}P_a^{n\beta n\alpha})-L_{\alpha\beta}h_{nn}P_a^{n\beta n\alpha}+L_{\alpha}^{\gamma}h_{\beta\gamma}P_a^{n\beta n\alpha}-h_{\beta n}L_{\alpha}^{\gamma}(P_a)_{n\gamma}^{\beta\alpha}.
\end{align}
Moreover,
\begin{align}
  P_a^{n\beta n\alpha}\frac{\partial}{\partial t}L_{\alpha\beta}\nonumber 
=  (P_a)_{n\gamma}^{n\beta}\frac{\partial}{\partial t}L_{\beta}^{\gamma}+(P_a)_{n\gamma }^{n\beta}h_{r}^{\gamma}L_{\beta}^{r}.\label{eq:boundary normal related}
\end{align}
Finally, since $\nabla_{\beta}h_{\gamma\alpha}=\bar{\nabla}_{\beta} h_{\gamma \alpha}-L_{\beta\gamma}h_{n\alpha}-L_{\beta \alpha}h_{\gamma n}$ and $P_a^{n\gamma\alpha\beta}$ is antisymmetric in $\alpha,\beta$, we have
\[
P_a^{n\gamma\alpha\beta}\nabla_{\beta}h_{\gamma\alpha}=P_a^{n\gamma\alpha\beta}\bar {\nabla}_{\beta} h_{\gamma \alpha}-P_a^{n\gamma\alpha\beta}L_{\beta\gamma}h_{n\alpha}.
\]
Substituting these identities into \eqref{eq:boundary terms} gives the desired formula.
\end{proof}

\begin{lem}\label{eq:La-derivative final}
We have
\begin{align}
\frac{d}{dt}\int_{M}L_{a}dv_{g}
= & -\int_{M}E_{rs}^{(a)}h^{rs}dv_{g}
  -4a\int_{\partial M}(P_a)_{n\gamma}^{n\beta}\frac{\partial}{\partial t}L_{\beta}^{\gamma}d\sigma_{g}\no\\
  &-2a\int_{\partial M}(P_a)_{n\gamma}^{n\beta}h_{r}^{\gamma}L_{\beta}^{r}
  -2a\int_{\partial M}P_a^{n\gamma\alpha\beta}\bar\nabla_{\beta}h_{\gamma\alpha}d\sigma_{g}.\nonumber 
\end{align}
\end{lem}

\begin{proof}
By Lemma~\ref{lem:La-t},
\begin{align*}
\frac{d}{dt}\int_{M}L_{a}dv_{g}
= -\int_{M}E_{rs}^{(a)}h^{rs}dv_{g}-2a\int_{\partial M}P_a^{ni_{2}qp}\nabla_{p}h_{i_{2}q}\,d\sigma_{g}.
\end{align*}
The conclusion follows from Lemma~\ref{thm: boundary terms detail}.
\end{proof}

We now turn to terms involving tangential boundary derivatives.

\begin{lem}\label{derivative bouddary rm cur}
On $\partial M$, we have
\begin{align*}
 & \frac{1}{2^{k}}\sum_{i,j=1}^{n-1}\delta_{j_{1}j_{2}\cdots j_{2a-3}j_{2a-2}j_{2a-1}}^{i_{1}i_{2}\cdots i_{2a-3}i_{2a-2}i_{2a-1}}\frac{d}{dt}\overline{R}_{i_{1}i_{2}}^{\quad j_{1}j_{2}}R_{i_{3}i_{4}}^{\quad j_{3}j_{4}}\cdots R_{i_{2k-1}i_{2k}}^{\quad j_{2k-1}j_{2k}}L_{i_{2k+1}}^{j_{2k+1}}\cdots L_{i_{2a-1}}^{j_{2a-1}}\\
= & -(\bar Q_{a,k-1})_{j_{1}j_{2}}^{i_{1}i_{2}}R_{i_{1}i_{2}\quad}^{\quad\,\,j_{1}\beta}h_{\beta}^{j_{2}}-2(\bar Q_{a,k-1})_{j_{1}j_{2}}^{i_{1}i_{2}}L_{i_{1}}^{j_{1}}L_{i_{2}}^{\beta}h_{\beta}^{j_{2}}\\
 & +2\bar{\nabla}_{i_{1}}\left((\bar{Q}_{a,k-1})_{j_{1}j_{2}}^{i_{1}i_{2}}g^{j_{2}p}g^{j_{1}q}\bar{\nabla}_{p}h_{i_{2}q}\right)\\
 & -2\left((2a-2k-1)(\bar{Q}_{a,k-1})_{j_{1}j_{2}j_{3}}^{i_{1}i_{2}i_{3}}\nabla_{i_{1}}L_{i_{3}}^{j_{3}}-2(k-1)(\bar{Q}_{a,k-2})_{j_{1}j_{2}j_{3}}^{i_{1}i_{2}i_{3}}\nabla_{i_{1}}L_{i_{3}}^{j_{3}}\right)g^{j_{2}p}g^{j_{1}q}\bar{\nabla}_{p}h_{i_{2}q}.
\end{align*}
\end{lem}

\begin{proof}
Using the curvature variation formula on $(\partial M,\bar g)$ and the Gauss equation, we have
\begin{align*}
\frac{d}{dt}\overline{R}_{i_{1}i_{2}}^{\quad j_{1}j_{2}}= & -\bar{g}^{j_{2}\alpha}h_{\alpha\beta}\overline{g}^{\beta p}\overline{R}_{i_{1}i_{2}\quad p}^{\quad\,\,j_{1}}\\
 & +\frac{1}{2}\overline{g}^{j_{2}p}\overline{g}^{j_{1}q}(\bar{\nabla}_{i_{1}}\bar{\nabla}_{i_{2}}h_{pq}-\bar{\nabla}_{i_{2}}\bar{\nabla}_{i_{1}}h_{pq}+\bar{\nabla}_{i_{1}}\bar{\nabla}_{p}h_{i_{2}q}\\
 & -\bar{\nabla}_{i_{2}}\bar{\nabla}_{p}h_{i_{1}q}-\bar{\nabla}_{i_{1}}\bar{\nabla}_{q}h_{i_{2}p}+\bar{\nabla}_{i_{2}}\bar{\nabla}_{q}h_{i_{1}p})\\
= & -\bar{g}^{j_{2}\alpha}h_{\alpha\beta}\overline{g}^{\beta p}(R_{i_{1}i_{2}\quad p}^{\quad\,\,j_{1}}+L_{i_{1}}^{j_{1}}L_{i_{2}p}-L_{i_{1}p}L_{i_{2}}^{j_{1}})\\
 & +\frac{1}{2}\overline{g}^{j_{2}p}\overline{g}^{j_{1}q}(\bar{\nabla}_{i_{1}}\bar{\nabla}_{i_{2}}h_{pq}-\bar{\nabla}_{i_{2}}\bar{\nabla}_{i_{1}}h_{pq}+\bar{\nabla}_{i_{1}}\bar{\nabla}_{p}h_{i_{2}q}\\
 & -\bar{\nabla}_{i_{2}}\bar{\nabla}_{p}h_{i_{1}q}-\bar{\nabla}_{i_{1}}\bar{\nabla}_{q}h_{i_{2}p}+\bar{\nabla}_{i_{2}}\bar{\nabla}_{q}h_{i_{1}p}).
\end{align*}
Substituting this into the definition of $\bar Q_{a,k-1}$ and integrating by parts gives the claimed identity; the final term follows from \eqref{eq:induction formula}.
\end{proof}

Furthermore, we have the following boundary computation.

\begin{lem}\label{derivative ofB}
On $\partial  M$,
\begin{align}
       \frac{d}{dt}B_{a}^{k}=  & -k(\bar Q_{a,k-1})_{j_{1}j_{2}}^{i_{1}i_{2}}R_{i_{1}i_{2}\quad}^{\quad\,\,j_{1}\beta}h_{\beta}^{j_{2}}-2k(\bar Q_{a,k-1})_{j_{1}j_{2}}^{i_{1}i_{2}}L_{i_{1}}^{j_{1}}L_{i_{2}}^{\beta}h_{\beta}^{j_{2}}\no\\
 & +2k\bar{\nabla}_{i_{1}}\left((\bar{Q}_{a,k-1})_{j_{1}j_{2}}^{i_{1}i_{2}}g^{j_{2}p}g^{j_{1}q}\bar{\nabla}_{p}h_{i_{2}q}\right)\no\\
 & -2k\left((2a-2k-1)(\bar{Q}_{a,k-1})_{j_{1}j_{2}j_{3}}^{i_{1}i_{2}i_{3}}\nabla_{i_{1}}L_{i_{3}}^{j_{3}}-2(k-1)(\bar{Q}_{a,k-2})_{j_{1}j_{2}j_{3}}^{i_{1}i_{2}i_{3}}\nabla_{i_{1}}L_{i_{3}}^{j_{3}}\right)\bar{\nabla}^{j_2}h_{i_{2}}^{j_1} \no\\
 &-4k(\bar{Q}_{a,k-1})_{j_{1}j_{2}}^{i_{1}i_{2}}L_{i_{2}}^{j_{2}}\frac{d}{dt}L_{i_{1}}^{j_{1}}+2(2a-2k-1)(\bar{Q}_{a,k})_{j_{1}j_{2}}^{i_{1}i_{2}}L_{i_{2}}^{j_{2}}\frac{d}{dt}L_{i_{1}}^{j_{1}}.\no
   \end{align}
\end{lem}
\begin{proof}
By Lemma~\ref{derivative bouddary rm cur},
\begin{align*}
\frac{d}{dt}B_{a}^{k}
=& \frac{k}{2^{k}}\sum_{i,j=1}^{n-1}\delta_{j_{1}j_{2}\cdots j_{2a-3}j_{2a-2}j_{2a-1}}^{i_{1}i_{2}\cdots i_{2a-3}i_{2a-2}i_{2a-1}}\frac{d}{dt}(\overline{R}_{i_{1}i_{2}}^{\quad j_{1}j_{2}}-L_{i_{1}}^{j_{1}}L_{i_{2}}^{j_{2}}+L_{i_{1}}^{j_{2}}L_{i_{2}}^{j_{1}})\\
 & \cdots(\overline{R}_{i_{2k-1}i_{2k}}^{\quad j_{2k-1}j_{2k}}-L_{i_{2k-1}}^{j_{2k-1}}L_{i_{2k}}^{j_{2k}}+L_{i_{2k-1}}^{j_{2k}}L_{i_{2k}}^{j_{2k-1}})L_{i_{2k+1}}^{j_{2k+1}}\cdots L_{i_{2a-1}}^{j_{2a-1}}\\
 & +\frac{2a-2k-1}{2^{k}}\sum_{i,j=1}^{n-1}\delta_{j_{1}j_{2}\cdots j_{2a-3}j_{2a-2}j_{2a-1}}^{i_{1}i_{2}\cdots i_{2a-3}i_{2a-2}i_{2a-1}}(\overline{R}_{i_{1}i_{2}}^{\quad j_{1}j_{2}}-L_{i_{1}}^{j_{1}}L_{i_{2}}^{j_{2}}+L_{i_{1}}^{j_{2}}L_{i_{2}}^{j_{1}})\\
 & \cdots(\overline{R}_{i_{2k-1}i_{2k}}^{\quad j_{2k-1}j_{2k}}-L_{i_{2k-1}}^{j_{2k-1}}L_{i_{2k}}^{j_{2k}}+L_{i_{2k-1}}^{j_{2k}}L_{i_{2k}}^{j_{2k-1}})\frac{d}{dt}L_{i_{2k+1}}^{j_{2k+1}}\cdots L_{i_{2a-1}}^{j_{2a-1}}\\
= & \frac{k}{2^{k}}\sum_{i,j=1}^{n-1}\delta_{j_{1}j_{2}\cdots j_{2a-3}j_{2a-2}j_{2a-1}}^{i_{1}i_{2}\cdots i_{2a-3}i_{2a-2}i_{2a-1}}\frac{d}{dt}\overline{R}_{i_{1}i_{2}}^{\quad j_{1}j_{2}}(\overline{R}_{i_{3}i_{4}}^{\quad j_{3}j_{4}}-L_{i_{3}}^{j_{3}}L_{i_{4}}^{j_{4}}+L_{i_{3}}^{j_{4}}L_{i_{4}}^{j_{3}})\\
 & \cdots(\overline{R}_{i_{2k-1}i_{2k}}^{\quad j_{2k-1}j_{2k}}-L_{i_{2k-1}}^{j_{2k-1}}L_{i_{2k}}^{j_{2k}}+L_{i_{2k-1}}^{j_{2k}}L_{i_{2k}}^{j_{2k-1}})L_{i_{2k+1}}^{j_{2k+1}}\cdots L_{i_{2a-1}}^{j_{2a-1}}\\
 & -4k(\bar{Q}_{a,k-1})_{j_{1}j_{2}}^{i_{1}i_{2}}L_{i_{2}}^{j_{2}}\frac{d}{dt}L_{i_{1}}^{j_{1}}+2(2a-2k-1)(\bar{Q}_{a,k})_{j_{1}j_{2}}^{i_{1}i_{2}}L_{i_{2}}^{j_{2}}\frac{d}{dt}L_{i_{1}}^{j_{1}}\\
=  & -k(Q_{a,k-1})_{j_{1}j_{2}}^{i_{1}i_{2}}R_{i_{1}i_{2}\quad}^{\quad\,\,j_{1}\beta}h_{\beta}^{j_{2}}-2k(Q_{a,k-1})_{j_{1}j_{2}}^{i_{1}i_{2}}L_{i_{1}}^{j_{1}}L_{i_{2}}^{\beta}h_{\beta}^{j_{2}}\\
 & +2k\bar{\nabla}_{i_{1}}\left((\bar{Q}_{a,k-1})_{j_{1}j_{2}}^{i_{1}i_{2}}g^{j_{2}p}g^{j_{1}q}\bar{\nabla}_{p}h_{i_{2}q}\right)\\
 & -2k\left((2a-2k-1)(\bar{Q}_{a,k-1})_{j_{1}j_{2}j_{3}}^{i_{1}i_{2}i_{3}}\nabla_{i_{1}}L_{i_{3}}^{j_{3}}-2(k-1)(\bar{Q}_{a,k-2})_{j_{1}j_{2}j_{3}}^{i_{1}i_{2}i_{3}}\nabla_{i_{1}}L_{i_{3}}^{j_{3}}\right)g^{j_{2}p}g^{j_{1}q}\bar{\nabla}_{p}h_{i_{2}q}\\
 & -4k(\bar{Q}_{a,k-1})_{j_{1}j_{2}}^{i_{1}i_{2}}L_{i_{2}}^{j_{2}}\frac{d}{dt}L_{i_{1}}^{j_{1}}+2(2a-2k-1)(\bar{Q}_{a,k})_{j_{1}j_{2}}^{i_{1}i_{2}}L_{i_{2}}^{j_{2}}\frac{d}{dt}L_{i_{1}}^{j_{1}}.
\end{align*}
\end{proof}

We next sum the boundary variations with the coefficients $C_3(a,k)$.

\begin{lem}\label{sumation of dBa^k}
\begin{align}
&\sum_{k=0}^{a-1}C_{3}(a,k)\frac{d}{dt}\int_{\partial M} B_{a}^{k}\,d\sigma_{g}
\no\\
=&\sum_{k=0}^{a-1}C_{3}(a,k)\frac{1}{2}\int_{\partial M} B_{a}^{k}g^{\alpha\beta}h_{\alpha\beta}\,d\sigma_{g}
+\frac{2}{(a-1)!2^{a-1}}\int_{\partial M}(P_a)_{n\beta}^{n\alpha}\frac{d}{dt}L_{\alpha}^{\beta}\,d\sigma_{g}
\no\\
&\quad+\frac{1}{(a-1)!2^{a-1}}\int_{\partial M}P_a^{n\gamma\alpha\beta}\bar \nabla_{\beta}h_{\gamma\alpha}\,d\sigma_{g}
\no\\
&\quad+\sum_{k=0}^{a-1}C_{3}(a,k)\Bigl(-k\int_{\partial M}\sum_{s=1}^{n-1}(\bar Q_{a,k-1})_{j_{1}j_{2}}^{i_{1}i_{2}}R_{i_{1}i_{2}\quad}^{\quad\,j_{1}s}h_{s}^{j_{2}}\,d\sigma_{g}
-2k\int_{\partial M}(\bar Q_{a,k-1})_{j_{1}j_{2}}^{i_{1}i_{2}}L_{i_{2}}^{j_{2}}h_{s}^{j_{1}}L_{i_{1}}^{s}\,d\sigma_{g}\Bigr).
\no
\end{align}
\end{lem}

\begin{proof}
By Lemma~\ref{derivative ofB} together with Lemma~\ref{induction},
\begin{align*}
 & \sum_{k=0}^{a-1}C_{3}(a,k)\frac{d}{dt}\int_{\partial M} B_{a}^{k}\,d\sigma_{g}-\sum_{k=0}^{a-1}C_{3}(a,k)\frac{1}{2}\int_{\partial M} B_{a}^{k}g^{\alpha\beta}h_{\alpha\beta}\,d\sigma_{g}\\
= & \sum_{k=0}^{a-1}C_{3}(a,k)\Bigl(-k\int_{\partial M}\sum_{s=1}^{n-1}(\bar Q_{a,k-1})_{j_{1}j_{2}}^{i_{1}i_{2}}R_{i_{1}i_{2}\quad}^{\quad\,j_{1}s}h_{s}^{j_{2}}\,d\sigma_{g}
-2k\int_{\partial M}(\bar Q_{a,k-1})_{j_{1}j_{2}}^{i_{1}i_{2}}L_{i_{2}}^{j_{2}}h_{s}^{j_{1}}L_{i_{1}}^{s}\,d\sigma_{g}\Bigr)\\
 & +\frac{2}{(a-1)!2^{a-1}}\int_{\partial M}(P_a)_{n\beta}^{n\alpha}\frac{d}{dt}L_{\alpha}^{\beta}\,d\sigma_{g}-\frac{1}{(a-2)!2^{a-2}}\int_{\partial M}(\bar Q_{a,a-2})_{j_{1}j_{2}j_{3}}^{i_{1}i_{2}i_{3}}\nabla_{i_{1}}L_{i_{3}}^{j_{3}}\bar{\nabla}^{j_2}h_{i_{2}}^{j_1}\,d\sigma_{g}.
\end{align*}
On the other hand, by \eqref{eq:for boundary third},
\begin{align*}
\frac{1}{(a-2)!2^{a-2}}\int_{\partial M}(\bar Q_{a,a-2})_{j_{1}j_{2}j_{3}}^{i_{1}i_{2}i_{3}}\nabla_{i_{1}}L_{i_{3}}^{j_{3}}\bar \nabla^{j_{2}}h_{i_{2}}^{j_{1}}\,d\sigma_{g}
= \frac{1}{(a-1)!2^{a-1}}\int_{\partial M}P_a^{n\gamma\alpha\beta}\bar \nabla_{\beta}h_{\gamma\alpha}\,d\sigma_{g}.
\end{align*}
Combining these identities proves the lemma.
\end{proof}

Combining Lemma~\ref{sumation of dBa^k} and Lemma~\ref{eq:La-derivative final} with the definition of $E_a^k$ completes the proof of Theorem~\ref{Main functional}.

On a locally conformally flat manifold, consider the conformal variation $h=-2\phi g$. A direct computation yields
\begin{equation}\label{eq:conformal-variation-Fa}
\frac{d}{dt}{\bf F}_{a}[g]
=(2a-n)\left(\int_{M}\phi\, L_{a}\,dv_{g}+2^{a}a!\int_{\partial M}\phi\sum_{k=0}^{a-1}C_{3}(a,k)B_{a}^{k}\,d\sigma_{g}\right).
\end{equation}
Up to the constant factor $2a-n$, this coincides with the conformal functional studied by S.~Chen \cite{Chen}.
\iffalse

\fi
\bibliography{bibmass}
\bibliographystyle{amsplain}

\end{document}